\documentclass[a4paper]{amsart}

\usepackage{amsthm, amsmath, amssymb, bbm}
\usepackage[utf8]{inputenc}
\usepackage{bbm}
\usepackage{mathrsfs}
\usepackage[mode=image]{standalone}

\usepackage{tikz}
\usetikzlibrary{decorations.markings}
\usetikzlibrary{calc} 
\usetikzlibrary{through}
\usetikzlibrary{backgrounds}

\usepackage{xy} \xyoption{all}
\newdir{ >}{{}*!/-10pt/@{>}}

\usepackage{color}

\usepackage{marginnote}
\usepackage{ulem}
\newenvironment{rouge}
{\relax\color{red}}
{\hspace*{.3ex}\relax}

\usepackage{enumerate}
\usepackage{hyperref}

\renewcommand{\emph}{\textit}
\theoremstyle{plain}
\newtheorem{theorem}{Theorem}[subsection]
\newtheorem*{theorem*}{Theorem}
\newtheorem{corollary}[theorem]{Corollary}
\newtheorem{proposition}[theorem]{Proposition}
\newtheorem{prop}[theorem]{Proposition}
\newtheorem{lemma}[theorem]{Lemma}
\newtheorem*{lemma*}{Lemma}

\theoremstyle{definition}
\newtheorem{definition}[theorem]{Definition}

\newtheorem*{example*}{Example}
\newtheorem{notation}[theorem]{Notation}
\newtheorem{remark}[theorem]{Remark}
\newtheorem{convention}[theorem]{Convention}

\numberwithin{equation}{section}
\numberwithin{figure}{section}

\addtocontents{toc}{\protect\setcounter{tocdepth}{1}}
\renewcommand{\emptyset}{\varnothing}

\newcommand{\dunion}{\sqcup}
\newcommand{\union}{\cup}
\newcommand{\Union}{\bigcup\limits}

\newcommand{\C}{\mathbb{C}}

\newcommand{\R}{\mathbb{R}}

\newcommand{\Z}{\mathbb{Z}}

\newcommand{\op}{\mathrm{op}}
\DeclareMathOperator{\id}{id}

\newcommand{\derd}{\mathsf{D}}
\newcommand{\dere}{\mathsf{E}}
\newcommand{\derr}{\mathsf{R}}
\newcommand{\derl}{\mathsf{L}}

\newcommand{\BDC}{\derd^{\mathrm{b}}}

\newcommand{\DSum}{\bigoplus}
\newcommand{\dsum}[1][]{\mathbin{\oplus_{#1}}}

\newcommand{\ilim}[1][]{\mathop{\varinjlim}\limits_{#1}}

\DeclareMathOperator{\coker}{coker}
\DeclareMathOperator{\im}{im}

\renewcommand{\to}[1][]{\xrightarrow{~#1~}}
\newcommand{\from}[1][]{\xleftarrow{~#1~}}
\newcommand{\isofrom}[1][]{\xleftarrow[~#1~]%
{\raisebox{-.4ex}[0ex][-.4ex]{$\mspace{2mu}\sim\mspace{2mu}$}}}
\newcommand{\isoto}[1][]{\xrightarrow[~#1~]{%
{\raisebox{-.4ex}[0ex][-.4ex]{$\mspace{1mu}\sim\mspace{2mu}$}}}}

\newcommand{\Endo}[1][]{\mathrm{End}_{\raise1.5ex\hbox to.1em{}#1}}
\newcommand{\Hom}[1][]{\mathrm{Hom}_{\raise1.5ex\hbox to.1em{}#1}}
\newcommand{\RHom}[1][]{\derr\mathrm{Hom}_{\raise1.5ex\hbox to.1em{}#1}}

\newcommand{\Ext}[2][]{\mathrm{Ext}_{\raise1.5ex\hbox to.1em{}#1}^{#2}}
\newcommand{\Mod}{\mathrm{Mod}}
\newcommand{\Tens}[1][]{\mathbin{\otimes_{\raise1.5ex\hbox to-.1em{}#1}}}
\newcommand{\LTens}[1][]{\mathbin{\otimes_{\raise1.5ex\hbox to-.1em{}#1}^{\derl}}}
\newcommand{\Tor}[2][]{\mathrm{Tor}^{\raise1.5ex\hbox to.1em{}#1}_{#2}}

\newcommand{\sheaffont}[1]{\mathcal{#1}}
\def\sha{\sheaffont{A}}

\def\she{\sheaffont{E}}

\def\shi{\sheaffont{I}}

\def\shl{\sheaffont{L}}
\def\shm{\sheaffont{M}}
\def\shn{\sheaffont{N}}

\newcommand{\sect}{\varGamma}
\newcommand{\rsect}{\derr\varGamma}
\newcommand{\shendo}[1][]{{\sheaffont{E}nd}_{\raise1.5ex\hbox to.1em{}#1}}
\renewcommand{\hom}[1][]{{\sheaffont{H}om}_{\raise1.5ex\hbox to.1em{}#1}}
\newcommand{\aut}[1][]{{\sheaffont{A}ut}_{\raise1.5ex\hbox to.1em{}#1}}
\newcommand{\inn}[1][]{{\sheaffont{I}nn}_{\raise1.5ex\hbox to.1em{}#1}}

\newcommand{\rhom}[1][]{{\derr\sheaffont{H}om}_{\raise1.5ex\hbox to.1em{}#1}}
\newcommand{\ext}[2][]{{\sheaffont{E}xt}_{\raise1.5ex\hbox to.1em{}#1}^{#2}}
\newcommand{\thom}[1][]{{\sheaffont{T}hom}_{\raise1.5ex\hbox to.1em{}#1}}
\newcommand{\tens}[1][]{\mathbin{\otimes_{\raise1.5ex\hbox to-.1em{}#1}}}
\newcommand{\ltens}[1][]{\mathbin{\otimes_{\raise1.5ex\hbox to-.1em{}#1}^{\derl}}}
\newcommand{\etens}[1][]{\mathbin{\boxtimes_{\raise1.5ex\hbox to-.1em{}#1}}}
\DeclareMathOperator{\supp}{supp}

\newcommand{\roim}[1]{\derr#1_*}
\newcommand{\roimv}[1]{\derr#1}
\newcommand{\reim}[1]{\derr#1_{\mspace{.5mu}!}\mspace{2mu}}

\newcommand{\reeim}[1]{\derr#1_{\mspace{1mu}!!}\mspace{1mu}}

\newcommand{\reimm}[1]{\derr#1_{\mspace{.5mu}!}}
\newcommand{\opb}[1]{#1^{-1}}

\newcommand{\epb}[1]{#1^{\mspace{1.5mu}!}\mspace{2mu}}

\DeclareMathOperator{\ori}{or}

\newcommand{\D}{\sheaffont{D}}
\renewcommand{\O}{\sheaffont{O}}

\newcommand{\Db}{\sheaffont{D}b}

\newcommand{\detens}[1][]%
{\mathbin{\boxtimes_{\raise1.5ex\hbox to-.1em{}#1}^{\mspace{2mu}\mathsf{D}}}}
\newcommand{\doim}[1]{{\mathsf{D}#1}_*\mspace{1mu}}

\newcommand{\dopb}[1]{{\mathsf{D}#1}^{\mspace{1mu}*}}

\newcommand{\dtens}[1][]{\mathbin{\otimes_{\raise1.5ex\hbox to-.1em{}#1}^{\mathsf{D}}}}

\newcommand{\hol}{\mathrm{hol}}

\newcommand{\reghol}{{\mathrm{reg\text-hol}}}

\newcommand{\field}{\mathbf{k}}
\newcommand{\ind}{\mathrm{I}\mspace{2mu}}
\newcommand{\ifield}{\ind\field}

\newcommand{\Rc}{{\R\text-\mathrm{c}}}
\newcommand{\Cc}{{\C\text-\mathrm{c}}}

\newcommand{\ctens}{\mathbin{\mathop\otimes\limits^+}}

\newcommand{\cihom}{{\shi hom}^+}

\newcommand{\PR}{\mathsf{P}}

\newcommand{\PP}{\mathbb{P}}

\newcommand{\tmp}{\mathsf{t}}

\newcommand{\sol}{\mathcal Sol}

\newcommand{\solt}{\mathcal Sol^{\mspace{2.5mu}\tmp}}

\newcommand{\Ot}{\O^{\mspace{2mu}\tmp}}

\renewcommand{\Re}{\operatorname{Re}}
\renewcommand{\Im}{\operatorname{Im}}

\newcommand{\ihom}[1][]{{\shi hom}_{\raise1.5ex\hbox to.1em{}#1}}
\newcommand{\rihom}[1][]{{\derr\mspace{2mu}\shi hom}_{\raise1.5ex\hbox to.1em{}#1}}
\newcommand{\ii}[1][]{{\sheaffont{I}h}_{\raise1.5ex\hbox to.1em{}#1}}

\newcommand{\indlim}[1][]{\mathop{\text{\rm``$\varinjlim$''}}\limits_{#1}}

\newcommand{\dcomp}[1][]{\mathbin{\circ_{\raise1.5ex\hbox to-.1em{}#1}^{\mathsf{D}}}}

\newcommand{\enh}{\mathsf{E}}

\newcommand{\solE}{\mathcal{S}ol^{\mspace{1mu}\enh}}
\newcommand{\fhom}{\derr\mathcal{H}om^\enh}

\newcommand{\BEC}[2][\ifield]{\dere^{\mathrm{b}}(#1_{#2})}
\newcommand{\BECRc}[2][\ifield]{\dere^{\mathrm{b}}_\Rc(#1_{#2})}
\newcommand{\BECp}[2][\ifield]{\dere^{\mathrm{b}}_+(#1_{#2})}
\newcommand{\BECm}[2][\ifield]{\dere^{\mathrm{b}}_-(#1_{#2})}
\newcommand{\BECpm}[2][\ifield]{\dere^{\mathrm{b}}_\pm(#1_{#2})}

\newcommand{\Eoim}[1]{{\enh#1}_*}

\newcommand{\Eeeim}[1]{{\enh#1}_{!!}}

\newcommand{\Eopb}[1]{{\enh#1}^{-1}}

\newcommand{\Eepb}[1]{{\enh\mspace{1mu}#1}^{\mspace{1.5mu}!}}

\newcommand{\semicolon}{\nobreak \mskip2mu\mathpunct{}\nonscript\mkern-\thinmuskip{;}\mskip6mu plus1mu\relax}

\newcommand{\dual}{\mathrm{D}}

\newcommand{\sep}{\mspace{2mu}}

\newcommand{\defeq}{\mathbin{:=}}
\newcommand{\eq}{\begin{eqnarray}}
\newcommand{\eneq}{\end{eqnarray}}
\newcommand{\eqn}{\begin{eqnarray*}}
\newcommand{\eneqn}{\end{eqnarray*}}

\newcommand{\To}[1][]{\xrightarrow[]{\mspace{10mu}{#1}\mspace{10mu}}}

\newcommand{\ba}{\begin{array}}
\newcommand{\ea}{\end{array}}

\newcommand{\nc}{\newcommand}

\newcommand{\be}{\begin{enumerate}}
\newcommand{\ee}{\end{enumerate}}

\nc{\bwr}{\mbox{\large{$\wr$}}}
\nc{\vphi}{\varphi}

\newcommand{\ghol}{\mathrm{g\text-hol}}
\newcommand{\conv}{\mathbin{\ast}}
\newcommand{\Ecomp}[1][]{\mathbin{\circ_{\raise1.5ex\hbox to-.1em{}#1}^{+}}}
\newcommand{\scomp}[1][]{\mathbin{\circ_{\raise1.5ex\hbox to-.1em{}#1}^{\ast}}}

\newcommand{\bb}{{\PP^*}}
\newcommand{\Tam}[2][\field]{\widetilde\dere^{\mathrm{b}}_\Rc(#1_{#2})}

\renewcommand{\dcomp}{\mathbin{\mathop\circ\limits^{\mathsf{D}}}}
\renewcommand{\Ecomp}{\mathbin{\mathop\circ\limits^+}}
\newcommand{\ccomp}{\mathbin{\mathop\circ\limits^{\conv}}}

\usepackage{dsfont}

\newcommand{\imm}{\sqrt{\text-1}}

\newcommand{\muu}{m}
\renewcommand{\imm}{\sqrt{\text-1}}

\newcommand{\U}{\mathbb{U}}
\newcommand{\V}{\mathbb{V}}
\newcommand{\W}{\V^*}
\newcommand{\HH}{{\mathbb{H}}}
\newcommand{\hh}{{\mathbb{H}^*}}
\newcommand{\Vi}{{\V_\infty}}
\newcommand{\Wi}{{\W_\infty}}
\renewcommand{\leq}{\leqslant}
\renewcommand{\geq}{\geqslant}

\newcommand{\FS}[1]{(#1)^\wedge}
\newcommand{\FSa}[1]{(#1)^{\wedge,a}}
\newcommand{\smsh}[1][\V]{\sigma_{#1}}
\newcommand{\Erest}[2]{\opb\pi\field_{#1}\tens{#2}}

\newcommand{\tconv}{\mathbin{\mathop\otimes\limits^*}}
\newcommand{\tchom}{\derr\mathcal{H}om^*}

\newcommand{\quot}{\mathsf Q}

\newcommand{\fra}{\mathfrak{a}}
\newcommand{\frb}{\mathfrak{b}}

\let\tilde\widetilde
\let\setminus\smallsetminus
\newcommand{\Perv}{\mathsf{Perv}}
\newcommand{\rc}{\mathrm{c}}
\newcommand{\wtL}{\widetilde L}
\newcommand{\PhiT}{\mathbb{T}}
\newcommand{\PsiT}{\operatorname{T}}

\newcommand{\sub}{{\operatorname{sub}}}
\newcommand{\cs}{\mathcal{CS}}
\newcommand{\bm}{\mathcal{BM}}
\newcommand{\BM}{\mathrm{BM}}
\newcommand{\enhK}{\mathcal{K}}
\newcommand{\enhE}{\mathbb{E}}

\DeclareRobustCommand\eqc{c}

\let\oldbigoplus\bigoplus
\renewcommand{\bigoplus}{\mathop{\textstyle\oldbigoplus}\displaylimits}
\let\oldbigcup\bigcup
\renewcommand{\bigcup}{\mathop{\textstyle\oldbigcup}\displaylimits}
\let\oldbigcap\bigcap
\renewcommand{\bigcap}{\mathop{\textstyle\oldbigcap}\displaylimits}
\let\oldbigotimes\bigotimes
\renewcommand{\bigotimes}{\mathop{\textstyle\oldbigotimes}\displaylimits}
\let\oldbigwedge\bigwedge
\renewcommand{\bigwedge}{\mathop{\textstyle\oldbigwedge}\displaylimits}
\let\oldcoprod\coprod
\renewcommand{\coprod}{\mathop{\textstyle\oldcoprod}\displaylimits}

\def\To#1{\mathchoice{\xrightarrow{\kern2pt#1\kern1.5pt}}{\stackrel{#1}{\longrightarrow}}{}{}}
\def\Mto#1{\mathrel{\mapstochar\To{#1}}}

\begin{document}

\title[Topological computation of some Stokes phenomena]{Topological computation of some Stokes phenomena on the affine line}

\author[A. D'Agnolo]{Andrea D'Agnolo}
\author[M. Hien]{Marco Hien}
\author[G. Morando]{Giovanni Morando}
\author[C. Sabbah]{Claude Sabbah}

\address[A.~D'Agnolo]{Dipartimento di Matematica,
Universit\`a di Padova,
via Trieste 63, 
35121 Padova, Italy}
\email{dagnolo@math.unipd.it}
\urladdr{http://docenti.math.unipd.it/dagnolo}
\address[M.~Hien]{Institut f\"ur Mathematik, Universit\"at Augsburg, 86135 Augsburg, Germany}
\email{marco.hien@math.uni-augsburg.de}
\address[G.~Morando]{I.T.T.\ Marconi, Padova, Italy}
\address[C.~Sabbah]{CMLS, \'Ecole polytechnique, CNRS, Universit\'e Paris-Saclay\\
F--91128 Palaiseau cedex\\
France}
\email{Claude.Sabbah@polytechnique.edu}
\urladdr{http://www.math.polytechnique.fr/perso/sabbah}

\thanks{The research of A.D'A. was supported by GNAMPA/INdAM and by grant CPDA159224 of Padova University.
The research of M.H. and G.M. was supported by the grant DFG-HI-1475/2-1 of the Deutsche Forschungsgemeinschaft.
The research of C.S.\ was supported by the grant ANR-13-IS01-0001-01 of the Agence nationale de la recherche.}

\begin{abstract}
Let $\mathcal{M}$ be a holonomic algebraic $\mathcal{D}$-module on the affine line, regular everywhere including at infinity. Malgrange gave a complete description of the Fourier-Laplace transform $\widehat{\mathcal{M}}$, including its Stokes multipliers at infinity, in terms of the quiver of $\mathcal{M}$. Let $F$ be the perverse sheaf of holomorphic solutions to $\mathcal{M}$. By the irregular Riemann-Hilbert correspondence, $\widehat{\mathcal{M}}$ is determined by the enhanced Fourier-Sato transform $F^\curlywedge$ of $F$. Our aim here is to recover Malgrange's result in a purely topological way, by computing $F^\curlywedge$ using Borel-Moore cycles. In this paper, we also consider some irregular $\mathcal{M}$'s, like in the case of the Airy equation, where our cycles are related to steepest descent paths.
\end{abstract}

\keywords{Perverse sheaf, enhanced ind-sheaf, Riemann-Hilbert correspondence, holonomic D-module, regular singularity, irregular singularity, Fourier transform, quiver, Stokes matrix, Stokes phenomenon, Airy equation, Borel-Moore homology}

\subjclass{34M40, 44A10, 32C38}

\maketitle


\section*{Introduction}

Let $\shm$ be a holonomic algebraic $\D$-module on the affine line $\V$. The Fourier-Laplace transform $\widehat\shm$ of $\shm$ is a holonomic $\D$-module on the dual affine line $\W$. The Riemann-Hilbert correspondence of \cite{DK16} (see \cite{KS16D,Kas16,Gui17} for expositions) associates to $\shm$ the enhanced ind-sheaf $F$ of its enhanced holomorphic solutions. By functoriality, the Riemann-Hilbert correspondence interchanges the Fourier-Laplace transform for holonomic $\D$-modules with the Fourier-Sato transform for enhanced ind-sheaves. (This was observed in \cite{KS16}, where the non-holonomic case is also discussed.) As a direct consequence, if~$F$ is defined over a subfield of $\C$, then so is the enhanced ind-sheaf $K$ associated with the Fourier-Laplace transform of $\shm$.

In \cite{Mal91}, Malgrange gave a complete description of $\widehat\shm$ if $\shm$ has only regular singularities, including at infinity. It turns out
that $\widehat\shm$ has $0$ and $\infty$ as its only possible singularities, and is regular at $0$. The exponential components of $\widehat\shm$ at $\infty$ are of linear type, with coefficients given by the singularities of $\shm$. Moreover, the Stokes multipliers of $\widehat\shm$ at $\infty$ are described in terms of the quiver of $\shm$. In such a case, the regular Riemann-Hilbert correspondence of \cite{Kas84,Meb84} associates to $\shm$ the perverse sheaf $F$ of its holomorphic solutions. As $\widehat\shm$ can be reconstructed from $K$, our aim here is to recover Malgrange's result in a purely topological way, by giving a complete description of the enhanced Fourier-Sato transform of $F$.

Before going into more detail, let us recall the classical local classification of a meromorphic connection at a singular point. The main idea is to examine the growth-properties of its solutions. There are various ways to implement this idea. One of them leads to the notion of a Stokes structure, a filtered local systems on a circle around the singular point (due to Deligne-Malgrange, see \cite{Del07}, \cite{Mal91}, \cite{Sab13}). Another possibility is to compute formal power series solutions and then use some machinery (multi-summation methods) to lift these to asymptotic solutions on sectors centered at the point (see \cite{BJL79}, \cite{vdPS03}). In the present case, since the exponents of the exponential factors are linear, it is known that one can work with two sectors of width slightly bigger than $\pi $ -- in the Deligne-Malgrange presentation, this means that the Stokes filtration can be trivialized on these sectors. Hence, the information can be completely encoded by two matrices, governing the transition of the asymptotic fundamental solutions or the glueing of the trivialized filtered local systems. Classically, such a pair of matrices $S_+, S_- $ is called the sequence of \emph{Stokes multipliers} of the system. It is certainly the most explicit way to encode the desired data. Of course, one has to be careful keeping track of all choices involved.

In the context of the irregular Riemann-Hilbert correspondence of D'Agnolo-Kashiwara, \cite{DK16}, the Stokes multipliers can be obtained from the associated enhanced solutions in the way described in \cite[\S9.8]{DK16}. In particular, Lemma 9.8.1 in loc.~cit.~ clarifies the ambiguities due to various choices. The relation to the Deligne-Malgrange approach mentioned above is discussed in \cite{Kas16}, section 8.

Let us describe our results (which were announced in \cite{Hie15})
with a few more details. Let $\Sigma\subset\V$ be a finite subset.
Denote by $\Perv_\Sigma(\C_\V)$ the abelian category of perverse
sheaves of $\C$-vector spaces on $\V$ with possible singularities in $\Sigma$. After suitably choosing a total order on $\Sigma$, we can encode an object $F\in\Perv_\Sigma(\C_\V)$ in its quiver $\bigl(\Psi, \Phi_c,\ u_c,\ v_c \bigr)_{c\in\Sigma}$, where $\Psi$ and $\Phi_c$ are finite dimensional vector spaces, and $u_c\colon\Psi\to\Phi_c$ and $v_c\colon\Phi_c\to\Psi$ are linear maps such that $1-u_c v_c$ is invertible for any~$c$ (or, equivalently, such that $1-v_c u_c$ is invertible for any $c$).

The main result of this paper is Theorem \ref{thm:Stokesmultipliers} which is the topological counterpart of Malgrange's result. It describes explicitly the Stokes multipliers of the enhanced Fourier-Sato transform $K$ of $F$ in terms of the quiver of $F$. We want to stress that the computation reduces to computation with ordinary perverse sheaves, since our starting~$F$ is so. In particular, we can reduce the computation to perverse sheaves which are isomorphic to their maximal extension (Beilinson's extension) at the singularities~$\Sigma$.

Recall that the classical Fourier-Sato transform sends $\Perv_{\Sigma}(\C_\V)$ to $\Perv_{\{0\}}(\C_{\W})$ (see~\cite{Mal91}). In Section \ref{sec:smash} we also give an explicit description of the quiver of the Fourier-Sato transform of $F$ in terms of that of $F$. For that purpose, we introduce the \emph{smash functor}, whose properties are explained in Appendix \ref{app:smash}. Together with the computation of the Stokes multipliers, this completes the description of $K$.

As another example of the method, we give in Section \ref{sec:moreStoles} a similar treatment for the Airy equation and for some elementary irregular meromorphic connections.

Sections \ref{sec:enh}--\ref{sec:Fouriertransform} recall the formalism of enhanced ind-sheaves together with the Riemann-Hilbert correspondence of \cite{DK16}, and the Fourier transforms in this context. Section \ref{sec:perverse} makes precise the correspondence which associates a quiver to a perverse sheaf on the affine line.

To conclude this introduction, let us mention some other works where the Stokes multipliers of $\widehat\shm$ are computed, for $\shm$ a holonomic algebraic $\D$-module on the affine line. In the book \cite{Mal91}, that we already mentioned, Malgrange discusses in fact the more general case where $\shm$ is regular at finite distance, but is allowed an irregularity at infinity of linear type. This condition is stable by Fourier-Laplace transform. In \cite{Moc10}, Mochizuki treats the case of a general holonomic algebraic $\D$-module~$\shm$, providing a way to compute the Stokes multipliers of $\widehat\shm$ in terms of steepest descent paths. An explicit description of the Stokes multipliers for some special classes of irregular holonomic $\D$-modules is discussed in \cite{Sab16} and \cite{HS14}.

In the context of Dubrovin's Conjectures in mirror symmetry, a lot of effort has been put into explicit computations of Stokes multipliers of the associated $\D $-module on the quantum cohomology of varieties in situations where one has some mirror symmetry statement -- e.g. \cite{Guz99}, \cite{MvdP15}, \cite{Ued05}. The mirror of the latter usually is given by the Fourier-Laplace transform of the regular singular Gauss-Manin system of some associated Landau-Ginzburg model. This situation therefore fits well into the framework examined in the present work.

\section{Enhanced ind-sheaves}\label{sec:enh}

Let us briefly recall from \cite{DK16} the triangulated category of enhanced
ind-sheaves (see \cite{DK16bis} for its natural $t$-structures).
This is similar to a construction in \cite{Tam08} (see~\cite{GS14} for an
exposition), in the framework of ind-sheaves from \cite{KS01}.
As we are interested in applications
to holonomic $\D$-modules, we restrict here to an analytic framework.

Let $M$ be a real analytic manifold and let $\field$ be a field (when considering the enhanced ind-sheaf attached to a $\D$-module, we set $\field=\C$).

\subsection{Sheaves} (Refer e.g.\ to \cite{KS90}.)
Denote by $\BDC(\field_M)$ the bounded derived category of sheaves of
$\field$-vector spaces on $M$.
For $S\subset M$ a locally closed subset, denote by $\field_S$ the zero
extension to $M$ of the constant sheaf on $S$ with stalk $\field$.

For $f\colon M\to N$ a morphism of real analytic manifolds, denote by $\tens$,
$\opb f$, $\reim f$ and $\rhom$, $\roim f$, $\epb f$ the six Grothendieck
operations for sheaves.

\subsection{Convolution}
Consider the maps
\[
\mu,q_1,q_2 \colon M\times\R\times\R \to M\times\R
\]
given by $q_1(x,t_1,t_2) = (x,t_1)$, $q_2(x,t_1,t_2) = (x,t_2)$, and
$\mu(x,t_1,t_2) = (x,t_1+t_2)$. For $F_1, F_2 \in \BDC(\field_{M\times\R})$,
the functors of convolution in the variable $t\in\R$ are defined by
\begin{align*}
F_1\tconv F_2 &= \reim \mu (\opb{q_1} F_1 \tens \opb{q_2} F_2), \\
\tchom(F_1,F_2)&= \roimv{q_{1\sep*}} \rhom(\opb{q_2} F_1, \epb\mu F_2).
\end{align*}

The convolution product $\tconv$ makes $\BDC(\field_{M\times\R})$ into a
commutative tensor category, with $\field_{M\times\{0\}}$ as unit object.
We will often write $\field_{\{t= 0\}}$ instead of $\field_{M\times\{0\}}$, and
similarly for $\field_{\{t\geq 0\}}$, $\field_{\{t\leq 0\}}$, etc.

\subsection{Enhanced sheaves} \label{sec:enhsheaves}
(Refer to \cite{Tam08,GS14} and \cite{DK16}.)
Consider the projection
\[
M\times\R \to[\pi] M.
\]
The triangulated category $\BEC[\field] M$ of enhanced sheaves is the quotient
of $\BDC(\field_{M\times\R})$ by the stable subcategory $\opb\pi\BDC(\field_M)$.
It splits as $\BEC[\field] M \simeq \BECp[\field] M \dsum \BECm[\field] M$,
where $\BECpm[\field] M$ is the quotient
of $\BDC(\field_{M\times\R})$ by the stable subcategory of objects $F$ satisfying $\field_{\{ \mp t \geq 0\}}\tconv F \simeq 0$.

The quotient functor
\[
\quot\colon \BDC(\field_{M\times\R}) \to \BEC[\field] M
\]
has a left and a right adjoint which are fully faithful.
Let us denote by $\widetilde\dere^{\mathrm{b}}_+(\field_M) \subset
\BDC(\field_{M\times\R})$ the essential image of $\BECp[\field]
M$ by the left adjoint, that is, the full subcategory whose objects $F$ satisfy $\field_{\{t\geq 0\}}\tconv F \simeq F$. Thus, one has an equivalence
\[
\quot\colon\widetilde\dere^{\mathrm{b}}_+(\field_M)\isoto\BECp[\field] M.
\]

The functor
\begin{equation}\label{eq:defeps}
\varepsilon \colon \BDC(\field_M) \to \widetilde\dere^{\mathrm{b}}_+(\field_M), \quad
F' \mapsto \field_{\{t\geq 0\}} \tens \opb\pi F'.
\end{equation}
is fully faithful.
For $f\colon M\to N$ a morphism of real analytic manifolds,
it interchanges the operations $\tens$, $\opb f$
and $\reim f$ with $\tconv$, $\opb{\tilde f}$ and $\reim{\tilde f}$, respectively. Here, we set
\[
\tilde f = f\times\id_\R\colon M\times\R \to
N\times\R.
\]

\subsection{Ind-sheaves}
(Refer to \cite{KS01}.)
An ind-sheaf is an ind-object in the category of sheaves with compact support.
There is a natural embedding of sheaves into ind-sheaves, and it has an exact
left adjoint $\alpha$ given by $\alpha(\smash{\indlim} F_i) = \smash{\ilim}
F_i$.
The functor $\alpha$ has an exact fully faithful left adjoint, denoted $\beta$.

Denote by $\BDC(\ifield_M)$ the bounded derived category of ind-sheaves. Denote
by $\tens$, $\rihom$, $\opb f$, $\roim f$, $\reeim f$, $\epb f$ the six
Grothendieck operations for ind-sheaves.

\subsection{Enhanced ind-sheaves}
(Refer to \cite{DK16}.)
Consider the morphisms
\begin{equation}
\label{eq:ipij}
M \to[i_\infty] M\times\PR \to[\pi] M,
\quad
M\times\R \to[j] M\times\PR,
\end{equation}
where $\PR=\R\union\{\infty\}$ is the real projective line, $i_\infty(x) = (x,\infty)$, $\pi$ is the projection, and $j$ is the embedding.

The triangulated category $\BEC M$ of enhanced ind-sheaves is defined by two
successive quotients of $\BDC(\ifield_{M\times\PR})$: first by the subcategory
of objects of the form $\roimv{i_{\infty\sep*}} F$, and then by the subcategory
of objects of the form $\opb\pi F$.
As for enhanced sheaves, the quotient functor has a left and a right adjoint
which are fully faithful. It follows that there are two realizations of $\BEC
M$ as a full subcategory of $\BDC(\ifield_{M\times\PR})$.

Enhanced ind-sheaves are endowed with an analogue of the convolution functors, denoted $\ctens$ and
$\cihom$.
For $f\colon M\to N$ a morphism of real analytic manifold, one also has external operations $\Eopb f$, $\Eoim f$, $\Eeeim f$, $\Eepb f$.
Here, $\Eopb f$ is the functor induced by $\opb{\tilde f}$ at the level of ind-sheaves, and similarly for the other operations.

There is a natural embedding $\BEC[\field] M\subset\BEC M$ induced by $\reim j$ or, equivalently, by $\roim j$.
Set
\[
\field_M^\enh = \indlim[a\rightarrow+\infty] \quot\field_{\{t\geq a\}} \in \BEC M.
\]
There is a natural fully faithful functor
\[
e:\BDC(\ind\field_M) \to \BEC M \ , \ F \mapsto \field_M^\enh \otimes \pi^{-1}F .
\]
We will need the following two lemmas on compatibility between operations for enhanced ind-sheaves and for usual sheaves with an additional variable.

By \cite[Remark 3.3.21, Proposition 4.7.14]{DK16}, one has

\begin{lemma}\label{lem:ops}
Let $F,F_1,F_2\in\BDC(\field_{M\times\R})$
and $G\in\BDC(\field_{N\times\R})$. One has
\begin{align*}
(\field_M^\enh \ctens \quot F_1) \ctens (\field_M^\enh \ctens \quot F_2) &\simeq
\field_M^\enh \ctens (\quot F_1 \ctens \quot F_2) \\
&\simeq \field_M^\enh \ctens \quot(F_1 \tconv F_2), \\
\Eopb f\quot(\field_N^\enh \ctens G) &\simeq \field_M^\enh \ctens
\quot\opb{\tilde f}G.
\end{align*}
If moreover $f$ is proper, then
\[
\Eeeim f\quot(\field_M^\enh \ctens F) \simeq \field_N^\enh \ctens
\quot\reim{\tilde f}F.
\]
\end{lemma}

The category $\BEC M$ has a natural hom functor $\fhom$ with values in $\BDC(\field_M)$.
By \cite[Proposition 4.3.10]{DK16} and \cite[Corollary 6.6.6]{KS16D}, one has

\begin{lemma}\label{lem:pi}
For $F\in\BDC(\field_{M\times\R})$ one has
\[
\fhom(\field_{\{t\geq 0\}}, \field_M^\enh \ctens \quot F) \simeq
\roim\pi(\field_{\{t\geq 0\}} \tconv F).
\]
\end{lemma}

\subsection{\texorpdfstring{$\R$}{R}-constructibility}
(Refer to \cite{DK16}.)
Denote by $\BDC_{\Rc}(\field_M)$ the full
subcategory of objects with $\R$-constructible cohomologies.
Using notations \eqref{eq:ipij},
denote $\BDC_{\Rc}(\field_{M\times\R_\infty})$ the full subcategory of
$\BDC_{\Rc}(\field_{M\times\R})$ whose objects $F$ are such that $\reim j F$ (or, equivalently, $\roim j F$) is
$\R$-constructible in $M\times\PR$.
Since $\reim j$ is fully faithful, we will also consider
$\BDC_{\Rc}(\field_{M\times\R_\infty})$ as a full subcategory of
$\BDC_{\Rc}(\field_{M\times\PR})$.

The triangulated category $\Tam M$ of $\R$-constructible enhanced sheaves is
the full subcategory of $\BDC_{\Rc}(\field_{M\times\R_\infty})$ whose objects
$F$ satisfy $F \simeq \field_{\{t\geq 0\}} \tconv F$.
It is a full subcategory of $\widetilde\dere^{\mathrm{b}}_+(\field_M)$.
Denote by $\BECRc[\field] M$ the full subcategory of $\BECp[\field] M$ whose objects are isomorphic to $\quot F$, for some $F\in\Tam M$.

The category $\BECRc M$ of $\R$-constructible enhanced ind-sheaves is defined as the full subcategory of $\BEC M$ whose objects $K$
satisfy the following property: for any relatively compact open subset
$U\subset M$ there exists $F\in \Tam M$ such that
\[
\opb\pi\field_U \tens K \simeq \field_M^\enh \ctens \quot F.
\]
The object $\field_M^\enh$ plays the role of the constant sheaf in $\BECRc M$.

\section{Riemann-Hilbert correspondence}

We recall here the Riemann-Hilbert correspondence for (not necessarily regular)
holonomic $\D$-modules established in \cite{DK16} (see \cite{DK16bis} for its $t$-exactness).
This is based on the theory of ind-sheaves from \cite{KS01}, and influenced by
the works~\cite{Tam08} and \cite{DAg14}.
One of the key ingredients of its proof is the description of the structure of
flat meromorphic connections by \cite{Moc11} and \cite{Ked11}.

Let $X$ be a complex manifold.
We set for short $d_X = \dim X$.

\subsection{\texorpdfstring{$\D$}{D}-modules}
(Refer e.g.\ to \cite{Kas03}.)
Denote by $\O_X$ and $\D_X$ the rings of holomorphic functions and of
differential operators, respectively.

Denote by $\BDC(\D_X)$ the bounded derived category of left $\D_X$-modules.
For $f\colon X\to Y$ a morphism of complex manifolds, denote by $\dtens$,
$\dopb f$, $\doim f$ the operations for $\D$-modules.

Consider the solution functor
\[
\sol_X \colon \BDC(\D_X)^\op \to \BDC(\C_X), \quad
\shm \mapsto \rhom[\D_X](\shm,\O_X).
\]

Denote by $\BDC_{\hol}(\D_X)$ and $\BDC_{\reghol}(\D_X)$ the full subcategory of
$\BDC(\D_X)$ of objects with holonomic and regular holonomic cohomologies, respectively, and by
$\BDC_{\ghol}(\D_X)$ the full subcategory
of objects with good and holonomic cohomologies.

Let $D\subset X$ be a complex analytic hypersurface and denote by
$\O_X(*D)$ the sheaf of meromorphic functions with poles along $D$.
Set $U=X\setminus D$.

For $\varphi\in\O_X(*D)$, set
\begin{align*}
\D_X e^\varphi &= \D_X/\{P\semicolon Pe^\varphi=0 \text{ on } U\}, \\
\she^\varphi_{U|X}&=\D_X e^\varphi \dtens \O_X(*D).
\end{align*}

\subsection{Tempered solutions}
(Refer to \cite{KS01}.)
By the functor $\beta$, there is a natural notion of $\beta\D_X$-module in the
category of ind-sheaves, cp. \cite[5.6]{KS01}. We denote by $\BDC(\ind\D_X)$ the corresponding derived category.

Denote by $\Ot_X\in \BDC(\ind\D_X)$
the complex of tempered holomorphic functions.
It is related to the functor $\thom$ of \cite{Kas84} by the relation
\[
\alpha\rihom(F,\Ot_X) \simeq \thom(F,\O_X)
\]
for any $F\in\BDC_\Rc(\C_X)$.

Consider the tempered solution functor
\[
\solt_X \colon \BDC(\D_X)^\op \to \BDC(\ind\C_X), \quad
\shm \mapsto \rhom[\D_X](\shm,\Ot_X).
\]

\subsection{Enhanced solutions}
(Refer to \cite{DK16}.)
There is a natural notion of $\D_X$-module in the category
of enhanced ind-sheaves, and we denote by $\BEC[\ind\D] X$ the corresponding
triangulated category.

Let $\PP=\C\union\{\infty\}$ be the complex projective line, and let
\[
i\colon X\times\PR \to X\times\PP
\]
be the closed embedding.
Denote by $\tau\in\PP$ the affine coordinate, so that $\tau\in\O_\PP(*\infty)$.
Consider the exponential $\D_\PP$-module $\she_{\C|\PP}^\tau$.

The enhanced solution functor is given by
\[
\solE_X \colon \BDC(\D_X)^\op \to \BEC[\ind\C] X, \quad
\shm \mapsto \quot\epb i \solt_{X\times\PP}(\shm\detens\she_{\C|\PP}^\tau)[2],
\]
where $\detens$ denotes the exterior tensor product for $\D$-modules.

The functorial properties of $\solE$ are summarized in the next theorem. The
statements on direct and inverse images are easy consequences of the
corresponding results for tempered solutions obtained in \cite{KS01}. The
statement on the tensor product is specific to enhanced solutions and is due to
\cite[Corollary 9.4.10]{DK16}.

\begin{theorem}
\label{thm:solEfunct}
Let $f\colon X\to Y$ be a complex analytic map.
Let $\shm\in\BDC_\ghol(\D_X)$, $\shm_1,\shm_2\in\BDC_\hol(\D_X)$ and
$\shn\in\BDC_\hol(\D_Y)$.
Assume that $\supp\shm$ is proper over $Y$.
Then one has
\begin{align*}
\solE_X(\dopb f\shn) &\simeq \Eopb f\solE_Y(\shn), \\
\solE_Y(\doim f\shm)[d_Y] &\simeq \Eeeim f\solE_X(\shm)[d_X], \\
\solE_X(\shm_1) \ctens \solE_X(\shm_2) &\simeq \solE_X(\shm_1 \dtens \shm_2).
\end{align*}
\end{theorem}

\begin{notation}\label{notation:Evi}
Let $D\subset X$ be a closed hypersurface and set $U=X\setminus D$.
For $\varphi\in\O_X(*D)$, we set
\begin{align*}
E^{\varphi} &\defeq \field_{\{ t+\Re\varphi(x)\geq0\} } \in \Tam M, \\
\enhE^{\varphi} &\defeq \field_M^\enh \ctens \quot E^{\varphi}\in \BECRc M,
\end{align*}
where we set for short
\[
\{t+\Re\varphi(x)\geq0 \} = \{(x,t)\in X\times\R \semicolon x\in U,\
t+\Re\varphi(x)\geq0\}.
\]
\end{notation}

We will also need the following computation from \cite[Corollary 9.4.12]{DK16}.

\begin{theorem}\label{thm:SolEExp}
With the above notations, for $\field=\C$, one has
\[
\solE_X(\she^\varphi_{U|X}) \simeq \enhE^{\varphi}.
\]
\end{theorem}

\subsection{Riemann-Hilbert correspondence}
Let us state the Riemann-Hilbert correspondence for holonomic
$\D$-modules established in \cite[Theorem 9.5.3]{DK16}.

\begin{theorem}
The enhanced solution functor induces a
fully faithful functor
\[
\solE_X \colon \BDC_\hol(\D_X)^\op \to \BECRc[\ind\C] X.
\]
Moreover, there is a functorial way of reconstructing $\shm\in\BDC_\hol(\D_X)$ from $\solE_X(\shm)$.
\end{theorem}

This implies in particular that the Stokes structure of a flat meromorphic
connection $\shm$ is encoded topologically in $\solE(\shm)$.
For this, we refer to \cite[\S9.8]{DK16} (see also \cite[\S8]{Kas16}).

\subsection{A lemma}
We will later use the following remark.
Let $D\subset X$ be a closed hypersurface,
set $U=X\setminus D$, and denote by $j\colon U\to X$ the embedding.

\begin{lemma}\label{lem:UFM}
With the above notations, let $\shm\in\BDC_\hol(\D_X)$ be such that
$\shm\simeq\shm(*D)$.
Assume that $X$ is compact.
Then there exists $F\in\Tam[\C] U$ such that
$\reim {\tilde\jmath} F\in\Tam[\C] X$ and
\[
\solE_X(\shm) \simeq \C_X^\enh \ctens \quot\reim {\tilde\jmath} F.
\]
\end{lemma}

\begin{proof}
Set $K=\solE_X(\shm)$.
Since $\shm$ is holonomic, $K$ is $\R$-constructible.
Since $X$ is compact,
there exists $F' \in \Tam[\C] X$ with
$K \simeq \C_X^\enh \ctens \quot F'$.
Since $\shm\simeq\shm(*D)$, one has
$K \simeq \opb\pi\C_U \tens K \simeq \C_X^\enh \ctens \quot(\opb\pi\C_U \tens F')$.
Hence $F=\opb {\tilde\jmath} F'$ satisfies the assumptions in the statement.
\end{proof}

\section{Fourier transform}\label{sec:Fouriertransform}

By functoriality, the enhanced solution functor interchanges integral transforms at the level of holonomic $\D$-modules with integral transforms at the level of enhanced ind-sheaves. (This was observed in \cite{KS16}, where the non-holonomic case is also discussed.)
We recall here some consequences of this fact, dealing in particular with the Fourier transform.

\subsection{Integral transforms}

Consider a diagram of complex manifolds
\begin{equation}
\label{eq:SXY}
\xymatrix@C=1ex{
& S \ar[dl]_p \ar[dr]^q \\
X && Y.
}
\end{equation}
At the level of $\D$-modules, the integral transform with kernel
$\shl\in\BDC(\D_S)$ is the functor
\[
\ast\dcomp\shl\colon\BDC(\D_X) \to \BDC(\D_Y), \qquad
\shm\dcomp\shl = \doim q(\shl \dtens \dopb p \shm).
\]
At the level of enhanced ind-sheaves, the integral transform with kernel
$H\in\BEC S$ is the functor
\[
\ast\Ecomp H\colon \BEC X \to \BEC Y, \qquad
K\Ecomp H = \Eeeim q(H \ctens \Eopb p K).
\]
Combining the isomorphisms in Theorem~\ref{thm:solEfunct}, one has

\begin{corollary}
\label{cor:Dint}
Let $\shm\in\BDC_\ghol(\D_X)$ and $\shl\in\BDC_\ghol(\D_S)$.
Assume that $\opb p\supp(\shm) \cap \supp(\shl)$ is proper over $Y$.
Set $K=\solE_X(\shm)$ and $H=\solE_S(\shl)$.
Then there is a natural isomorphism in $\BECRc Y$
\[
\solE_Y(\shm \dcomp \shl) \simeq
K \Ecomp H [d_S-d_Y].
\]
\end{corollary}

\begin{remark}
There is a similar statement with the solution functor replaced by the de Rham
functor.
This has been extended in \cite{KS16} to the case where $\shm$ is good, but not
necessarily holonomic.
\end{remark}

\subsection{Globally \texorpdfstring{$\R$}{R}-constructible enhanced ind-sheaves}

Consider the diagram of real analytic manifolds induced by \eqref{eq:SXY}
\[
\xymatrix@C=1ex{
& S\times\R \ar[dl]_{\tilde p} \ar[dr]^{\tilde q} \\
X\times\R && Y\times\R,
}
\]
where $\tilde p = p\times\id_\R$ and $\tilde q = q\times\id_\R$.

The natural integral transform
for $\R$-constructible enhanced sheaves with kernel $L\in\Tam S$
is the functor
\[
\ast\ccomp L\colon \Tam X \to \Tam Y, \qquad
F \ccomp L = \reim {\tilde q}(L \tconv \opb {\tilde p} F).
\]

Combining Lemma~\ref{lem:ops} and Corollary~\ref{cor:Dint} we get

\begin{proposition}
Let $\shm\in\BDC_\ghol(\D_X)$, $\shl\in\BDC_\ghol(\D_S)$,
and assume that $\opb p\supp(\shm) \cap \supp(\shl)$ is proper over $Y$.
Let $F\in\Tam[\C] X$, $L\in\Tam[\C] S$,
and assume that there are isomorphisms
\begin{equation}
\label{eq:mFlL}
\solE_X(\shm) \simeq \C_X^\enh \ctens \quot F, \quad
\solE_S(\shl) \simeq \C_S^\enh \ctens \quot L.
\end{equation}
Then there is a natural isomorphism in $\BECRc Y$
\[
\solE_Y(\shm \dcomp \shl) \simeq
\C_Y^\enh \ctens \quot (F \ccomp L) [d_S-d_Y].
\]
\end{proposition}

Note that if $X$ and $S$ are compact, then for any $\shm\in\BDC_\hol(\D_X)$ and
$\shl\in\BDC_\hol(\D_S)$ there exist $F\in\Tam[\C] X$ and $L\in\Tam[\C] S$ satisfying
\eqref{eq:mFlL}.

\subsection{Fourier-Laplace transform}

Let $\V$ be a finite dimensional complex vector space, denote by $\PP$ its
projective compactification, and set $\HH = \PP \setminus \V$.

\begin{definition}
Let $\BDC_\hol(\D_\Vi)$ be the full triangulated subcategory of
$\BDC_\hol(\D_\PP)$ whose objects $\shm$ satisfy $\shm \simeq \shm(*\HH)$.
\end{definition}

Let $\W$ be the dual vector space of $\V$, denote
by $\bb$ its projective compactifications, and set $\hh = \bb \setminus \W$.
The pairing
\[
\V \times \W \to \C, \quad (z,w) \mapsto \langle z,w \rangle
\]
defines a meromorphic function on $\PP\times\bb$
with poles along $$(\PP \times \hh) \union (\HH \times \bb) =
(\PP\times\bb) \setminus (\V\times\W).$$
Consider the projections
\begin{equation}
\xymatrix@C=1ex{
& \PP\times\bb \ar[dl] \ar[dr] \\
\PP && \bb.
}
\end{equation}

\begin{definition}
Set
\[
\shl = \she^{\langle z,w\rangle}_{\V\times\W | \PP\times\bb}, \quad
\shl^a = \she^{-\langle w,z\rangle}_{\W\times\V | \bb\times\PP}.
\]
(Beware that the signs in $\pm\langle w,z\rangle$ are sometimes reversed in the literature.)
The Fourier-Laplace transform of $\shm\in\BDC_\hol(\D_\Vi)$ is given by
\[
\shm^\wedge = \shm \dcomp \shl\in\BDC_\hol(\D_\Wi).
\]
The inverse Fourier-Laplace transform of $\shn\in\BDC_\hol(\D_\Wi)$ is given by
\[
\shn^\vee = \shn \dcomp \shl^a\in\BDC_\hol(\D_\Vi).
\]
\end{definition}

\begin{theorem}
\label{thm:FLD}
The Fourier-Laplace transform gives an equivalence of categories
\[
\wedge\colon\BDC_\hol(\D_\Vi) \isoto \BDC_\hol(\D_\Wi).
\]
A quasi-inverse is given by $\shn\mapsto\shn^\vee$.
\end{theorem}

This result is classical, see~\cite[App.\,2]{Mal91}. The idea of the proof is as follows.
Denoting by $D_\V$ the Weyl algebra, there is an equivalence of categories
$\BDC_\hol(D_\V) \simeq \BDC_\hol(\D_\Vi)$.
Under this equivalence, the Fourier-Laplace transform is induced by the
algebra isomorphism $D_\V \simeq D_{\W}$ given by $z_i\mapsto -\partial_{w_i}$,
$\partial_{z_i}\mapsto w_i$.

Using the Riemann-Hilbert correspondence and a result of~\cite{Tam08},
we give an alternative topological proof of
Theorem~\ref{thm:FLD} in Remark~\ref{rem:FouT} below.

\subsection{Enhanced Fourier-Sato transform}

Recall that $j\colon\V\to\PP$ denotes the embedding

\begin{definition}\label{def:Vi}
Let $\Tam \Vi$ be the full triangulated subcategory of
$\Tam \V$ whose objects $F$ satisfy
$\reim{\tilde\jmath}F\in \Tam \PP$.
\end{definition}

Consider the projections
\[
\xymatrix@C=.5em{
& \V\times \W\times\R \ar[dr]^{\tilde q} \ar[d]_{\overline p} \ar[dl]_{\tilde
p} \\
\V\times\R & \V & \W\times\R.
}
\]

\begin{definition}\label{def:FouSato}
Using Notation~\ref{notation:Evi}, set
\[
L = E^{\langle z,w\rangle}, \quad
L^a = E^{-\langle w,z\rangle}.
\]
The enhanced Fourier-Sato transform of $F \in \Tam \Vi$ is given by
\[
F^\curlywedge = F \ccomp L
\in \Tam\Wi.
\]
The enhanced inverse Fourier-Sato transform of $G \in \Tam \Wi$ is given by
\[
G^\curlyvee = G \ccomp L^a
\in \Tam\Vi.
\]
\end{definition}

The transform $F\mapsto F^\curlywedge$ has been investigated by
Tamarkin~\cite{Tam08} (in the more general case of vector spaces over $\R$).
He proved in particular the following result.

\begin{proposition}[{{\cite[Theorem~3.5]{Tam08}}, see
also~\cite[\S5.1]{KS16}}]\label{pro:FouT}
The enhanced Fourier-Sato transform gives an equivalence of categories
\[
\curlywedge\colon\Tam\Vi \isoto \Tam\Wi.
\]
A quasi-inverse is given by $G \mapsto G^\curlyvee[2\dim\V]$.
\end{proposition}

The following result will be used when dealing with
enhanced solutions of regular holonomic $\D$-modules.
It was noticed in~\cite[Proposition~A.3]{DAg14}. We let $\varepsilon$ as in \eqref{eq:defeps}.

\begin{lemma}\label{lem:F'TamFou}
For $F'\in \BDC_\Rc(\field_\Vi)$ in the notation analogous to Definition \ref{def:Vi}, one has
\[
\varepsilon(F')^\curlywedge \simeq \reim{\tilde q}(L\tens\opb{\overline p} F').
\]
\end{lemma}

As explained in \cite[\S3]{DK16} there is a natural notion of ``bordered space'', of which $\V_\infty = (\V,\PP)$ is an example.
The notion of enhanced ind-sheaves naturally extends to this setting (see \cite{KS16} or \cite{DK16bis}), and there is a natural equivalence of $\BEC {\V_\infty}$ with the full subcategory of $\BEC \PP$ whose objects~$K$ satisfy
$K\simeq \opb\pi\field_\V \tens K$. There are a natural functor
\[
\quot: \widetilde\dere^{\mathrm{b}}_+(\field_\V) \to \BEC {\V_\infty},
\]
and a natural object $\field^\enh_{\V_\infty}\in\BEC {\V_\infty}$.

Similarly to Definition \ref{def:FouSato}, there is an enhanced Fourier-Sato transform for ind-sheaves
\[
\curlywedge\colon \BEC {\Vi} \isoto \BEC{\Wi}.
\]
By \cite[Corollary 9.4.12]{DK16}, one has
\[
\solE_{\V_\infty\times\W_\infty}(\shl) \simeq \enhE^{\langle z,w\rangle}\defeq \C^\enh_{\V_\infty\times\W_\infty} \ctens \quot L.
\]
The next result is then easily checked.

\begin{lemma}\mbox{}\label{lem:FouML}
\begin{itemize}
\item[(i)]
Let $\shm\in\BDC_\hol(\D_\Vi)$ and $F\in\Tam[\C] \Vi$ satisfy
\begin{equation}
\label{eq:MFhyp}
\solE_{\Vi}(\shm) \simeq \C_{\Vi}^\enh \ctens \quot F.
\end{equation}
Then, there is an isomorphism
\[
\solE_{\Wi}(\shm^\wedge) \simeq \C_{\Wi}^\enh \ctens \quot F^\curlywedge [\dim\V].
\]
\item[(ii)]
Let $\shm\in\BDC_\reghol(\D_\Vi)$ and set $F'=\sol_\PP(\shm)|_\V$.
Then, there is an isomorphism
\[
\solE_{\Wi}(\shm^\wedge) \simeq \C_{\Wi}^\enh \ctens
\quot \reim{\tilde q}(L \tens \opb{\overline p} F') [\dim\V].
\]
\end{itemize}
Analogous results hold for $\wedge$ and $\curlywedge$ replaced by $\vee$
and $\curlyvee$, respectively.
\end{lemma}

Note that, in view of Lemma~\ref{lem:UFM}, for any $\shm\in\BDC_\hol(\D_\Vi)$
there is an $F\in\Tam[\C] \Vi$ satisfying \eqref{eq:MFhyp}.

\begin{remark}
\label{rem:FouT}
Lemma~\ref{lem:FouML}~(i) and Proposition~\ref{pro:FouT} imply Theorem \ref{thm:FLD}, due to the Riemann-Hilbert correspondence. In fact, with
notations as in \eqref{eq:MFhyp}, one has
\begin{align*}
\solE_{\Vi}(\shm^{\wedge\vee})
&\simeq \C_{\Vi}^\enh \ctens \quot F^{\curlywedge\curlyvee}[2\dim\V]\\
&\simeq \C_{\Vi}^\enh \ctens \quot F
\simeq \solE_{\Vi}(\shm).
\end{align*}
Hence $\shm^{\wedge\vee} \simeq \shm$.
One similarly gets $\shn^{\vee\wedge} \simeq \shn$
for $\shn\in \BDC_\hol(\D_{\W_\infty})$.
\end{remark}

\section{Perverse sheaves on the affine line}\label{sec:perverse}

The results in this section concern perverse sheaves on the affine line and
their quivers. These are classical and well-known results (see e.g.\ \cite{GMV96}). However, we present them here from a point of view slightly different than usual, better suited to our needs. In particular, we give a purely topological description of Beilinson's maximal extension.

From now on, let $\V$ denote a complex affine line.

\subsection{Constructible complexes and perverse sheaves}
Let $\Sigma\subset\V$ be a discrete subset and let us denote by $\BDC_\Cc(\field_\V,\Sigma)$ the full triangulated subcategory of $\BDC(\field_\V)$ whose objects $F$ have cohomology sheaves which are $\C$-constructible with respect to the stratification $(\V\setminus\Sigma, \Sigma)$. In particular, $H^j(F)_{|\V\setminus\Sigma}$ is a locally constant sheaf. The abelian category $\Perv_\Sigma(\field_\V)$ of $\Sigma$-perverse sheaves is the heart of the middle perversity $t$-structure on $\BDC_\Cc(\field_\V,\Sigma)$. Note that $\field_\V[1]$ and $\field_\Sigma$ are objects of $\Perv_\Sigma(\field_\V)$ and, for an object $F$ of $\Perv_\Sigma(\field_\V)$, the shifted restriction $F_{|\V\setminus\Sigma}[-1]$ is a locally constant sheaf that we usually denote by~$L$.

\begin{definition}
Let $\Sigma, X \subset \V$. We say that $X$ is $\Sigma$-negligible if it is
disjoint from~$\Sigma$, locally closed, simply connected, and if $\rsect_\rc(\V;
\field_X) \simeq 0$.
\end{definition}

\begin{lemma}
\label{lem:pervbasic}
Let $\Sigma\subset\V$ be a discrete subset, and let $X\subset \V$ be
$\Sigma$-negligible.
If $F \in\BDC_\Cc(\field_\V,\Sigma)$, then
$$
\rsect_\rc(\V; \field_X\tens F) \simeq 0.\eqno\qed
$$
\end{lemma}

\begin{lemma}
\label{lem:pervbasic2}
Let $\Sigma\subset\V$ be a discrete subset, and let $B_c\subset\V$ be a convex
open neighborhood of $c\in \Sigma$ such that $B_c\cap\Sigma = \{c\}$.
If $F \in\BDC_\Cc(\field_\V,\Sigma)$, then
\begin{align*}
&\rsect_\rc(\V; \field_{\{c\}} \tens F) \simeq \opb i_c F
\simeq \rsect(B_c; F), \\
&\rsect_\rc(\V; \field_{B_c} \tens F) \simeq \epb{i_c} F
\simeq \rsect_{\{c\}}(\V;F),
\end{align*}
where $i_c\colon \{c\} \to \V$ denotes the embedding.\qed
\end{lemma}

We will also use the following general lemma.

\begin{lemma}\label{lem:vanishingsemiopen}
If $\ell$ is a semi-open interval, $X$ is a manifold and $L$ is a locally constant sheaf on $X\times\ell$, then $\rsect_\rc(X\times\ell; L)=0$.\qed
\end{lemma}

\subsection{From perverse sheaves to quivers}\label{sse:pvqv}
Assume that $\Sigma\subset\V$ is finite. We will denote by $\mathsf{Quiv}_\Sigma(\field)$ the category whose objects are the tuples
\[
\left(\Psi, \Phi_c, u_c, v_c\right)_{c\in\Sigma},
\]
where $\Psi$ and $\Phi_c$ are finite dimensional $\field$-vector spaces, and $u_c\colon\Psi\!\to\!\Phi_c$ and $v_c\colon\Phi_c\!\to\!\Psi$ are $\field$-linear maps such that $1-u_c v_c$ is invertible for any~$c$ (or, equivalently, such that $1-v_c u_c$ is invertible for any $c$). Morphisms in $\mathsf{Quiv}_\Sigma(\field)$ are the natural ones. Our aim in this section is to present a functor
\[
Q_\Sigma^{(\alpha,\beta)} \colon \Perv_\Sigma(\field_\V) \to \mathsf{Quiv}_\Sigma(\field)
\]
which, as we will recall in the next section, provides an equivalence of categories. This functor depends on choices $(\alpha,\beta)$ that we fix now, for the remaining part of this section.

\begin{convention}
\label{con:alphabeta}
We fix $\alpha\in\V\setminus\{0\}$ and $\beta\in\W\setminus\{0\}$ such that
\begin{align}
\label{eq:alpha2}
&\Re\langle\alpha,\beta\rangle = 0, \\
\label{eq:beta}
&\Re\langle c-c',\beta\rangle \neq 0, \quad\forall c, c'\in\Sigma,\ c\neq c'.
\end{align}
\end{convention}

\begin{notation}\mbox{}\label{nota:Bell}
\begin{enumerate}
\item\label{nota:Bell1}
The $\R$-linear projection
\[
p_\beta\colon \V_\R \to \R, \quad z \mapsto \Re\langle z,\beta\rangle
\]
enables one, according to \eqref{eq:beta}, to define a total order on $\Sigma$ by the formula
\[
c\mathbin{<_\beta}c' \quad\text{if}\quad p_\beta(c) < p_\beta(c').
\]
We will enumerate the elements of $\Sigma$ as
\[
c_1\mathbin{<_\beta}c_2\mathbin{<_\beta}\cdots\mathbin{<_\beta}c_n,
\]
and we will set
\[
r_i=p_\beta(c_i),\ r_0=-\infty,\ r_{n+1}=+\infty \quad (i=1,\dots,n),
\]
so that $-\infty=r_0<r_1<\cdots<r_n<r_{n+1}=+\infty$.

\item\label{nota:Bell2}
For $i=1,\dots,n$, the open bands
\[
B_{c_i}(\beta) \defeq p_\beta^{-1}(\,(r_{i-1},r_{i+1})\,)
\]
cover $\V$ and satisfy $B_{c_i}(\beta)\cap\Sigma = \{c_i\}$. We set
\begin{align*}
B_{c_i}^>(\beta) &\defeq p_\beta^{-1}(\,(r_{i},r_{i+1})\,),&
B_{c_i}^\leq(\beta) &\defeq p_\beta^{-1}(\,(r_{i-1},r_{i}]\,),\\
\ell_{c_i}(\alpha) &\defeq c_i + \R_{\geq 0}\alpha,&
\ell^\times_{c_i}(\alpha) &\defeq c_i + \R_{> 0}\alpha.
\end{align*}
Note that the half lines $\ell_{c_i}(\alpha)$ are disjoint, and set
\[
\ell_\Sigma(\alpha) \defeq \ell_{c_1}(\alpha) \cup \ell_{c_2}(\alpha)
\cup \dots \cup \ell_{c_n}(\alpha).
\]
We will write for short $B_{c_i}$, $B_{c_i}^\leq$, $\ell^\times_{c_i}$, etc., instead of
$B_{c_i}(\beta)$, $B_{c_i}^\leq(\beta)$, $\ell^\times_{c_i}(\alpha)$, etc.
These sets are pictured in Figure~\ref{fig:Quiver1}.
\begin{figure}
\includestandalone[scale=.8]{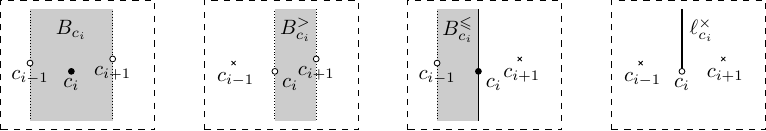}
\caption{The sets $B_{c_i}$, $B_{c_i}^>$, $B_{c_i}^\leq$, and
$\ell^\times_{c_i}$.}\label{fig:Quiver1}
\end{figure}
\end{enumerate}
\end{notation}

\begin{definition}\label{def:vannearcycles}
Let $c\in \Sigma$ and $F \in\BDC_\Cc(\field_\V,\Sigma)$. The complexes of \emph{nearby cycles at $c$}, \emph{vanishing cycles at $c$}, and \emph{global nearby cycles} are respectively defined by the formulas
\begin{align*}
\Psi_c^{(\alpha,\beta)}(F)=\Psi_c(F) &\defeq \rsect_\rc(\V; \field_{\ell^\times_c} \tens F), \\
\Phi_c^{(\alpha,\beta)}(F)=\Phi_c(F) &\defeq \rsect_\rc(\V;\field_{\ell_c} \tens F),\\
\Psi^{(\alpha,\beta)}(F)=\Psi(F) &\defeq \rsect_\rc(\V;\field_{\V\setminus\ell_\Sigma} \tens F)[1].
\end{align*}
\end{definition}

We construct below the morphisms $u_{cc}:\Psi_c(F)\to\Phi_c(F)$ and $v_{cc}:\Phi_c(F)\to\Psi_c(F)$ and we define the monodromy as
\[
T_{cc}:=1-v_{cc}u_{cc}:\Psi_c(F)\to\Psi_c(F).
\]
On the one hand, the exact sequence $0\to\field_{\ell^\times_c}\to\field_{\ell_c}\to\field_{\{c\}}\to0$ enables us to obtain the distinguished triangle defining $u_{cc}$:
\[
\Psi_c(F)\to[u_{cc}]\Phi_c(F)\to \opb{i_c}F\to[+1].
\]
In order to define the morphism
\[
v_{cc}:\Phi_c(F)\to\Psi_c(F),
\]
we consider the exact sequences
\begin{equation}
\label{eq:sesiso}
\begin{split}
&0 \to \field_{B_c^>} \to \field_{B_c^> \cup \ell^\times_c} \to \field_{\ell^\times_c} \to 0, \\
&0 \to \field_{B_c^>} \to \field_{B_c \setminus \ell_c} \to \field_{B_c^\leq \setminus \ell_c} \to 0.
\end{split}
\end{equation}
Since $B_c^\leq \setminus \ell_c$ and $B^>_c \cup \ell^\times_c$ are $\Sigma$-negligible, they give rise to the isomorphisms
\begin{equation}\label{eq:isoPsicPsic}
\Psi_c(F)\isoto\rsect_\rc(\V;\field_{B^>_c} \tens F)[1]\isoto\rsect_\rc(\V;\field_{B_c\setminus\ell_c} \tens F)[1].
\end{equation}
On the other hand, the exact sequence
\[
0 \to \field_{B_c\setminus\ell_c} \to \field_{B_c} \to \field_{\ell_c} \to 0
\]
gives rise to the morphism
\[
v_{cc}:\Phi_c(F)\to\rsect_\rc(\V;\field_{B_c\setminus\ell_c} \tens F)[1]\underset{\eqref{eq:isoPsicPsic}}\simeq\Psi_c(F).
\]

For each $c\in\Sigma$, we have a short exact sequence
\begin{equation}
\label{eq:sesisoS}
0 \to \field_{B_c \setminus \ell_c} \to \field_{\V \setminus \ell_\Sigma} \to
\field_{(\V \setminus \ell_\Sigma) \setminus B_c} \to 0,
\end{equation}
and we note that $(\V\setminus \ell_\Sigma) \setminus B_c$ is $\Sigma$-negligible. We deduce an isomorphism
\refstepcounter{equation}
\begin{multline}\label{eq:isoPsicPsi}\tag*{{(\theequation)$_\eqc$}}
\Psi_c(F)\xrightarrow[\eqref{eq:isoPsicPsic}]{~\sim~}\rsect_\rc(\V;\field_{B_c\setminus\ell_c} \tens F)[1]\\[-5pt]
\isoto\rsect_\rc(\V;\field_{\V\setminus\ell_\Sigma} \tens F)[1]=\Psi(F)
\end{multline}
for each $c \in \Sigma $. The subsets appearing in this construction are sketched in Figure \ref{fig:PsicPsi}.
\begin{figure}
\includestandalone{PsicPsi}
\caption{The sets $\ell_c$, $B_c \smallsetminus \ell_c $ and $V \smallsetminus \ell_\Sigma $ appearing in \ref{eq:isoPsicPsi}.}\label{fig:PsicPsi}
\end{figure}

Let us now apply this to perverse sheaves.

\begin{lemma}
Let $F$ be an object of $\Perv_\Sigma(\field_\V)$.
Recalling that $L$ denotes the local system $F|_{\V\setminus\Sigma}[-1]$, we have
\begin{align*}
\Psi_c(F) &\simeq H^0\Psi_c(F)=H^0\rsect_\rc(\ell^\times_c;F)=H^1_\rc(\ell^\times_c;L), \\
\Phi_c(F) &\simeq H^0\Phi_c(F)=H^0\rsect_\rc(\ell_c;F),\\
\Psi(F) &\simeq H^0\Psi(F)=H^2_\rc(\V\setminus\ell_\Sigma;L),
\end{align*}
and these are exact functors from $\Perv_\Sigma(\field_\V) $ to finite dimensional $\field $-vector spaces.\qed
\end{lemma}

\begin{lemma}\label{lem:fidelite}
Let $F,G\in \Perv_\Sigma(\field_\V)$ and let $\varphi,\psi:F\to G$ be two morphisms. Assume that $\varphi$ and $\psi$ coincide on $\V\setminus\Sigma$ and that $\Phi_c(\varphi)=\Phi_c(\psi)$ for every $c\in \Sigma$. Then $\varphi=\psi$.
\end{lemma}

\begin{proof}
The image $F' $ of $\varphi-\psi $ is a sub-perverse sheaf of $G $. By the first assumption, it is supported on $\Sigma $. By exactness of $\Phi_c $ for perverse sheaves, the second assumption gives $\Phi_c F'=0 $, hence $F'=0 $. Therefore, $\varphi-\psi=0 $.
\end{proof}

\begin{definition}\label{def:quiverF}
Let $F \in \Perv_\Sigma(\field_\V)$.
The \emph{quiver} of $F$ is the datum
\[
Q_\Sigma^{(\alpha,\beta)}(F) = \left(\Psi(F), \Phi_c(F), u_c, v_c
\right)_{c\in\Sigma},
\]
with
\[
u_c\defeq u_{cc}\circ \text{\ref{eq:isoPsicPsi}}^{-1},\quad v_c\defeq
\text{\ref{eq:isoPsicPsi}} \circ v_{cc}.
\]
\end{definition}

Note that the choice of $\alpha$ and $\beta$ induces an orientation of $\V\simeq\R^2$ by the ordered basis given by $\alpha$ and a vector tangent at time zero to the rotation of $\ell^\times_c$ around $c$ in the direction of $B_c^<$. (The orientation class depends on the sign of $\Im\langle\alpha,\beta\rangle$.)

This implies a canonical isomorphism $\pi_1(B_c\setminus\{c\})\simeq\Z$, independent on the choice of a base point. We fix a base point $b_c\in\ell^\times_c$, and we denote by $T_{cc}^{(j)}$ the monodromy of~$H^j\opb{i_{b_c}}F$ induced by $1\in\Z\simeq\pi_1(B_c\setminus\{c\})$.

Let $(c,b_c)$ denote the open interval in $\ell^\times_c$. The exact sequence
\[
0\to\field_{(c,b_c)}\to \field_{(c,b_c]}\to\field_{\{b_c\}}\to0,
\]
together with the inclusion morphism $\field_{(c,b_c)}\to\field_{\ell^\times_c}$, induce an isomorphism $\chi_c \colon \opb{i_{b_c}}F[-1]\isoto\Psi_c(F)$, according to Lemma \ref{lem:vanishingsemiopen}.

\begin{lemma}\label{lem:PsiT}
Let $F\in \BDC_\Cc(\field_\V,\Sigma)$.
The isomorphism $$\chi_c \colon \opb{i_{b_c}}F[-1]\isoto\Psi_c(F)$$ intertwines $T_{cc}^{(j)}$ and $H^j(T_{cc})$ on the $j$-th cohomology sheaves.
\end{lemma}

Before proving the lemma, let us introduce the real blow-up $\varpi_\Sigma\colon\widetilde \V_\Sigma\to \V$ of $\V$ at all points $c\in\Sigma$, and let us set $S_c=\opb{\varpi_\Sigma}(c)$, $S_\Sigma=\bigcup_{c\in\Sigma}S_c=\opb{\varpi_\Sigma}(\Sigma)$. We have a commutative diagram with cartesian square:
\[
\xymatrix{
& \widetilde \V_\Sigma \ar[d]^{\varpi_\Sigma} & S_\Sigma \ar[l] \ar[d] \\
\V\setminus\Sigma \ar@{ >->}[r]^-{j_\Sigma} \ar@{ >->}[ur]^-{\tilde\jmath_\Sigma} & \V &
\Sigma \ar[l]_{i_\Sigma}.
}
\]
We write for short $\widetilde\V_c$ and $\varpi_c$ instead of $\widetilde\V_{\{c\}}$ and $\varpi_{\{c\}}$.

\begin{notation}\mbox{}\label{nota:realblowup}
\begin{enumerate}
\item
We identify $\ell_c^\times$ with $\varpi_\Sigma^{-1} \ell^\times_c$, and we denote by $\widetilde\ell_{c}$ its closure in $\widetilde \V_\Sigma$, with $\widetilde c\defeq\partial\widetilde\ell_{c}=\widetilde\ell_{c}\cap S_c$. We also set $\widetilde\Sigma=\bigcup_{c\in\Sigma}\{\widetilde c\}$ and $\widetilde B_c=\opb{\varpi_\Sigma}B_c$.
\item
For an object $F\in\BDC_\Cc(\field_\V,\Sigma)$, we will set $\widetilde F \defeq \derr\tilde\jmath_{\Sigma\,*} \opb{j_\Sigma} F$. If $F$ is in $\Perv_\Sigma(\field_\V)$, we have $\widetilde F=\wtL[1]$, if $\wtL$ denotes the locally constant sheaf $\widetilde\jmath_{\Sigma\,*}L$.
\end{enumerate}
\end{notation}

\begin{proof}[Proof of Lemma \ref{lem:PsiT}]
We will prove the similar property for the isomorphism $\opb{i_{\widetilde c}} \widetilde F[-1]\isoto\Psi_c(F)$ obtained in a way similar to $\chi_c$. The composition of the morphisms
\[
\field_{\ell^\times_c} \to \field_{\ell_c} \to \field_{B_c\setminus\ell_c}[1],
\]
together with \eqref{eq:isoPsicPsic}, gives rise to $v_{cc} u_{cc}$. Due to the morphism of distinguished triangles
\[
\xymatrix@R=.5cm{
\field_{B_c\setminus \{c\}}\ar[d]\ar[r]&\field_{\ell^\times_c}\ar[r]^-\varphi\ar[d]&\field_{B_c\setminus\ell_c}[1]\ar[r]^-{+1}\ar@{=}[d]&\\
\field_{B_c}\ar[r]&\field_{\ell_c}\ar[r]&\field_{B_c\setminus\ell_c}[1]\ar[r]^-{+1}&
}
\]
we identify it with the morphism $\varphi$. We now consider the morphism of distinguished triangles
\[
\xymatrix@R=.5cm{
\field_{\{\widetilde c\}}[-1]\ar[r]\ar[d]^{\widetilde\varphi}&\field_{\ell^\times_c}\ar[d]^\varphi\ar[r]&\field_{\widetilde\ell_c}\ar[d]\ar[r]^-{+1}&\\
\field_{S_c\setminus \{\widetilde c\}}\ar[r]&\field_{B_c\setminus\ell_c}[1]\ar[r]&\field_{\widetilde B_c\setminus\widetilde\ell_{c}}[1]\ar[r]^-{+1}&
}
\]
By Lemma \ref{lem:vanishingsemiopen}, the upper left horizontal morphism induces an identification $\opb{i_{\widetilde c}} \widetilde F[-1]\isoto\Psi_c(F)$ and, with respect to this identification, $v_{cc} u_{cc}$ is identified with the natural morphism $\opb{i_{\widetilde c}} \widetilde F[-1]\to \rsect_\rc(S_c\setminus \{\widetilde c\};\widetilde F)$ induced by $\widetilde\varphi$, whose cone is $\rsect_\rc(S_c;\widetilde F)$. By the same argument, the identification \eqref{eq:isoPsicPsic} is the composed isomorphism $\opb{i_{\widetilde c}} \widetilde F[-1]\to\rsect_\rc(S_c^>;\widetilde F)\to \rsect_\rc(S_c\setminus \{\widetilde c\};\widetilde F)$. Here $S_c^>\subset S_c$ denotes the open half circle of directions in $B_c^>$.

In order to prove the lemma, we can now assume that $F\in\Perv_\Sigma(\field_\V)$. Let us identify (with obvious notation) $H^1_\rc(S_c\setminus \{\widetilde c\};\widetilde L)$ with the cokernel of the restriction morphism $H^0([\widetilde c,\widetilde c-2\pi];\widetilde L)\to \widetilde L_{\widetilde c}\oplus\widetilde L_{\widetilde c-2\pi}$ that we represent as $\psi:\widetilde L_{\widetilde c}\to\widetilde L_{\widetilde c}\oplus\widetilde L_{\widetilde c}$ given by $x\mapsto(x,T_{cc}^{(0)-1}x)$, and let us set $\eta:\widetilde L_{\widetilde c}\oplus\widetilde L_{\widetilde c}\to\coker\psi$, that we identify with $\widetilde L_{\widetilde c}\oplus\widetilde L_{\widetilde c} \to \widetilde L_{\widetilde c} $ defined by $(x,x') \mapsto x-T_{cc}^{(0)}(x') $.
On the one hand, the identification \eqref{eq:isoPsicPsic} is the composition $z\mapsto(z,0)\mapsto\eta(z,0)$ and, on the other hand, the morphism (induced by) $\widetilde\varphi$ is the composition $y\mapsto(y,y)\mapsto\eta(y,y)$. It follows that $z=y-T_{cc}^{(0)} y$ holds in $\coker(\psi) $.
\end{proof}

\subsection{From quivers to perverse sheaves}\label{sec:Beilinson}
Our aim in this section is to recall that the functor $Q_\Sigma^{(\alpha,\beta)}$ is an equivalence of categories, by using Beilinson's maximal extension. We continue fixing $\alpha,\beta$ as in Convention \ref{con:alphabeta} and using Notation \ref{nota:realblowup}.

Let us set (see Figure~\ref{fig:Blow} for $\Sigma=\{c\}$)
\[
\widetilde \V_\Sigma^\Xi \defeq \widetilde \V_\Sigma \setminus \widetilde\Sigma.
\]
\begin{figure}
\includestandalone{Blow-cs}
\caption{The maps $\widetilde\V_c^\Xi \rightarrowtail \widetilde\V_c \to[\varpi_c] \V $.}
\label{fig:Blow}
\end{figure}

\begin{notation}\mbox{}\label{nota:Beilinsonblowup}
Consider the following subsets of $\widetilde\V^\Xi_\Sigma$, some of which will be of use in Section~\ref{sec:Stokes} (see Figure \ref{fig:PhipmBeilinson})
\begin{enumerate}
\item
We set $B_{\widetilde c} \defeq \varpi_\Sigma^{-1}(B_c) \smallsetminus \{ \widetilde c \}$, $B^\geq_{\widetilde c} \defeq \varpi_\Sigma^{-1}(B_c^\geq) \smallsetminus \{ \widetilde c \}$, $B^<_{\widetilde c} \defeq \varpi_\Sigma^{-1}(B_c^<) \smallsetminus \{ \widetilde c \}$.
We also set $S_{\widetilde c}\defeq S_c \smallsetminus \{ \widetilde c \} $ and $S^\geq_{\widetilde c} \defeq B^\geq_{\widetilde c} \cap S_{\widetilde c}$, $S^<_{\widetilde c} \defeq B^<_{\widetilde c} \cap S_{\widetilde c}$.
\item
We set $\ell^\pm_{\widetilde c} \defeq \varpi_\Sigma^{-1}(\ell_c(\pm\alpha)) \smallsetminus \{ \widetilde c \}$, so that $\ell^\pm_{\widetilde c} = \ell_c^\times(\pm\alpha) \cup S_{\widetilde c} $. Note that this is a disjoint union in case of $\ell^+_{\widetilde c} $.
We also write for short $\ell_{\widetilde c}$ instead of $\ell^+_{\widetilde c}$.
\end{enumerate}
\end{notation}

\begin{definition}\label{def:BeilMax}
Beilinson's maximal extension (with respect to $\alpha$) is the functor
\[
\Xi_\Sigma\colon \BDC_\Cc(\field_\V,\Sigma) \to \BDC_\Cc(\field_\V,\Sigma)
\]
defined by
\[
\Xi_\Sigma(F) \defeq \reim{{\varpi_\Sigma}}(\field_{\widetilde
\V_\Sigma^\Xi}\tens\widetilde F).
\]
\end{definition}
Note that
\begin{equation}\label{eq:XiXiF}
\Xi_\Sigma(F) \simeq \Xi_\Sigma(\field_{\V\setminus\Sigma}\tens F) \quad\text{and}\quad\Xi_\Sigma(F)|_{\V\setminus\Sigma} \simeq F|_{\V\setminus\Sigma}.
\end{equation}

\begin{prop}\label{prop:BeiMax}
Let $F$ be an object of $\BDC_\Cc(\field_\V,\Sigma)$. We have functorial (in $F$) distinguished triangles
\begin{align*}
&\reim{{j_\Sigma}}\opb{j_\Sigma}F \to[a'] \Xi_\Sigma(F) \to[b'] \bigoplus_{c\in\Sigma}\roim{{i_c}}\Psi_c(F) \to[+1],\\
&\bigoplus_{c\in\Sigma}\roim{{i_c}}\Psi_c(F) \to[b''] \Xi_\Sigma(F) \to[a''] \roim{{j_\Sigma}}\opb{j_\Sigma}F
\to[+1].
\end{align*}
The composition $a''\circ a'$ coincides with the canonical morphism
$\reim{{j_\Sigma}}\opb{j_\Sigma}F \to \roim{{j_\Sigma}}\opb{j_\Sigma}F$, and
one has $b' \circ b'' = \DSum_{c\in\Sigma}(1-T_{cc})$.

If moreover $F\in\Perv_\Sigma(\field_\V)$, then also $\Xi_\Sigma(F)\in\Perv_\Sigma(\field_\V)$, and the above triangles are short exact sequences in $\Perv_\Sigma(\field_\V)$:
\begin{align*}
&0 \to \reim{{j_\Sigma}}\opb{j_\Sigma}F \to[a'] \Xi_\Sigma(F) \to[b']
\DSum_{c\in\Sigma}\roim{{i_c}}\Psi_c(F) \to 0, \\
&0 \to \DSum_{c\in\Sigma}\roim{{i_c}}\Psi_c(F) \to[b''] \Xi_\Sigma(F) \to[a'']
\roim{{j_\Sigma}}\opb{j_\Sigma}F \to 0.
\end{align*}
\end{prop}

\begin{proof}
Applying the functor $\reim{{\varpi_\Sigma}}(\bullet\tens\widetilde F)$ to
the short exact sequences
\begin{align*}
& 0 \to \field_{\widetilde \V_\Sigma\setminus S_\Sigma} \to \field_{\widetilde \V_\Sigma^\Xi} \to \field_{S_{\widetilde\Sigma}} \to 0,\\
& 0 \to \field_{\widetilde \V_\Sigma^\Xi} \to \field_{\widetilde \V_\Sigma}
\to \field_{\widetilde \Sigma} \to 0,
\end{align*}
respectively, we get the desired distinguished triangles by using the identifications $\opb{i_{\widetilde c}}\widetilde F\isoto\Psi_c(F)$ and $\Psi_c(F)\isoto \rsect_\rc(S_{\widetilde c};\widetilde F)$ already used in the proof of Lemma \ref{lem:PsiT}.

If $F\in\Perv_\Sigma(\field_\V)$ then, since $\reim{{j_\Sigma}}\opb{j_\Sigma}F$, $\roim{{j_\Sigma}}\opb{j_\Sigma}F$ and $\roim{{i_c}}\Psi_c(F)$ are perverse, we obtain the perversity of $\Xi_\Sigma(F)$ as well as the short exact sequences in the proposition.

The composition $a''\circ a'$ is induced by the composition
\[
\field_{\widetilde \V_\Sigma\setminus S_\Sigma} \to \field_{\widetilde \V_\Sigma^\Xi} \to \field_{\widetilde \V_\Sigma},
\]
which is the morphism induced by the open inclusion $\widetilde \V_\Sigma\setminus
S_\Sigma \subset \widetilde \V_\Sigma$.
Hence the first assertion is clear.

The composition $b'\circ b''$ is induced by the composition
\[
\field_{\{\widetilde c\}}[-1] \to \field_{\widetilde \V_c^\Xi} \to
\field_{S_{\widetilde c}},
\]
and coincides with the first morphism of the triangle
\[
\field_{\{\widetilde c\}}[-1] \to \field_{S_{\widetilde c}} \to \field_{S_c}\overset{+1}\longrightarrow,
\]
which was denoted by $\widetilde\varphi$ in the proof of Lemma \ref{lem:PsiT}. The latter gives the proof of the desired assertion for $b'\circ b''$.
\end{proof}

\begin{corollary}\label{cor:BeiQuiv}
Let $F$ be an object of $\BDC_\Cc(\field_\V,\Sigma)$ and $c\in\Sigma$. There are natural isomorphisms
\[
\Psi_c(\Xi_c(F)) \simeq \Psi_c(F) \quad\text{and}\quad
\Phi_c(\Xi_c(F)) \simeq \Psi_c(F) \dsum \Psi_c(F).
\]
If moreover $F\in\Perv_\Sigma(\field_\V)$ then, with respect to these isomorphisms, $u_{cc}$ and $v_{cc}$ for $\Xi_c(F)$ are respectively given by
\[
\Psi_c(F) \to[{\left(\begin{smallmatrix}1\\ \\0\end{smallmatrix}\right)}]{}
\begin{matrix}\Psi_c(F)\\ \dsum \\\Psi_c(F)\end{matrix} \qquad\text{and}\qquad
\begin{matrix}\Psi_c(F)\\ \dsum \\\Psi_c(F)\end{matrix}
\to[{\left(\begin{smallmatrix}1-T_{cc}&&-1\end{smallmatrix}\right)}] \Psi_c(F).
\]
\end{corollary}

\begin{proof}
We can assume that $\Sigma=\{c\}$. Applying $\Psi_c$ to the distinguished triangles of Proposition \ref{prop:BeiMax}, we find that $\Psi_c(a'):\Psi_c(F)\isoto\Psi_c(\Xi_c(F))$ and $\Psi_c(a''):\Psi_c(\Xi_c(F))\isoto\Psi_c(F)$ are isomorphisms which are quasi-inverse one to the other.

Since $\reim{{j_c}}\opb{j_c}F$, $\Xi_c(F)$ and $\roim{{j_c}}\opb{j_c}F$ are respectively the push-forward by the proper map $\varpi_c$ of $\field_{\widetilde\V_c\setminus S_c}\otimes\widetilde F$, $\field_{\widetilde\V^\Xi_c}\otimes\widetilde F$ and $\widetilde F$, we can compute their vanishing cycles $\Phi_c$ by applying $\rsect_\rc(\widetilde\V_c,\field_{\varpi_c^{-1}(\ell_c)}\otimes\bullet)$ to the latter.
By noticing that $\varpi_c^{-1}(\ell_c)\cap\widetilde\V_c^\Xi = \ell_{\widetilde c} = \ell_c^\times\sqcup S_{\widetilde c}$, we get
\[
\Phi_c(\Xi_c(F))=\rsect_\rc(\widetilde\V_c,\field_{\ell_{\widetilde c}}\otimes\widetilde F)=\rsect_\rc(\widetilde\V_c,\field_{\ell_c^\times}\otimes\widetilde F)\oplus\rsect_\rc(\widetilde\V_c,\field_{S_{\widetilde c}}\otimes\widetilde F).
\]
that we recognize, due to \eqref{eq:isoPsicPsic}, to be isomorphic to $\Psi_c(F)\oplus\Psi_c(F)$.

Let us apply $\Phi_c$ to the distinguished triangles of Proposition \ref{prop:BeiMax}. We have \hbox{$u_{cc}:\Psi_c(\reim{{j_c}}\opb{j_c}F)\isoto\Phi_c(\reim{{j_c}}\opb{j_c}F)$} and $v_{cc}:\Phi_c(\roim{{j_c}}\opb{j_c}F)\isoto\Psi_c(\roim{{j_c}}\opb{j_c}F)$. We conclude that
\[
\Phi_c(a''\circ a')=v_{cc}u_{cc}=1-T_{cc}:\Psi_c(F)\to\Psi_c(F).
\]
Moreover, $\Phi_c(a')$ (resp.\ $\Phi_c(a'')$) is identified with $u_{cc}$ (resp.\ $v_{cc}$) for $\Xi_c(F)$. With respect to the above decomposition, the map $\Phi_c(a')$ is given by $x\mapsto(x,0)$ and the map $\Phi_c(b')$ is the second projection.

Let us now assume that $F\in\Perv_\Sigma(\field_\V)$, so that we can use linear algebra. The morphism $\Phi_c(b''):\Psi_c(F)\to\Psi_c(F)\oplus\Psi_c(F)$ is induced by $\field_{\{\widetilde c\}}[-1]\to\field_{\ell_c^\times}\oplus\field_{S_{\widetilde c}}$, and we can write it as $y\mapsto(y,(1-T_{cc})y)$ according to the proof of Lemma \ref{lem:PsiT} for~$\widetilde\varphi$.

Finally, the morphism $\Phi_c(a''):\Psi_c(F)\oplus\Psi_c(F)\to\Psi_c(F)$ is such that $\Phi_c(a'')(x,0)=(1-T_{cc})x$ and $\ker\Phi_c(a'')=\im\Phi_c(b'')$, hence \hbox{$\Phi_c(a'')(x,y)=(1-T_{cc})x-y$}.
\end{proof}

Let us consider the functor which associates to $F$ in $\Perv_\Sigma(\field_{\V})$ the double complex
\begin{equation}\label{eq:GF}
G_2(F)=\left\{\begin{array}{c}
\xymatrix@C=1.5cm{
\bigoplus_{c\in\Sigma}\roim{{i_c}}\Psi_c(F) \ar[r]^-{b''} \ar[d]_{\bigoplus u_{cc}} & \Xi_\Sigma(F) \ar[d]_{-b'} \\
\bigoplus_{c\in\Sigma}\roim{{i_c}}\Phi_c(F) \ar[r]^-{\bigoplus v_{cc}} & \bigoplus_{c\in\Sigma}\roim{{i_c}}\Psi_c(F).
}
\end{array}\right.
\end{equation}
We denote by $G(F)$ the associated total complex shifted by one, so that the first term is in degree $-1$ (we will omit the change of signs in the differentials due to the shift).

\begin{corollary}\label{cor:BeiReconstr}
If $F$ is an object of $\Perv_\Sigma(\field_\V)$, then so is $G(F)$, and the functor $G:\Perv_\Sigma(\field_\V)\to\Perv_\Sigma(\field_\V)$ is an equivalence of categories which satisfies $G(F)\simeq F$, that is,
\begin{multline*}
F \simeq H^1\Bigg(
\DSum_{c\in\Sigma}\roim{{i_c}}\Psi_c(F)
\to[{\biggl(\begin{smallmatrix}b''\\[5pt]\DSum\limits_{c\in\Sigma}u_c\end{smallmatrix}
\biggr)}]
\begin{matrix}
\Xi_\Sigma(F) \\ \dsum \\ \DSum\limits_{c\in\Sigma}\roim{{i_c}}\Phi_c(F)
\end{matrix}\\
\to[{\left(\begin{smallmatrix}-b'&\DSum\limits_{c\in\Sigma}v_c\end{smallmatrix}
\right)}] \DSum_{c\in\Sigma}\roim{{i_c}}\Psi_c(F) \Bigg).
\end{multline*}
\end{corollary}

\begin{proof}
We obviously have $\opb{j_\Sigma}G(F)=\opb{j_\Sigma}F$. For the perversity of $G(F)$, it is then enough to check that $\Psi_c(G(F))$ and $\Phi_c(G(F))$ have cohomology in degree zero only, for any $c\in\Sigma$. This is clear for $\Psi_c(G(F))=\Psi_c(F)$. By the computation of Corollary \ref{cor:BeiQuiv}, $\Phi_c(G_2(F))$ is isomorphic to the double complex
\[
\xymatrix@C=1.5cm{
\Psi_c(F) \ar[r]^-{(1,1-T_{cc})} \ar[d]_{u_{cc}} & \Psi_c(F)\oplus\Psi_c(F) \ar[d]_{-p_2} \\
\Phi_c(F) \ar[r]^-{v_{cc}} & \Psi_c(F),
}
\]
and the claim follows, since the upper horizontal arrow is injective and the right vertical arrow is onto. Note also that $u_{cc}$ for $G_2(F)$ is induced by the inclusion of the first summand in $\Psi_c(F)\oplus\Psi_c(F)$.

We now prove the essential surjectivity of $G$ by showing $G(F)\simeq F$. If $F$ is supported on $\Sigma$, then $G(F)=\bigoplus_{c\in\Sigma}\roim{{i_c}}\Phi_c(F)=F$. If $F$ is equal to $\reim{{j_\Sigma}}\opb{j_\Sigma}F$, then~$u_{cc}$ is an isomorphism for every $c$. Let us prove that \emph{$u_{cc}$ for $G(F)$ is an isomorphism}. We can assume $u_{cc}=\id$ for $F$, so that $v_{cc}=1-T_{cc}$ for $F$. Then $u_{cc}$ for $G(F)$ is induced by the vertical morphism of complexes
\begin{equation}\label{eq:Gj}
\begin{array}{c}
\xymatrix@C=1.8cm{
0\ar[r]\ar[d]&0\oplus\Psi_c(F)\oplus0\ar[r]\ar[d]^{(0,1,0)}&0\ar[d]\\
\Psi_c(F)\ar[r]^-{(1,1,1-T_{cc})}&\Psi_c(F)\oplus\Psi_c(F)\oplus\Psi_c(F)\ar[r]^-{(1-T_{cc})+0-1}&\Psi_c(F).
}
\end{array}
\end{equation}
In such a way, the claim is obvious, and implies that $G(\reim{{j_\Sigma}}\opb{j_\Sigma}F)=\reim{{j_\Sigma}}\opb{j_\Sigma}F$.

Let us denote by $\eta$ the morphism $\reim{{j_\Sigma}}\opb{j_\Sigma}F\to F$. The cone of $G(\eta)$ in $\BDC_\Cc(\field_\V,\Sigma)$ is isomorphic to the complex
\[
\bigoplus_{c\in\Sigma}\Psi(F)\to[\bigoplus_cu_{cc}]\bigoplus_{c\in\Sigma}\Phi(F),
\]
that is, to $\opb{i_\Sigma}F$. We obtain in this way a morphism $\eta':\opb{i_\Sigma}F[-1]\to\reim{{j_\Sigma}}\opb{j_\Sigma}F$. We claim that this morphism is equal to the natural morphism, which will prove $G(F)\simeq F$. This is obvious away from $\Sigma$, so it suffices, according to Lemma \ref{lem:fidelite}, to show that $\Phi_c(\eta')$ is $\Phi_c$ of the natural morphism for every $c\in \Sigma$. In other words, it suffices to show that $\Phi_c(G(\eta))=\Phi_c(\eta)$ for every $c\in \Sigma$. This is done by using a presentation like in the lower line of \eqref{eq:Gj}.

The previous result also implies the fullness. For the faithfulness we note that, if a morphism $\eta:F\to F'$ is such that $G(\eta)=0$, then $\eta_{|\V\setminus\Sigma}=0$, and $\Phi_c(G(\eta))=0$ for every $c\in \Sigma$, which implies $\Phi_c(\eta)=0$ by an argument similar to the above one.
\end{proof}

Note that the datum of $F_{|\V\setminus\Sigma}$ is enough to recover $\Xi_\Sigma(F)$, and that $\Xi_\Sigma$ induces a faithful functor $\Perv(\field_{\V\setminus\Sigma})\to\Perv_\Sigma(\field_\V)$.
Then, by looking at local systems on $\V\setminus\Sigma$ as monodromy representations, and by using the above equivalence $G$, we get

\begin{corollary}
The functor $Q_\Sigma^{(\alpha,\beta)}$ is an equivalence of categories.
\end{corollary}

\begin{corollary}\label{cor:Bei}
Let $A,B\colon \Perv_\Sigma(\C_\V) \to \Mod(\field)$ be two exact functors, and let $\alpha,\beta\colon A\to B$ be two morphisms of functors. Assume that $\alpha$ and $\beta$ coincide when restricted to the subcategory consisting of perverse sheaves supported on $\Sigma$, and to $\Xi(\Perv(\field_{\V\setminus\Sigma}))$. Then $\alpha=\beta$.
\end{corollary}

\begin{proof}
The assumption implies $\alpha\circ G=\beta\circ G$. We conclude by using that $G$ is an equivalence.
\end{proof}

\section{Stokes phenomena}\label{sec:Stokes}

Let $F\in\Perv_\Sigma(\field_\V)$, and consider its enhanced Fourier-Sato transforms
\begin{align}
\enhK &\defeq \varepsilon(F)^\curlywedge \in\BECRc[\field] {\W_\infty}, \notag \\
K &\defeq e(F)^\curlywedge \simeq \field_{\W_\infty}^\enh \ctens \quot\enhK \in\BECRc {\W_\infty}.\label{eq:defK}
\end{align}
We will describe here the exponential components and the Stokes multipliers at infinity of $K$, along the lines described in \cite[\S9.8]{DK16} (see also \cite[\S8]{Kas16}).
Recall that if $F=\sol(\shm)$ for $\shm$ a regular holonomic $\D_{\V_\infty}$-module, then $K\simeq\solE(\shm^\wedge)$.

\subsection{Exponential form on sectors}

Fix a pair $(\alpha, \beta) $ as in Convention~\ref{con:alphabeta}.
Set
\begin{align*}
H_{\pm\alpha} &\defeq \{ w\in\W\setminus\{0\} \semicolon \pm\Re\langle \alpha,w \rangle \geq 0\}, \\
h_{\pm\beta} &\defeq \pm\R_{>0}\beta \subset \W,
\end{align*}
so that $\W\setminus\{0\} = H_\alpha \cup H_{-\alpha}$ and $H_\alpha \cap H_{-\alpha} = h_\beta \cup h_{-\beta}$. We consider $H_{\pm\alpha}$ as closed sectors centered at infinity.

Recall from Lemma \ref{lem:F'TamFou} that
\[
\enhK \simeq \reim{\tilde q}(\field_{\{t+\Re zw\geq 0\}}\tens\opb{\overline p} F),
\]
where we considered the projections
\[
\W \times\R \from[\tilde q] \V\times\W \times\R \to[\overline p] \V.
\]
\begin{prop}\label{prop:expsectors}
There are natural isomorphisms in $\BECRc[\field] {\W_\infty}$
$$
s_{\pm\alpha} \colon \Erest{H_{\pm\alpha}} \enhK \isoto
\Erest{H_{\pm\alpha}}{\DSum_{c\in\Sigma}
\bigl(\Phi^{(\pm\alpha,\beta)}_c (F) \tens E^{c w} \bigr)}
$$
(using Notation \ref{notation:Evi} and Definition \ref{def:vannearcycles}).
In particular, the exponents of $e(F)^\curlywedge$ at $\infty$ are the linear functions $c w$, for $c\in\Sigma$.
\end{prop}

\begin{proof}
Since the proofs are similar, let us only discuss $s_\alpha$.

(i) In order to better motivate our geometric constructions, let us start by computing the stalk $\enhK_{(w,t)}$
for $w\in H_\alpha$ and $t\in\R$.
One has
\[
\enhK_{(w,t)} \simeq\rsect_\rc(\V;\field_{Z_{(w,t)}}\tens F),
\]
where we set
\[
Z_{(w,t)} = \{z\in\V\semicolon t+\Re zw \geq 0\}.
\]
Note that $Z_{(w,t)}= -(t/|w|^2)\overline w + \{z\in\V \semicolon \Re z\geq 0\} \overline w$ is a closed half-space with inner conormal $\overline w$. Note also that $Z_{(w,t)}\supset Z_{(w,s)}$ for $s<t$.

Consider the pairwise disjoint half-lines $\ell_c(\alpha) = c+\R_{\geq 0}\alpha$, for $c \in \Sigma $. Since $w\in H_\alpha$, one has $\Re\alpha w\geq 0$. Hence $\ell_c(\alpha) \subset Z_{(w,t)}$ if and only if
$c \in Z_{(w,t)}$ if and only if $t+\Re c w\geq 0$.
Set
\begin{align*}
\Sigma_{(w,t)} &= \{ c \in \Sigma; t+ \Re cw \geq 0 \}
= \Sigma\cap Z_{(w,t)}, \\
\ell_{\Sigma_{(w,t)}}(\alpha) &= \Union_{c\in\Sigma_{(w,t)}} \ell_c(\alpha).
\end{align*}
Since $Z_{(w,t)} \setminus \ell_{\Sigma_{(w,t)}}(\alpha) $ is $\Sigma $-negligible (see Figure~\ref{fig:hyperplane}),
the short exact sequence
\begin{equation}
\label{eq:Hwtdt}
0 \to \field_{Z_{(w,t)}\setminus \ell_{\Sigma_{(w,t)}}(\alpha)}
\to \field_{Z_{(w,t)}} \to \DSum_{c\in\Sigma_{(w,t)}} \field_{\ell_c(\alpha)} \to 0
\end{equation}
induces an isomorphism
\begin{equation}\label{eq:Gwttriangle}
\enhK_{(w,t)} \simeq
\DSum_{c\in\Sigma_{(w,t)}} \rsect_\rc(\V; \field_{\ell_c(\alpha)} \tens F) = \DSum_{c\in\Sigma_{(w,t)}} \Phi^{(\alpha,\beta)}_c(F).
\end{equation}

\begin{figure}
\includestandalone{hyperplane}
\caption{\label{fig:hyperplane}Solid points represent elements of $\Sigma_{(w,t)}$, and crosses represent the other elements of $\Sigma$.}
\end{figure}

(ii) Let us now treat the global case. For $c\in\Sigma$, set
\begin{align*}
L_c(\alpha) &= \ell_c(\alpha) \times \{(w,t)\in\W\times\R\semicolon t+\Re c w\geq 0 \}, \\
L_\Sigma(\alpha) &= \Union_{c\in\Sigma} L_c(\alpha),
\end{align*}
which are closed subsets of $\V\times\W\times \R$.
The short exact sequence
\[
0 \to
\field_{\{t+\Re z w\geq 0\} \setminus L_\Sigma(\alpha)}
\to \field_{\{t+\Re z w\geq 0\}}
\to \DSum_{c\in\Sigma} \field_{L_c(\alpha)} \to 0
\]
induces a distinguished triangle
\begin{equation}
\label{eq:Gwttriangle2}
\opb\pi\field_{H_\alpha} \tens \enhK' \to
\opb\pi\field_{H_\alpha} \tens \enhK \to[\rho_\alpha]
\opb\pi\field_{H_\alpha} \tens \enhK'' \to[+1]
\end{equation}
with
\begin{align*}
\enhK' & = \reim{\tilde q}(\field_{\{t+\Re z w\geq 0\} \setminus
L_\Sigma(\alpha)} \tens \opb{\overline p} F), \\
\enhK'' & = \DSum_{c\in\Sigma} \reim{\tilde q}(\field_{L_c(\alpha)} \tens
\opb{\overline p} F).
\end{align*}
For $(w,t)\in H_\alpha$, the set
\[
\widetilde q^{-1}(w,t) \cap \bigl(\{t+\Re z w\geq 0\} \setminus
L_\Sigma(\alpha)\bigr)=Z_{(w,t)} \setminus \ell_{\Sigma_{(w,t)}}(\alpha)
\]
is $\Sigma $-negligible. It follows that $\enhK'_{(w,t)} \simeq 0$. Hence $\opb\pi\field_{H_\alpha} \tens \enhK' \simeq 0$,
and $\rho_\alpha$ is an isomorphism.
Consider the diagram with Cartesian squares
\[
\xymatrix{
\V\times\W\times\R \ar[d]_{\overline p} \ar@/^1.5pc/[rrr]^{\widetilde q} & L_c(\alpha) \ar@{ >->}[l] \ar[r]^-{q_2} \ar[d]^{q_1} & \{t+\Re c w\geq 0 \} \ar@{ >->}[r]_-j \ar[d] & \R\times\W \\
\V & \ell_c(\alpha) \ar[r] \ar@{ >->}[l] & \{pt\}.
}
\]
One has the isomorphisms
\begin{align}\label{eq:rhoalpha}
\opb\pi\field_{H_\alpha} \tens \enhK \notag
&\isoto[\rho_\alpha] \opb\pi\field_{H_\alpha} \tens \enhK''\\ \notag
&= \opb\pi\field_{H_\alpha} \tens
\left(\textstyle\DSum_{c\in\Sigma} \reim{\tilde q}(\field_{L_c(\alpha)} \tens
\opb{\overline p} F) \right) \\
&\simeq \opb\pi\field_{H_\alpha} \tens
\left(\textstyle\DSum_{c\in\Sigma} \reim j \reim{{q_2}} \opb{q_1} (F|_{\ell_c(\alpha)}) \right) \\ \notag
&\xleftarrow[(*)]{~\sim~} \opb\pi\field_{H_\alpha} \tens
\left(\textstyle\DSum_{c\in\Sigma} \rsect_\rc(\V;\field_{\ell_c(\alpha)} \tens F) \tens \field_{\{t+ \Re cw \geq 0 \}}
\right) \\ \notag
&= \opb\pi\field_{H_\alpha} \tens
\left(\textstyle\DSum_{c\in\Sigma} \Phi^{(\alpha,\beta)}_c(F) \tens E^{cw} \right),
\end{align}
where $(*)$ follows from \cite[Proposition 2.6.7]{KS90}.
\end{proof}

By convolution with $\field_{\W_\infty}^\enh $ we deduce the isomorphisms
\begin{equation}\label{eq:sprime}
\boldsymbol{s}_{\pm\alpha} \colon \Erest{H_{\pm\alpha}} K \isoto
\Erest{H_{\pm\alpha}}{\DSum_{c\in\Sigma}
\bigl(\Phi_c (F) \tens \enhE^{c w} \bigr)}
\end{equation}
with $\enhE^{cw} =\field_{\W_\infty}^\enh \ctens\quot E^{cw} $ and $K $ as in \eqref{eq:defK}. Similar to \cite[9.8]{DK16}, we have the following lemma.

\begin{lemma}\label{lem:isomH}
Let $(\Phi_c)_{c\in\Sigma}$ be $\field $-vector spaces and set $\Phi_\Sigma=\bigoplus_{c\in\Sigma}\Phi_c$. Let $c,d \in \Sigma$.
\begin{enumerate}[(i)]
\item\label{lem:isomH1}
Assume that $\Re\langle c,\beta \rangle > \Re\langle d,\beta \rangle$. Then there are natural isomorphisms:
\begin{multline*}
\Hom[{\BEC {\W_\infty}}] ( \Erest{h_{\beta}} \enhE^{cw} \tens \Phi_c, \Erest{h_{\beta}}\enhE^{dw} \tens \Phi_d) \\
\simeq\Hom[{\BEC[\field]{\W_\infty}}] ( \Erest{h_{\beta}} E^{cw} \tens \Phi_c, \Erest{h_{\beta}} E^{dw} \tens \Phi_d) \simeq
\Hom[\field] ( \Phi_c, \Phi_d).
\end{multline*}
\item \label{lem:isomH2}
If $\Re\langle c,\beta \rangle < \Re\langle d,\beta \rangle$, we have
\begin{multline*}
\Hom[{\BEC {\W_\infty}}] ( \Erest{h_{\beta}} \enhE^{cw} \tens \Phi_c, \Erest{h_{\beta}}\enhE^{dw} \tens \Phi_d) \\
\simeq\Hom[{\BEC[\field]{\W_\infty}}] ( \Erest{h_{\beta}} E^{cw} \tens \Phi_c, \Erest{h_{\beta}} E^{dw} \tens \Phi_d)=0.
\end{multline*}
\item\label{lem:isomH3}
Each of the sectors $H_{\pm \alpha} $ contains exactly one Stokes-line for each pair $c \neq d $ in $\Sigma $ and
\begin{multline*}
\Endo[{\BEC {\W_\infty}}] \Bigl( \Erest{H_{\alpha}} {\textstyle\DSum}_{c\in \Sigma} (\enhE^{cw} \tens \Phi_c)\Bigl)\\
\simeq\Endo[{\BEC[\field]{\W_\infty}}] \Bigl( \Erest{H_{\alpha}} {\textstyle\DSum}_{c\in \Sigma} (E^{cw} \tens \Phi_c)\Bigr)\simeq \mathfrak{t},
\end{multline*}
with $\mathfrak{t} \subset \Endo[\field] \Phi_\Sigma $ denoting the subspace of all block diagonal endomorphisms, for $\Phi_\Sigma = \dsum_c\Phi_c$.
\end{enumerate}
\noindent The statements \eqref{lem:isomH1} and \eqref{lem:isomH2} hold for $-\beta $ instead of $\beta $ in the same way.
\end{lemma}

By \eqref{lem:isomH3} the isomorphisms of Proposition \ref{prop:expsectors} and \eqref{eq:sprime} above are unique up to base-change by block-diagonal isomorphisms in $\mathfrak{t}$.

\subsection{Stokes multipliers} \label{sec:Stokesmulti}

We will now take up the point of view of Stokes multipliers as explained in the introduction.

As in the previous subsection, consider the covering $\W = H_\alpha\cup H_{-\alpha}$, with $H_\alpha \cap\nobreak H_{-\alpha} = h_\beta \cup h_{-\beta} $.
Let
\[
Q_\Sigma^{(\alpha,\beta)}(F) = \left(\Psi(F), \Phi_c^{(\alpha,\beta)}(F), u_c, v_c
\right)_{c\in\Sigma}
\]
be the quiver associated to $F $. Let us write
\[
\Psi:=\Psi(F) , \quad \Phi_c \defeq \Phi^{(\alpha,\beta)}_c(F) , \quad \Phi_c^- \defeq \Phi^{(-\alpha,\beta)}_c(F)
\]
There are two types of monodromy operators attached to the quiver:
\begin{align*}
\PhiT_c & \defeq 1- u_c v_c \in \mathrm{End}(\Phi_c), \\
\PsiT_c & \defeq 1- v_c u_c \in \mathrm{End}(\Psi).
\end{align*}

\begin{remark}\label{rem:locglob}
Note that by Definition \ref{def:quiverF}, the morphisms $u_c $ and $v_c $ include the isomorphism \ref{eq:isoPsicPsi} for each $c $ which identifies the nearby cycle $\Psi_c(F) $ at $c $ with the global one $\Psi(F) $. In particular, we observe that~$\PhiT_c $ equals the local monodromy $1-v_{cc}u_{cc} $ of $\Phi_c(F) $, whereas
\begin{equation}\label{eq:TlocT}
\PsiT_c := \text{\ref{eq:isoPsicPsi}} \circ \PsiT_{cc} \circ
\text{\ref{eq:isoPsicPsi}}^{-1}
\end{equation}
is conjugate by \ref{eq:isoPsicPsi} to the local monodromy $\PsiT_{cc}=1-v_{cc}u_{cc} $.

Similar considerations have to be kept in mind for the composition
\[
u_{c} v_{d}: \Phi_c(F) \to \Psi(F) \to \Phi_d(F)
\]
for $c, d \in \Sigma $, which will appear in the main result below and which reads as follows in terms of the local data $u_{cc} $ and $v_{cc} $:
\[
u_cv_d = u_{cc} \circ
\text{{\renewcommand\eqc{c}\ref{eq:isoPsicPsi}}}^{-1} \circ
\text{{\renewcommand\eqc{d}\ref{eq:isoPsicPsi}}} \circ
v_{dd}.
\]
\end{remark}

For the main result to follow, we will also consider the quiver $\bigl(\Phi_\Sigma,\ \Psi,\
V_\Sigma,\ U_\Sigma \bigr),
$
where $\Phi_\Sigma \defeq \DSum_{i=1}^n \Phi_{c_i}$, $U_\Sigma \defeq
{}^t(U_1,\dots,U_n)$, $V_\Sigma \defeq (V_1,\dots,V_n)$,
\begin{align*}
U_i &\defeq u_{c_i} \PsiT_{c_{i+1}} \PsiT_{c_{i+2}} \cdots \PsiT_{c_n} \, \\
V_i &\defeq v_{c_i}
\end{align*}
(which will turn out to be the quiver of the Fourier-Sato transform of $F $ at the origin with respect to the pair $(\beta,-\alpha) $, see Proposition \ref{pro:Foupervsmash}. Consequently, the monodromy isomorphism $1-U_\Sigma V_\Sigma $ of $\Phi_\Sigma $ around $0 $ is the one given by the induced orientation of $(\beta,-\alpha) $ and hence coincides with the monodromy of $\Phi_\Sigma $ around $\infty $ in the orientation given by $(\beta,\alpha) $).

By Proposition \ref{prop:expsectors}, there are natural isomorphisms
\[
\sigma'_{\pm\beta} \colon
\opb\pi\field_{h_{\pm \beta}}\tens\DSum_{c\in\Sigma} (\Phi_c \tens E^{cw}) \isoto
\opb\pi\field_{h_{\pm \beta}}\tens\DSum_{c\in\Sigma} (\Phi^-_c \tens E^{cw}),
\]
given by $\sigma'_{\pm\beta}:= \big(\opb\pi\field_{h_{\pm \beta}} \tens s_{-\alpha} \big) \circ \big(\opb\pi\field_{h_{\pm \beta}} \tens s_\alpha\big)^{-1}$.

On the other hand, we have a natural identification $\tau:\Phi_c^-\isoto\Phi_c$ obtained as follows. We~define~$B_c^\geq$ as in Notation \ref{nota:Bell}\eqref{nota:Bell2}. Since $B_c^\geq\setminus \ell_c(\pm\alpha)$ are $\Sigma $-negligible (see Figure~\ref{fig:Phipm}), the exact sequences
\begin{equation}\label{eq:Phipm}
0 \to \field_{B_c^\geq \setminus \ell_c(\pm\alpha)} \to \field_{B_c^\geq} \to \field_{\ell_c(\pm\alpha)} \to 0
\end{equation}
induce isomorphisms
\[
\tau_\pm\colon \Phi_c^{(\pm\alpha,\beta)} (F) \isoto
\rsect_\rc(\V;\field_{B_c^\geq(\beta)} \tens F),
\]
from which we define
\begin{equation}\label{eq:Phipmiso}
\tau\defeq\tau_+^{-1}\circ\tau_- \colon \Phi_c^{(-\alpha,\beta)} (F)\isoto
\Phi_c^{(\alpha,\beta)} (F).
\end{equation}

\begin{figure}
\includestandalone{Phipm}
\caption{The sets inducing one of the exact sequences \eqref{eq:Phipm}. For the other, replace $\ell_c(\alpha) $ with $\ell_c(-\alpha) $.}\label{fig:Phipm}
\end{figure}

Setting $\sigma_{\pm\beta} = \tau\circ\sigma'_{\pm\beta}$, we obtain the automorphisms
\[
\sigma_{\pm\beta} \in \mathrm{End}(\opb\pi\field_{h_{\pm \beta}}\tens\DSum_{c\in\Sigma} (\Phi_c \tens E^{cw})).
\]

If $\mathrm{End}^{\pm}(\Phi_\Sigma) $ denotes the subspace of upper/lower block diagonal matrices, the ordering $c_1 <_\beta \ldots <_\beta c_n $ induces, according to Lemma \ref{lem:isomH}, identifications
\begin{equation}
\label{eq:epmebta}
e_{\pm\beta}\colon \mathrm{End}\left(\opb\pi\field_{h_{\pm \beta}} \tens \textstyle\DSum_{c \in \Sigma} (\Phi_c \otimes E^{cw})\right) \isoto \mathrm{End}^\pm(\Phi_\Sigma),
\end{equation}
and, using notation as in \eqref{eq:sprime}, similar identifications $\boldsymbol{e}_{\pm\beta}$. The {\em Stokes multipliers} are then given by
\[
S_{\pm\beta} := e_{\pm\beta} (\sigma_{\pm\beta})=\boldsymbol{e}_{\pm\beta} (\boldsymbol{\sigma}_{\pm\beta}) \in \mathrm{End}^\pm(\Phi_\Sigma),
\]
an upper/lower block triangular matrix, respectively.

\begin{theorem}\label{thm:Stokesmultipliers}
Let $F\in\Perv_\Sigma(\C_\V) $ be a perverse sheaf with quiver $Q_\Sigma^{(\alpha,\beta)}(F)=(\Psi, \Phi_{c},u_{c}, v_{c})_{c \in \Sigma} $. Write for short $u_i=u_{c_i}$, $v_i=v_{c_i}$ and $\PhiT_i = 1-u_i v_i$.

Let $K\in\BEC {\W_\infty}$ be the enhanced Fourier-Sato transform of $F $. Then the Stokes multipliers $S_{\pm\beta} $ of $K$ at $\infty$ are given by the block triangular matrices
\begin{align}
\label{eq:S+}
S_\beta & =
\begin{pmatrix}
1 & u_1 v_2 & u_1v_3 & \cdots & u_1v_n \\
& 1 & u_2 v_3 & \cdots & u_2 v_n \\
& & \ddots & & \vdots \\
&&& & 1
\end{pmatrix},
\end{align}
and
\begin{align}
\label{eq:S-}
S_{-\beta} &=
\begin{pmatrix}
\PhiT_1\\
-u_2v_1 & \PhiT_2 \\
-u_3 v_1 & -u_3 v_2 & \ddots \\
\vdots & \vdots && \ddots\\
-u_n v_1 & -u_n v_2 & \cdots & -u_n v_{n-1} & \PhiT_n
\end{pmatrix}.
\end{align}
In particular, one has
$$
\opb{S_\beta} \, S_{-\beta} = \PhiT_\Sigma,
$$
where $\PhiT_\Sigma \defeq 1- U_\Sigma V_\Sigma$ is the monodromy of $\Phi_\Sigma$ around $\infty $.
\end{theorem}

\begin{proof}
The equalities \eqref{eq:S+} and \eqref{eq:S-} are equalities
between natural endo-transformations of the exact functor $\Phi_\Sigma \colon \mathrm{Perv}_\Sigma(\field_\V) \to \Mod(\field)$.
Such equalities hold if $F$ is supported on $\Sigma$, where they read $S_{\pm\beta} = \id$.
By Corollary~\ref{cor:Bei}, we are then left to consider the case where
$F=\Xi_\Sigma(\reim{{j_\Sigma}}L[1])$ is the Beilinson maximal extension of a local system $L$ on $\V\setminus\Sigma$. We do it in Subsection \ref{sec:Beilmaxext}.

Finally, in order to prove the equality $\opb{S_\beta} \, S_{-\beta} = \PhiT_\Sigma$,
one checks that
\[
S_\beta^{-1} =
\begin{pmatrix}
1 & -u_1v_2 & -u_1 \PsiT_2 v_3 & \dots & -u_1\PsiT_2\PsiT_3\cdots \PsiT_{n-1}v_n \\
& 1 & -u_2v_3 & & -u_2\PsiT_3\cdots \PsiT_{n-1}v_n \\
& & 1 & & -u_3\PsiT_4\cdots \PsiT_{n-1}v_n \\
& & & \ddots & \vdots \\
& & & & -u_{n-1}v_n \\
& & & & 1
\end{pmatrix}
\]
i.e., for $i<j$ one has
$(S_\beta)^{-1}_{ij} =
-u_i \PsiT_{i+1}\PsiT_{i+2}\cdots \PsiT_{j-1}v_j$.
Hence
\begin{multline*}
\opb{S_{\beta}} S_{-\beta}\\
 = 1-
\begin{pmatrix}
u_1\PsiT_2 \PsiT_3 \cdots \PsiT_nv_1 & u_1\PsiT_2\PsiT_3 \cdots \PsiT_n v_2 & \dots & u_1\PsiT_2\PsiT_3\cdots \PsiT_nv_n \\
u_2\PsiT_3 \cdots \PsiT_n v_1 & u_2\PsiT_3 \cdots \PsiT_n v_2 & \dots & u_2\PsiT_3 \cdots \PsiT_nv_n \\
\vdots & \vdots & \ddots & \vdots \\
u_{n-1}T_nv_1 &u_{n-1}T_nv_2 & \dots & u_{n-1}T_nv_n \\
u_nv_1 & u_nv_2 & \dots & u_nv_n
\end{pmatrix} \\
= 1 - U_\Sigma V_\Sigma.\qedhere
\end{multline*}
\end{proof}

\subsection{Proof of the main result for Beilinson's maximal extension}\label{sec:Beilmaxext}

We now prove Theorem \ref{thm:Stokesmultipliers} in the case where
\[
F = \Xi_\Sigma(\reim{{j_\Sigma}}L[1])
\]
is a Beilinson maximal extension for $j_\Sigma\colon \V\setminus\Sigma \to \V$ the open embedding and $L$ a local system on $\V\setminus\Sigma$. Due to \eqref{eq:XiXiF}, we then also have $F=\Xi_\Sigma(F) $.

In the terminology of Notation \ref{nota:realblowup},
\[
F = \reim{{\varpi_\Sigma}}(\field_{\widetilde
\V_\Sigma^\Xi}\tens\wtL[1]),
\]
for $\wtL \defeq \derr\tilde\jmath_{\Sigma\,*} L$. Recall Notation~\ref{nota:Beilinsonblowup}.

We want to prove that the isomorphism
\begin{equation}\label{eq:SBeil}
S_{\pm\beta}: \DSum_{c\in\Sigma}\Phi^{(\alpha,\beta)}_c(F) \isoto
\DSum_{c\in\Sigma}\Phi^{(\alpha,\beta)}_c(F)
\end{equation}
has the form given in \eqref{eq:S+} or \eqref{eq:S-} respectively.

Now, $S_{\pm \beta} $ is induced by the restrictions of $\rho_{\pm\alpha} $ \eqref{eq:rhoalpha} to $h_{\pm\beta} $:
\begin{equation}\label{eq:Spmbeta}
\begin{array}{c}
\xymatrix{
&\opb\pi\field_{h_{\pm\beta}} \tens
\left(\DSum_{c\in\Sigma} \Phi^{(\alpha,\beta)}_c(F) \tens E^{cw} \right) \ar[dd] \\
\opb\pi\field_{h_{\pm\beta}} \tens
\enhK
\ar[ru]_-{\simeq}^-{\rho_\alpha}
\ar[rd]_-{\simeq}^-{\rho_{-\alpha}} \\
&\opb\pi\field_{h_{\pm\beta}} \tens
\left(\DSum_{c\in\Sigma} \Phi^{(-\alpha,\beta)}_c(F) \tens E^{cw} \right)
}
\end{array}
\end{equation}
followed by \eqref{eq:Phipmiso}. Due to the natural isomorphism \eqref{eq:epmebta}, the bottom line of \eqref{eq:Spmbeta} is determined by the linear maps it induces on a stalk at a point
\[
(w,t) \in \pi^{-1}(h_{\pm\beta}) \cap \bigcap_{c\in\Sigma} \{ t+ \Re cw \ge 0 \} .
\]
Let us fix such a point $(w,t) $ -- e.g. fix any $w \in h_{\pm\beta} $ and then any $t\gg 0 $ big enough will satisfy this condition. Recall that $$\enhK \simeq \reim{\tilde q}(\field_{\{t+\Re zw\geq 0\}}\tens\opb{\overline p} F)$$ and that $Z_{(w,t)} = \{z\in\V\semicolon t+\Re zw \geq 0\} $ denoted the  fiber of $\widetilde{q} $ over $(w,t) $. Hence, on the level of stalks over $(w,t) $ and invoking the real oriented blow-up due to the definition of the maximal Beilinson extension, the morphism $\rho_{\pm\alpha} $ reads as
\[
\rsect_\rc(\widetilde \V; \field_{\opb\varpi_\Sigma Z_{(w,t)} \cap \widetilde\V^\Xi_\Sigma} \tens \wtL[1]) \isoto
\DSum_{c\in\Sigma} \rsect_\rc(\widetilde \V; \field_{\ell_{\widetilde c}^\pm} \tens \wtL[1]).
\]

We will now pass to the dual side using Borel-Moore homology based on subanalytic chains in the spirit of \cite[\S 9.2]{KS90}. In Appendix \ref{app:BM}, we will explain how to adapt the methods of \cite{KS90} to our situation, and in particular how to derive Lemma \ref{lem:lemmaFT} which will be the main technical tool for our computations. The definition of the Borel-Moore homology groups $H_j^{BM}(X;L) $ on a real analytic manifold $X $ with values in a local system is given in the Appendix.

Let $U\subset X$ be an open subanalytic subset, and set $Y = X \setminus U$.
If $L$ is a local system of finite rank on $X$,
the exact sequence
\begin{equation}
\label{eq:UXY}
0 \to \field_U \to \field_X \to \field_Y \to 0
\end{equation}
induces an exact homology sequence
\begin{equation}
\label{eq:HBMex}
H_{j+1}^{BM}(X; L^*) \to H_{j+1}^{BM}(U; L^*)
\to[\delta] H_j^{BM}(Y; L^*) \to H_j^{BM}(X; L^*)
\end{equation}
whose morphisms are induced, at the level of chains, by restriction $[S]\to{}[S|_U]$, boundary value $\delta\colon[S]\to{[\partial S]}$, and natural embedding $[S]\to{}[S]$, respectively.
The boundary value $\delta$ is also easily described using the following lemma, proved in the Appendix.

\begin{lemma}\label{lem:lemmaFT}
Let $F\in\Perv_\Sigma(\C_\V)$, and assume that $X\subset\V\setminus\Sigma$.
Let $L=F|_X[-1]$ be the associated local system on $X$.
\begin{itemize}
\item[(i)]
For any $Z\subset X$ locally closed subanalytic, and any $j\in\Z$, one has
\[
(H^j\rsect_\rc(\V; \field_Z\tens F))^* \simeq H_{j+1}^{BM}(Z; L^*).
\]
\item[(ii)]
Let $U\subset X$ be an open subanalytic subset, and set $Y = X \setminus U$.
Then, by the above duality, the morphisms
\begin{multline*}
H^{j-1}\rsect_\rc(\V; \field_X\tens F) \to H^{j-1}\rsect_\rc(\V; \field_Y\tens F)
\\
\to H^j\rsect_\rc(\V; \field_U\tens F) \to H^j\rsect_\rc(\V; \field_X\tens F)
\end{multline*}
induced by \eqref{eq:UXY} are adjoint to the morphisms in \eqref{eq:HBMex}.
\end{itemize}
\end{lemma}

We will make extensive use of (ii) in the following. Applied to the definition of $S_{\pm\beta} $ in \eqref{eq:SBeil}, we see that the latter is the transpose of the composition
\begin{multline}\label{eq:BMHwtBeil}
\sigma_{\pm\beta}:
\DSum_{c\in\Sigma} H_1^{BM}(\ell_{\widetilde c}^+ ; \wtL^\ast) \isoto[\tau]
\DSum_{c\in\Sigma} H_1^{BM}(\ell_{\widetilde c}^- ; \wtL^\ast) \\
\isoto[\mathrm{(i)}]
H_1^{BM}(\opb\varpi_\Sigma Z_{(w,t)} \cap \widetilde\V_\Sigma^\Xi ; \wtL^\ast) \isofrom[\mathrm{(ii)}] \DSum_{c\in\Sigma} H_1^{BM}(\ell_{\widetilde c}^+ ; \wtL^\ast)
\end{multline}
where (i) and (ii) are induced by the respective closed embeddings and $\tau $ is the direct sum of the transposes of \eqref{eq:Phipmiso} which we now compute on the real oriented blow-up $\widetilde\V_\Sigma $. Instead of \eqref{eq:Phipm}, we now use the exact sequence
\[
0 \to \field_{B_{\widetilde c}^\geq \setminus \ell^\pm_{\widetilde c}} \to \field_{B_{\widetilde c}^\geq} \to \field_{\ell^\pm_{\widetilde c}} \to 0.
\]
Push-forward with respect to $\varpi_\Sigma $ induces \eqref{eq:Phipm}. We obtain the isomorphisms
\begin{equation}\label{eq:BMPhipm}
\tau_{c}: H_1^{BM}(\ell_{\widetilde c}^+ ; \wtL^\ast) \isoto
H_1^{BM}(B_{\widetilde c}^\geq; \wtL^\ast) \isofrom
H_1^{BM}(\ell_{\widetilde c}^- ; \wtL^\ast)
\end{equation}
induced by the closed embeddings (see Figure \ref{fig:PhipmBeilinson}). The isomorphism $\tau $ of \eqref{eq:BMHwtBeil} is the direct sum $\tau=\bigoplus_{c \in \Sigma} \tau_c $.

\begin{figure}
\includestandalone{PhipmBeilinsonn}
\caption{The sets inducing the isomorphisms in \eqref{eq:BMPhipm}.}\label{fig:PhipmBeilinson}
\end{figure}

\setlength{\unitlength}{1mm}
\begin{figure}
\includestandalone{BM_Beilinson_0nn}
\caption{A choice of basis for $H_1^{BM}(\ell^+_{\widetilde c}; \field) $ and $H_1^{BM}(\ell^-_{\widetilde c}; \field) $ respectively.}\label{fig:BM_Beilinson_1}
\end{figure}

The decomposition $\ell_{\widetilde c}^+=\ell_c^\times \sqcup S_{\widetilde c}$ leads to an isomorphism
\begin{equation}\label{eq:lcSc}
\eta_c: H_1^{BM}(\ell_c^\times; \widetilde{L}^\ast) \oplus H_1^{BM}(S_{\widetilde c}; \widetilde{L}^\ast) \isoto H_1^{BM}(\ell_{\widetilde c}^+; \widetilde{L}^\ast) ,
\end{equation}
which in turn is the transpose of the isomorphism $\Phi_c(\Xi_\Sigma F) \simeq \Psi_c(F) \oplus \Psi_c(F) $ of Corollary \ref{cor:BeiQuiv}.

The Borel-Moore cycles $\fra_c^+:=\ell_c^\times $ with orientation towards $\widetilde c $ and $\frb_c:=S_c \smallsetminus \widetilde c $ in counter-clockwise orientation, represent a basis (see Figure \ref{fig:BM_Beilinson_1}) for $H^{BM}_1(\ell_{\widetilde c}^+; \field) $ compatible with the decomposition \eqref{eq:lcSc}. The basis $\fra_c^- $, $\frb_c^- $ we use for $H^{BM}_1(\ell_{\widetilde c}^-; \field) $ can be read off in Figure \ref{fig:BM_Beilinson_1}.

Let us fix the index $j \in \{1, \ldots, n \} $. Consider the Borel-Moore 2-chain pictured in Figure \ref{fig:BM2chain} twisted by a section $v $ of the local system $L $. In the following, we will write $\fra_j \tens v $ for the induced Borel-Moore cycle, where we consider $\ell^\times_j $ as a subset of the boundary of $B^\geq_{\widetilde c_j} $ and take $v $ to be the
restriction to $\ell^\times_j $ of its
extension to $B^\geq_{\widetilde c_j} $ (in other words we extend $v $ to $\ell^\times_j $ approaching from the right in the given orientation).

Considering the boundary of this Borel-Moore 2-chain and the convention on the notation just given, we obtain the relation
\begin{equation}\label{eq:BMjrelation}
\fra_j^- \tens v = \fra_j^+ \tens v - \frb_j^+ \tens v + \sum_{\nu=j+1}^n \big(\fra_\nu^+ \tens (1-\opb{(\PsiT_{\nu\nu}^\ast)})v - \frb_\nu^+ \tens v \big)
\end{equation}
in $H_1^{BM}(\opb\varpi_\Sigma Z_{(w,t)} \cap \widetilde\V_\Sigma^\Xi ; \widetilde{L^\ast}_\Sigma ) $, where $\PsiT_{\nu\nu} $ is the local monodromy of $\Psi_\nu $ (see Lemma \ref{lem:PsiT}) and $\opb{(\PsiT_{\nu\nu}^\ast)} $ is its transpose due to the following remark.

\begin{remark}\label{rem:dualmonodromy}
For a local system $L $ on the punctured disc around a point with generic fiber $L_b $ and monodromy $\PsiT $, the monodromy $\PsiT^\ast $ of the dual local system $L^\ast $ is characterized by the requirement $\langle \PsiT^\ast v, \PsiT z\rangle = \langle v,z \rangle $ for $z \in L_b $ and $v \in L_b^\ast $. Therefore the transpose of the monodromy isomorphism $\PsiT:L_b \to L_b $ is given by $\opb{(\PsiT^\ast)} $.
\end{remark}

For each $i\in\{1,\ldots, n\} $, let us introduce the following notation
\begin{multline}\label{eq:sigtildij}
(\widetilde\sigma_{\pm\beta})_{ij} :
\begin{matrix}
H_1^{BM}(\ell_j^\times; \widetilde{L}^\ast) \\
\oplus\\
H_1^{BM}(S_{\widetilde j}; \widetilde{L}^\ast)
\end{matrix}
\isoto[\eta_j] H_1^{BM}(\ell_{\widetilde j}^+; \widetilde{L}^\ast)\\ \to[(\sigma_{\pm\beta})_{ij}]
H_1^{BM}(\ell_{\widetilde i}^+; \widetilde{L}^\ast) \isoto[\eta_i^{-1}]
\begin{matrix}
H_1^{BM}(\ell_i^\times; \widetilde{L}^\ast) \\
\oplus\\
H_1^{BM}(S_{\widetilde i}; \widetilde{L}^\ast)
\end{matrix}
\end{multline}
Note, that $H_1^{BM}(S_{\widetilde j}; \widetilde L^\ast) \simeq H_1^{BM}(\ell_j^\times; \widetilde L^\ast) $. Since $\tau $ maps
\begin{equation} \label{eq:tautilde}
\fra_j^+ \tens v \mapsto \fra_j^- \tens v + \frb_j^- \tens v \text{\quad and \quad} \frb_j^+ \tens v \mapsto \frb_j^- \tens v \ ,
\end{equation}
we obtain from \eqref{eq:BMjrelation} that
\begin{equation}\label{eq:sigma+matrix}
(\widetilde\sigma_\beta)_{ij} =\left\{
\begin{array}{ll}
0 & \text{if $i<j $.}\\[.3cm]
\mathrm{id} & \text{if $i=j $ and} \\[.3cm]
\begin{pmatrix}
1-\opb{(\PsiT_{ii}^\ast)}&0 \\ -1& 0
\end{pmatrix}
& \text{if $i>j $,}
\end{array}\right.
\end{equation}
\begin{figure}
\includestandalone{BM_Beilinson_3nn}
\caption{The region $Z_{(w,t)} $ with $t\gg 0 $ (light grey) and the Borel-Moore 2-chain (dark grey) inducing the relation \eqref{eq:BMjrelation} for $w \in h_\beta $ (left hand side) and \eqref{eq:BMjrelation-} for $w \in h_{-\beta} $ (right hand side).}
\label{fig:BM2chain}
\end{figure}
The situation for $\widetilde\sigma_{-\beta} $ is similar. We obtain the relation
\begin{equation}\label{eq:BMjrelation-}
\fra_j^- \tens v = \fra_j^+ \tens \opb{(\PsiT_{jj}^\ast)} v + \sum_{\nu=1}^{j-1} \big( \fra_\nu^+ \tens (\opb{(\PsiT_{\nu\nu}^\ast)} -1)v + \frb_\nu^+ \tens v \big).
\end{equation}
Together with \eqref{eq:tautilde}, we deduce that
\begin{equation}\label{eq:sigma-matrix}
(\widetilde\sigma_{-\beta})_{ij} =\left\{
\begin{array}{ll}
0 & \text{if $i>j $,} \\[.3cm]
\begin{pmatrix}
\opb{(\PsiT_{ii}^\ast)} & 0\\ 1& 1
\end{pmatrix}& \text{if $i=j $ and} \\[.5cm]
\begin{pmatrix}
\opb{(\PsiT_{ii}^\ast)}-1 & 0\\ 1 & 0
\end{pmatrix}
& \text{if $i<j $.}
\end{array}\right.
\end{equation}

In order to be able to compare to the right hand side of \eqref{eq:S+} and \eqref{eq:S-}, we have to apply the transpose of {\renewcommand\eqc{c}\ref{eq:isoPsicPsi}} to both sides. We will use the following variant of \text{{\renewcommand\eqc{c}\ref{eq:isoPsicPsi}}} involving the punctured real oriented blow-up $\widetilde V_\Sigma^\Xi$. In the Notation \ref{nota:realblowup} and additionally setting $\ell_{\widetilde\Sigma}=\bigcup_{c \in \Sigma}\ell_{\widetilde c} $, we consider the short exact sequences (see Figure~\ref{fig:sesisotilde}) 
\begin{align}\label{eq:sesisotilde}
& 0 \to \field_{B_{\widetilde c}^>} \to \field_{B_{\widetilde c}^> \cup \ell_c^\times} \to \field_{\ell_c^\times} \to 0\\
\notag
& 0 \to \field_{B_{\widetilde c}^>} \to \field_{B_{\widetilde c} \setminus \ell_{\widetilde c}} \to \field_{B_{\widetilde c}^\leq \smallsetminus \ell_{\widetilde c}} \to 0\\
\notag
&
0 \to \field_{B_{\widetilde c} \setminus \ell_{\widetilde c}} \to \field_{\widetilde V_\Sigma^\Xi\setminus \ell_{\widetilde\Sigma}} \to
\field_{(\widetilde V_\Sigma^\Xi \setminus \ell_{\widetilde\Sigma}) \setminus B_{\widetilde c}} \to 0,
\end{align}
analogous to \eqref{eq:sesiso} and \eqref{eq:sesisoS}. Since $B_{\widetilde c}^> \cup \ell_c^\times $, $B_{\widetilde c}^\leq \smallsetminus \ell_{\widetilde c} $ and $(\widetilde V_\Sigma^\Xi \setminus \ell_{\widetilde\Sigma}) \setminus B_{\widetilde c} $ have no compactly supported cohomology, the remaining morphisms induced by each of the sequences in \eqref{eq:sesisotilde} are isomorphisms. The resulting isomorphism
\refstepcounter{equation}
\begin{multline}\label{eq:isoPsicPsidual}
H_2^{BM}(\widetilde V_\Sigma^\Xi\smallsetminus\ell_{\widetilde\Sigma}; \widetilde L^\ast)
\isoto H_2^{BM}(B_{\widetilde c} \setminus \ell_{\widetilde c}; \widetilde L^\ast)\\
\isoto H_2^{BM}(B_{\widetilde c}^> ; \widetilde L^\ast)
\isoto
H_1^{BM}(\ell_c^\times; \widetilde L^\ast)
\tag*{{(\theequation)$_\eqc$}}
\end{multline}
is the transpose of {\renewcommand\eqc{c}\ref{eq:isoPsicPsi}}. Let us write
\[
(\widehat\sigma_{\pm\beta})_{ij} :=
(\text{{\renewcommand\eqc{i}\ref{eq:isoPsicPsidual}}} \oplus
\text{{\renewcommand\eqc{i}\ref{eq:isoPsicPsidual}}} )
\circ (\widetilde\sigma_{\pm\beta})_{ij} \circ
(\text{{\renewcommand\eqc{j}\ref{eq:isoPsicPsidual}}} \oplus
\text{{\renewcommand\eqc{j}\ref{eq:isoPsicPsidual}}} )^{-1}.
\]
\begin{figure}
\begin{align*}
&\includestandalone{PsicPsitilde1n}\\
&\includestandalone{PsicPsitilde3n}\\
&\includestandalone{PsicPsitilde2n}
\end{align*}
\caption{\label{fig:sesisotilde}The exact sequences \eqref{eq:sesisotilde}.}
\end{figure}
By \eqref{eq:TlocT} (and its dual version), we deduce from \eqref{eq:sigma+matrix} that
\begin{equation}\label{eq:sigmaijfinal}
(\widehat\sigma_\beta)_{ij} =\left\{
\begin{array}{ll}
0 & \text{if $i<j $.}\\[.3cm]
\mathrm{id} & \text{if $i=j $ and} \\[.3cm]
\begin{pmatrix}
1-\opb{(\PsiT_{i}^\ast)}&0 \\ -1& 0
\end{pmatrix}
& \text{if $i>j $,}
\end{array}\right.
\end{equation}
and analogously
\begin{equation}\label{eq:sigmafinal-}
(\widehat\sigma_{-\beta})_{ij} =\left\{
\begin{array}{ll}
0 & \text{if $i>j $,} \\[.3cm]
\begin{pmatrix}
\opb{(\PsiT_{i}^\ast)} & 0\\ 1& 1
\end{pmatrix}& \text{if $i=j $ and} \\[.5cm]
\begin{pmatrix}
\opb{(\PsiT_{i}^\ast)}-1 & 0\\ 1 & 0
\end{pmatrix}
& \text{if $i<j $.}
\end{array}\right.
\end{equation}

Now, let us consider the right hand side of \eqref{eq:S+} and \eqref{eq:S-}. We put
\[
a_j:=
\text{{\renewcommand\eqc{j}\ref{eq:isoPsicPsi}}} \oplus
\text{{\renewcommand\eqc{j}\ref{eq:isoPsicPsi}}}:
\Psi_j(F) \oplus \Psi_j(F) \isoto \Psi(F) \oplus \Psi(F) .
\]
By Corollary \ref{cor:BeiQuiv} and Remark \ref{rem:locglob}, we obtain the commutative diagram
\begin{equation}\label{eq:uivj}
\xymatrix@C=.5cm{
\Phi_j(\Xi_\Sigma F) \ar[r]^{v_{jj}} \ar[d]_{b_j}^{\simeq}
\ar@/^2pc/[rrrr]^{u_iv_j}
& \Psi_j(\Xi_\Sigma F) \ar[r]^{
\text{{\renewcommand\eqc{j}\ref{eq:isoPsicPsi}}}}
& \Psi(\Xi_\Sigma F) \ar[r]^{
\text{{\renewcommand\eqc{i}\ref{eq:isoPsicPsi}}}^{-1}}
& \Psi_i(\Xi_\Sigma F) \ar[r]^{u_{ii}}
& \Phi_i(\Xi_\Sigma F)
\ar[d]^{b_i}_{\simeq}
\\
\Psi(F) \oplus \Psi(F)
\ar[rr]^{
\left(\begin{smallmatrix}1-T_j&&-1\end{smallmatrix}\right)}
&& \Psi(F) \ar[rr]^{
\left(\begin{smallmatrix}1\\ 0\end{smallmatrix}\right)}
&& \Psi(F) \oplus \Psi(F),
}
\end{equation}
where $b_j $ is the composition
\[
b_j: \Phi_j(\Xi_\Sigma F) \isoto \Psi_j(F) \oplus \Psi_j(F) \isoto[a_j] \Psi(F) \oplus \Psi(F),
\]
being the transpose of
\[
\big( \text{{\renewcommand\eqc{j}\ref{eq:isoPsicPsidual}}} \oplus
\text{{\renewcommand\eqc{j}\ref{eq:isoPsicPsidual}}} \big)
\circ \eqref{eq:lcSc}
\colon \begin{matrix}
H_2^{BM}(\widetilde V_\Sigma^\Xi\setminus\ell_{\widetilde\Sigma}; \widetilde L^\ast)\\
\oplus\\
H_2^{BM}(\widetilde V_\Sigma^\Xi\setminus\ell_{\widetilde\Sigma}; \widetilde L^\ast)
\end{matrix}
\isoto H_1^{BM}(\ell_{\widetilde c}^+; \widetilde{L}^\ast) .
\]
Hence
\begin{align}\label{eq:uivjfinal}
& b_j \circ (u_i v_j) \circ b_j^{-1} =
\begin{pmatrix}
1-\PsiT_j & -1\\ 0&0
\end{pmatrix},
\\ \notag
& b_i \circ \PhiT_i \circ b_i^{-1} =1-b_i \circ (u_iv_i) \circ b_i^{-1}=
\begin{pmatrix}
\PsiT_i & 1\\ 0 & 1
\end{pmatrix}.
\end{align}
By duality and Remark \ref{rem:dualmonodromy}, the equations \eqref{eq:uivjfinal} together with \eqref{eq:sigmaijfinal} and \eqref{eq:sigmafinal-} prove Theorem \ref{thm:Stokesmultipliers} in the case where $F$ is a maximal Beilinson extension.
\qed

\begin{remark}
Let us comment on the case where
$F=\reim{{j_\Sigma}}L[1]$ is a localized perverse sheaf, i.e. it is associated to a regular singular meromorphic connection by the Riemann-Hilbert correspondence. Then $u_i $ is an isomorphism and up to isomorphism of quivers, we can assume that $u_{ii}=\mathrm{id} $ and consequently $\PsiT_{ii}=\PhiT_{ii}=1-v_{ii} $. Furthermore, $u_i= \text{{\renewcommand\eqc{i}\ref{eq:isoPsicPsi}}}^{-1} $. The Stokes matrices $S_{\pm \beta} $ of Theorem \ref{thm:Stokesmultipliers} then have the form
\begin{align*}
S_\beta & =
\begin{pmatrix}
1 & 1-\PsiT_2 & 1-\PsiT_3 & \cdots & 1-\PsiT_n \\
& 1 & 1-\PsiT_3 & \cdots & 1-\PsiT_n \\
& & \ddots & & \vdots \\
&&& & 1
\end{pmatrix},
\\[.3cm]
S_{-\beta} &=
\begin{pmatrix}
\PsiT_1\\
\PsiT_1-1 & \PsiT_2 \\
\PsiT_1-1 & \PsiT_2-1 & \ddots \\
\vdots & \vdots && \ddots\\
\PsiT_1-1 & \PsiT_2-1 & \cdots & \PsiT_{n-1}-1 & \PsiT_n
\end{pmatrix},
\end{align*}
where $\PsiT_i $ is the monodromy of the local system $L $ around the singularity~$c_i $ -- keeping in mind Remark \ref{rem:locglob}.
\end{remark}

\section{Quiver at finite distance}\label{sec:smash}

\subsection{Statement of the result}

Let $F\in\Perv_\Sigma(\field_\V)$, and consider the embeddings
\[
\xymatrix{
\W \ar@{ >->}[r]^-u & \W_\infty = (\W,\bb) & (\W\setminus\{0\},\bb\setminus\{0\}) \ar@{ >->}[l]_-v.
}
\]
Proposition~\ref{prop:expsectors} and Theorem~\ref{thm:Stokesmultipliers}
describe the structure of $\Eopb v e(F)^\curlywedge$.
Our aim is to analyze $\Eopb u e(F)^\curlywedge$, so that the description of $e(F)^\curlywedge$ will be complete.
This is achieved in Propositions~\ref{pro:fousma} and \ref{pro:Foupervsmash} below.

Consider the maps
\begin{equation*}
\xymatrix{
\V & \V\times\R_{<0} \ar@{ >->}[r]^-{j} \ar[l]_-\gamma & \V\times\R & \V
\ar@{ >->}[l]_-i,
}
\end{equation*}
where $j$ is the open embedding, $i(x) = (x,0)$, and $\gamma(x,s) = x/|s|$.

\begin{definition}\label{def:tildenu}
The smash functor is given by
\[
\smsh \colon \BDC(\field_\V) \to \BDC_{\R^+}(\field_\V), \quad
F \mapsto \opb i \roim {j} \opb\gamma F.
\]
\end{definition}

\begin{prop}\label{pro:fousma}
Let $F\in\Perv_\Sigma(\field_\V)$.
Then
\[
\Eopb u e(F)^\curlywedge \simeq e(\smsh(F)^\wedge).
\]
\end{prop}

\begin{proof}
Consider the triangle
\[
\opb\pi\field_{\W\setminus\{0\}}\tens \Eopb u e(F)^\curlywedge \to \Eopb u e(F)^\curlywedge \to \opb\pi\field_{\{0\}}\tens \Eopb u e(F)^\curlywedge \to[+1].
\]

For $c\in\Sigma$, the function $cw$ is locally bounded on $\W$. Hence $\field_{\W}^\enh\ctens E^{cw} \simeq \field_{\W}^\enh$.
Thus Proposition~\ref{prop:expsectors} implies that
$\opb\pi\field_{\W\setminus\{0\}}\tens \Eopb u e(F)^\curlywedge \simeq e(G')$ with $G' = \reim j' L'$, for $j'\colon \W\setminus\{0\} \to \W$ the embedding and $L'$ a local system on $\W\setminus\{0\}$.

One has $\opb\pi\field_{\{0\}}\tens \Eopb u e(F)^\curlywedge\simeq e(G'')$ for some $G''\in\BDC(\field_{\W})$ with $G''|_{\W\setminus\{0\}} \simeq 0$. It follows that $\Eopb u e(F)^\curlywedge\simeq e(G)$ for $G\in\BDC(\field_{\W})$ entering a distinguished triangle
\[
G' \to G \to G'' \to[+1].
\]

Since $\Eopb u e(F)^\curlywedge\simeq e(G)$, one has
\begin{align*}
G
&\simeq \rhom^\enh(\field_{\{t\geq 0\}},e(G)) \\
&\simeq \rhom^\enh(\field_{\{t\geq 0\}},\Eopb u e(F)^\curlywedge) \\
&\underset{(*)}\simeq \roim\pi (\epsilon(F)^\curlywedge),
\end{align*}
where $(*)$ follows from Lemma~\ref{lem:pi}.
One then concludes using Lemma~\ref{lem:TamFou} below.
\end{proof}

\begin{lemma}\label{lem:TamFou}
For $F\in\BDC(\field_\V)$, one has
\[
\roim\pi(\epsilon(F)^\curlywedge) \simeq \smsh(F)^\wedge.
\]
\end{lemma}

\begin{proof}
One has
\begin{align*}
\epsilon(F)^\curlywedge
&\underset{(1)}\simeq (F \etens \field_{\{-1\}})^\wedge \\
&\underset{(2)}\simeq (\reim {j'} \opb\gamma F)^\wedge[1] \\
&\simeq (\reim {j'} \epb{\gamma} F)^\wedge,
\end{align*}
where $(1)$ follows from \cite[Proposition A.3]{DAg14}, and $(2)$ from
\cite[Lemma 6.1]{DAg14}.
Hence
\begin{align*}
\roim\pi(\epsilon(F)^\curlywedge)
&\simeq \roim\pi(\reim {j'} \epb{\gamma} F)^\wedge \\
&\simeq (\epb i \reim {j'} \epb{\gamma} F)^\wedge.\qedhere
\end{align*}
\end{proof}

\begin{prop}\label{pro:Foupervsmash}
Let $F\in\Perv_\Sigma(\field_\V)$ with quiver
\[
Q^{(\alpha,\beta)}_\Sigma(F) = \bigl(\Psi, \Phi_c, u_c,
v_c\bigr)_{c\in\Sigma}.
\]
Then $\smsh (F)^\wedge$ is an object of $\Perv_{\{0\}}(\field_{\W})$ with quiver
\[
Q^{(\beta,-\alpha)}_{\{0\}}(\smsh (F)^\wedge) = \bigl(\Phi_\Sigma,\ \Psi,\
V_\Sigma,\ U_\Sigma \bigr),
\]
where $\Phi_\Sigma \defeq \DSum_{i=1}^n \Phi_{c_i}$, $U_\Sigma \defeq
{}^t(U_1,\dots,U_n)$, $V_\Sigma \defeq (V_1,\dots,V_n)$,
\begin{align*}
U_i &\defeq u_{c_i} T_{c_{i+1}} T_{c_{i+2}} \cdots T_{c_n}, \\
V_i &\defeq v_{c_i},
\end{align*}
and $T_c \defeq 1-v_c u_c$ is the monodromy of $\Psi$ around
$c\in\Sigma$.
\end{prop}

\begin{proof}
This follows from Proposition~\ref{pro:pervsmash} and Lemma~\ref{lem:FouSat}.
\end{proof}

\subsection{Computation of the smash functor}

In this section, we compute the smash $\smsh(F)$ for objects $F\in\BDC(\field_\V)$ which are locally constant in a neighborhood of $\infty$. This is enough for our purposes, since we are interested in the case where $F$ is perverse.
We collect in Appendix~\ref{app:smash} some results on $\smsh(F)$ for general $F$'s, which might be of independent interest (e.g.~perverse sheaves in a higher dimensional affine space no longer satisfy the previous condition.)

Fix $R>0$, and consider the map $c^R \colon \V \to \V$ given by
\[
c^R(x) =
\begin{cases}
0 &\text{if }|x|\leq R, \\
\dfrac{|x|-R}{|x|}\,x &\text{if }|x|\geq R,
\end{cases}
\]
that contracts the ball $|x|\leq R$ to the origin.

\begin{lemma}\label{lem:smash}
If $F\in\BDC(\field_\V)$ is locally constant on $|x|\geq R$, then
\[
\smsh(F)\simeq \roim c^R F\simeq \reimm c^R F.
\]
\end{lemma}
\begin{proof}
Consider the maps
\[
\xymatrix{
\V & \V\times\R_{<0} \ar@{ >->}[r]^{j} \ar[l]_-\gamma & \V\times\R_{\leq0} & \V
\ar@{ >->}[l]_-i .
}
\]
The isomorphism $\roim c^R F\simeq \reimm c^R F$ is clear since $c^R$ is proper.
Hence, by the definition of $\smsh$, we are left to prove
\[
\opb i \roim j \opb\gamma F \simeq
\roim c^R F.
\]
Consider the maps
\begin{align*}
p,\eta_R &\colon \V\times\R_{<0} \to \V, \\
\overline p,\xi^R &\colon \V\times\R_{\leq 0} \to \V,
\end{align*}
where $p$ and $\overline p$ are the natural projections, and
\begin{align*}
\eta_R(x,s) &=
\begin{cases}
\dfrac{1}{|s|}\,x &\text{if }|x|\leq R|s|, \\[1em]
\dfrac{|x|-R(|s|-1)}{|x|}\,x &\text{if }|x|\geq R|s|,
\end{cases}
\\
\xi^R(x,s) &=
\begin{cases}
|s|x &\text{if }|x|\leq R, \\
\dfrac{|x|+R(|s|-1)}{|x|}\,x &\text{if }|x|\geq R.
\end{cases}
\end{align*}
Note that one has $\eta_R(x,-1)=x=\xi^R(x,-1)$, and
\[
c^R(x) = \xi^R(x,0).
\]
Consider the endomorphisms
\begin{align*}
\widetilde\eta_R &\colon \V\times\R_{<0} \to \V\times\R_{<0},
\quad(x,s) \mapsto (\eta_R(x,s),s),\\
\widetilde\xi^R &\colon \V\times\R_{\leq0} \to \V\times\R_{\leq0},
\quad(x,s) \mapsto (\xi^R(x,s),s).
\end{align*}
Note that
\begin{equation}
\label{eq:xieta}
\widetilde\eta_R = (\widetilde\xi^R|_{\{s<0\}})^{-1}.
\end{equation}

The sheaves $\opb{\gamma} F$ and $\opb{\eta_R} F$ are locally constant on $|x|\geq R |s|$. Moreover, they coincide on $|x|\leq R |s|$, since $\gamma=\eta_R$ on that region.
As $|x|= R |s|$ is a deformation retract of $|x|\geq R |s|$, it follows that $\opb{\gamma} F \simeq \opb{\eta_R} F$ on $\V\times\R_{<0}$.
(See Figure~\ref{fig:GammaEta}.)
\begin{figure}
\includestandalone[scale=.85]{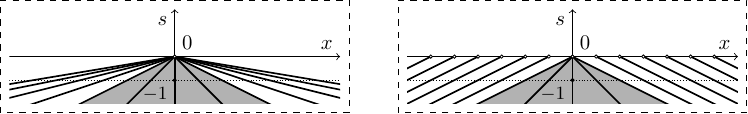}
\caption{Some level sets of $\gamma$ (on the left) and $\eta_R$ (on the right). The region $|x|\leq R|s|$ is pictured in gray.}\label{fig:GammaEta}
\end{figure}
We are then left to show
\[
\opb i \roim j \opb{\eta_R} F \simeq
\roim c^R F.
\]
As a preliminary result, note that if $G\in\BDC(\field_{\V\times\R_{\leq 0}})$ is locally constant on $|x|\geq R|s|$, then
\begin{equation}
\label{eq:ijjG}
\opb i G \isoto \opb i \roim j \opb j G.
\end{equation}
Indeed, one has $(\opb i G)_0 \simeq (\opb i \roim j \opb j G)_0$ by direct computation, and the isomorphism for $x\neq 0$ follows from the fact that $\field_{\V\times\R_{\leq0}} \isoto \roim j \opb j \field_{\V\times\R_{\leq0}}$.
Note also that one has
\begin{equation}
\label{eq:cxit}
\opb i \roim{\widetilde\xi^R} \simeq \roim{c^R} \opb i,
\end{equation}
since $\widetilde\xi^R$ and $c^R$ are proper, and there is a Cartesian square
\[
\xymatrix{
\V\times\R_{\leq0} \ar[r]^{\widetilde\xi^R} & \V\times\R_{\leq0} \\
\V \ar[u]^i \ar[r]^{c^R} & \V \ar[u]^i .
}
\]
To conclude, one has
\begin{align*}
\opb i \roim j \opb{\eta_R} F
&\simeq \opb i \roim j \opb{\widetilde\eta_R} \opb{p} F \\
&\underset{(1)}\simeq \opb i \roim j \roim{(\widetilde\xi^R|_{\{s<0\}})} \opb{p} F \\
&\simeq \opb i \roim j \opb j \roim{\widetilde\xi^R} \opb{\overline p} F \\
&\xleftarrow[(2)]{~\sim~} \opb i \roim{\widetilde\xi^R} \opb{\overline p} F \\
&\underset{(3)}\simeq \roim{c^R} \opb i \opb{\overline p} F \\
&\simeq \roim c^R F,
\end{align*}
where $(1)$ follows from \eqref{eq:xieta}, $(2)$ from \eqref{eq:ijjG}, and $(3)$ from \eqref{eq:cxit}.
\end{proof}

\subsection{Quiver of the smash}

For $R>0$, let $\Delta = \Delta_R \subset \V$ be the closed ball centered at the origin with radius~$R$.
Choose $R$ big enough so that $\Delta\supset\Sigma$, and set
\begin{align*}
&\ell_\Delta(\alpha) \defeq \Delta + \R_{\geq 0}\alpha, \\
&\ell^\times_\Delta(\alpha) \defeq \ell_\Delta(\alpha) \setminus \Delta,\\
&B_\Delta^>(\beta) \defeq \{z\in \V \semicolon \Re\langle z,\beta\rangle < R \}, \\
&B_\Delta^\leq(\beta) \defeq \{z\in \V \semicolon \Re\langle z,\beta\rangle \geq R \}.
\end{align*}
We will write for short $\ell^\times_\Delta$, $B^\leq_\Delta$, etc.\ instead of
$\ell^\times_\Delta(\alpha)$, $B^\leq_\Delta(\beta)$, etc. These sets are pictured in Figure~\ref{fig:Quiver_Delta}.
\begin{figure}
\includestandalone[scale=.9]{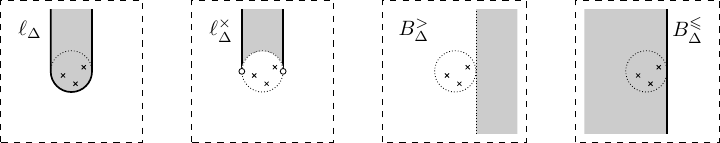}
\caption{The sets $\ell_\Delta$, $\ell^\times_\Delta$, $B_\Delta^>$ and $B_\Delta^\leq$. The crosses represent the points in $\Sigma$}\label{fig:Quiver_Delta}
\end{figure}

Consider the short exact sequences
\begin{equation}
\label{eq:sesuvDelta}
\begin{split}
&0 \to \field_{\ell^\times_\Delta} \to \field_{\ell_\Delta} \to \field_\Delta \to 0,
\\
&0 \to \field_{B_\Delta^>} \to \field_\V \to \field_{B_\Delta^\leq} \to 0.
\end{split}
\end{equation}

\begin{lemma}\label{lem:tildenuglob}
Let $F \in \Perv_\Sigma(\field_\V)$.
\begin{itemize}
\item[(i)]
Denoting $i_0\colon \{0\} \to \V$ the embedding, one has
\begin{align*}
\opb{i_0}\smsh(F) &\simeq \rsect_\rc(\V; \field_\Delta\tens F), \\
\epb{i_0}\smsh(F) &\simeq \rsect_\rc(\V; F).
\end{align*}
\item[(ii)]
There are isomorphism
\begin{align*}
\Psi_0(\smsh(F)) & \simeq \rsect_\rc(\V; \field_{\ell^\times_\Delta} \tens F) \simeq \rsect_\rc(\V; \field_{B_\Delta^>} \tens F)[1], \\
\Phi_0(\smsh(F)) & \simeq \rsect_\rc(\V; \field_{\ell_\Delta} \tens F) \simeq \rsect_\rc(\V; \field_{B_\Delta^\leq} \tens F).
\end{align*}
\item[(iii)]
By the isomorphisms in (i) and (ii), the distinguished triangles
\begin{align*}
&\Psi_0(\smsh(F)) \to[\tilde u_0] \Phi_0(\smsh(F)) \to \opb i_0 \smsh(F) \to[+1], \\
&\epb{i_0} \smsh(F) \to \Phi_0(\smsh(F)) \to[\tilde v_0] \Psi_0(\smsh(F)) \to [+1],
\end{align*}
are identified with the distinguished triangles
\begin{align*}
&\rsect_\rc(\V; \field_{\ell^\times_\Delta} \tens F) \to[\tilde u_\Delta] \rsect_\rc(\V; \field_{\ell_\Delta} \tens F) \to \rsect_\rc(\V; \field_\Delta\tens F) \to[+1], \\
&\rsect_\rc(\V; F) \to \rsect_\rc(\V; \field_{B_\Delta^\leq} \tens F) \to[\tilde v_\Delta] \rsect_\rc(\V; \field_{B_\Delta^>} \tens F)[1] \to [+1],
\end{align*}
induced by \eqref{eq:sesuvDelta}.
\end{itemize}
\end{lemma}

\begin{proof}
By Lemma~\ref{lem:smash} one has $\smsh(F) \simeq \reim c F$, where we set $c=c^R$.
By direct computation, one obtains a statement analogue to the above, where
$\ell_\Delta$, $\ell^\times_\Delta$, $B_\Delta^<$, and $B_\Delta^\geq$ are replaced with
the sets
$\widetilde\ell_\Delta$, $\widetilde\ell^\times_\Delta$, $\widetilde B_\Delta^>$, and $\widetilde B_\Delta^\leq$, respectively, pictured in Figure~\ref{fig:Quiver_Delta_c}.
\begin{figure}
\includestandalone[scale=.9]{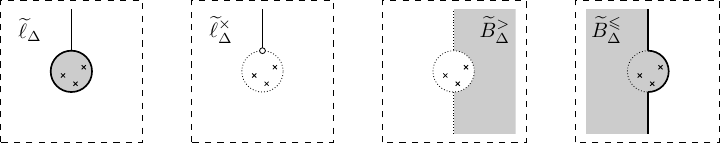}
\caption{The sets $\widetilde \ell_\Delta$, $\widetilde \ell^\times_\Delta$, $\widetilde B_\Delta^>$ and $\widetilde B_\Delta^\leq$. The crosses represent the points in $\Sigma$}\label{fig:Quiver_Delta_c}
\end{figure}
Since $\reim c F$ is locally constant on the closure of $\V\setminus\Delta$, one can deform these sets to the original sets, in a way which is compatible with the morphisms \eqref{eq:sesuvDelta}.
\end{proof}

\begin{prop}\label{pro:pervsmash}
Let $F\in\Perv_\Sigma(\field_\V)$ with quiver
\[
Q^{(\alpha,\beta)}_\Sigma(F) = \bigl(\Psi, \Phi_c, u_c,
v_c\bigr)_{c\in\Sigma}.
\]
Then $\smsh (F)$ is an object of $\Perv_{\{0\}}(\field_\V)$ with quiver
\[
Q^{(\alpha,\beta)}_{\{0\}}(\smsh (F)) = \bigl(\Psi,\ \Phi_\Sigma,\
U'_\Sigma,\ V'_\Sigma \bigr),
\]
where $\Phi_\Sigma \defeq \DSum_{i=1}^n \Phi_{c_i}$, $U_\Sigma \defeq
{}^t(U_1,\dots,U_n)$, $V_\Sigma \defeq (V_1,\dots,V_n)$,
\begin{align*}
U_i &\defeq u_{c_i} T_{c_{i+1}} T_{c_{i+2}} \cdots T_{c_n}, \\
V_i &\defeq v_{c_i},
\end{align*}
and $T_c \defeq 1-v_c u_c$ is the monodromy of $\Psi$ around
$c\in\Sigma$.
\end{prop}

\begin{proof} The statement is natural in $F $ and is certainly true for perverse sheaves supported on $\Sigma $. The general case can then be proved by considering the case of a maximal Beilinson extension. For simplicity, we will consider the localized case. The arguments for a maximal Beilinson extension are the same working on the real oriented blow-up $\widetilde\V_\Sigma $ instead.

Since $B_\Delta^\leq \smallsetminus \ell_\Sigma $ is $\Sigma$-negligible, the short exact sequence
\[
0 \to \field_{B_\Delta^>} \to \field_{\V \smallsetminus \ell_\Sigma} \to \field_{B_\Delta^\leq \smallsetminus \ell_\Sigma}
\to 0
\]
induces an isomorphism
\begin{multline}\label{eq:Psismash}
\Psi_0(\smsh(F)) \simeq
\rsect_\rc(\V; \field_{B^>_\Delta} \tens F)[1]\\
\isoto[\eta]
\rsect_\rc(\V; \field_{\V \smallsetminus \ell_\Sigma} \tens F)[1] = \Psi(F).
\end{multline}
The isomorphism
\begin{equation}\label{eq:Phismash}
\Phi_0(\smsh(F)) \simeq \bigoplus_{c \in \Sigma} \Phi_c(F)=\Phi_\Sigma(F)
\end{equation}
is induced by the closed embedding $\sqcup_{c\in\Sigma} \ell_c \hookrightarrow \ell_\Delta $ whose open complement is again $\Sigma $-negligible.
We then obtain the commutative diagram
\begin{equation}\label{eq:VSigma}
\begin{array}{c}
\xymatrix@C=3mm@R=2mm{
\Phi_0(\smsh(F))\ar@{}[d]|\simeq\\
\rsect_\rc(\V; \field_{\ell_\Delta} \tens F)
\ar[r]^{\simeq} \ar[dd]^{\wr} &
\rsect_\rc(\V; \field_{B^\leq_\Delta} \tens F) \ar[r]^-{\widetilde v_\Delta} &
\rsect_\rc(\V; \field_{B^>_\Delta} \tens F)[1] \ar[dd]^{\wr}_{\eta}\\&&\\
\bigoplus_{c\in\Sigma} \rsect_\rc(\V; \field_{\ell_c} \tens F)\ar[rr]^{V_\Sigma}
&&
\rsect_\rc(\V; \field_{\V \smallsetminus \ell_\Sigma} \tens F)[1]
}
\end{array}
\end{equation}

Since we assumed $F $ to be localized, i.e. $F = \reim{{j_\Sigma}} L[1] $ for a local system $L $, we additionally have
\[
\rsect_\rc(\V; \field_{\ell_c} \tens F) \isoto \rsect_\rc(\V; \field_{\ell^\times_c} \tens F) \isoto[
\text{{\renewcommand\eqc{c}\ref{eq:isoPsicPsi}}}]
\rsect_\rc(\V; \field_{\V \smallsetminus \ell_\Sigma} \tens F)[1]
\]
We pass to the dual Borel-Moore side and are led to consider the induced morphism
\begin{multline}\label{eq:VSigmaBM}
H_2^{BM}(\V \smallsetminus \ell_\Sigma; L^\ast) \to[(V_\Sigma)^\ast] \bigoplus_{c\in \Sigma} H_1^{BM}(\ell_c; L^\ast)\\[-5pt]
\isoto
\bigoplus_{c\in \Sigma} H_2^{BM}(\V \smallsetminus \ell_\Sigma; L^\ast).
\end{multline}
Now, $H_2^{BM}(\V \smallsetminus \ell_\Sigma) $ is generated by the whole space as a Borel-Moore 2-cycle, denoted by $\Gamma $, with the orientation induced by the pair $(\alpha, \beta) $.
Following $\Gamma \tens w $ for a section $w $ of $L^\ast $ along the dual of the diagram \eqref{eq:VSigma}, we observe that it is mapped to the element
\[
\gamma \tens w \in H_1^{BM}(\ell_\Delta; L^\ast)
\]
given by the boundary curve of $\ell_\Delta $ in clockwise direction. The open subset $\ell_\Delta \smallsetminus ( \gamma \cup \ell_\Sigma ) $ with the obvious orientation gives a Borel-Moore 2-chain $\Upsilon \tens w $ whose boundary induces the relation
\begin{equation}\label{eq:Vrelation}
\gamma \tens w = \sum_{c\in\Sigma} \fra_c \tens (1- (T_{cc}^\ast)^{-1}) w
\end{equation}
where $\fra_c \tens w $ is the Borel-Moore homology class represented by the 1-cycle given by $\ell_c^\times $ in the orientation towards $c $ (see Figure \ref{fig:Quiver_DeltanewBM}). Here, we use the same convention we used for \eqref{eq:BMjrelation}, so that $w $ simply denotes the extension of the section $w $ to $\ell^\times_c $ coming from $B^\geq_c $ -- i.e. from the right.

It follows that \eqref{eq:VSigmaBM} maps
\[
( \Gamma \tens w) \Mto{\text{\eqref{eq:VSigmaBM}}}\sum_{c\in\Sigma} \fra_c \tens (1-(T_c^\ast)^{-1}) w
\]
which proves the assertion on $V_\Sigma $ since $v_c=1-T_c $ in the localized case and $(T_c^\ast)^{-1} $ is the transpose of $T_c $ (see Remark \ref{rem:dualmonodromy}).

\begin{figure}
\includestandalone{Quiver_DeltanewBM}
\caption{Left: The Borel-Moore 2-chain $\Upsilon $ inducing the relation \eqref{eq:Vrelation}. Right: Following $\fra_j \tens v $ around the transpose of \eqref{eq:USigma} results in this figure inducing \eqref{eq:USigmaformula}.}\label{fig:Quiver_DeltanewBM}
\end{figure}

Let us now consider the morphism $U_\Sigma $.
\begin{equation}\label{eq:USigma}
\begin{array}{c}
\xymatrix@C=3mm{
\Psi_0(\smsh(F)) \simeq \rsect_\rc(\V; \field_{\ell^\times_\Delta} \tens F) \ar[r]^(.6){\widetilde u_\Delta} \ar[d]^{\wr}_{\text{\eqref{eq:Psismash}}} &
\rsect_\rc(\V; \field_{\ell_\Delta} \tens F) \ar[d]^{\wr}\\
\rsect_\rc(\V; \field_{\V\smallsetminus \ell_\Sigma} \tens F)[1] \ar[r]^{U_\Sigma}
&
\bigoplus_{c\in\Sigma} \rsect_\rc(\V; \field_{\ell_c} \tens F)\\
\rsect_\rc(\V; \field_{\ell^\times_j} \tens F) \ar[u]_-{\wr}^-{\text{{\renewcommand\eqc{j}\ref{eq:isoPsicPsi}}}}&
}
\end{array}
\end{equation}
It is important to recall that \text{{\renewcommand\eqc{j}\ref{eq:isoPsicPsi}}} is defined as a composition
\[
\rsect_\rc(\V; \field_{\ell^\times_j} \tens F) \isoto
\rsect_\rc(\V; \field_{B_j \smallsetminus \ell_j} \tens F)[1] \isoto
\rsect_\rc(\V; \field_{\V\smallsetminus \ell_\Sigma} \tens F)[1]
\]
via the open embedding $B_j \smallsetminus \ell_j \hookrightarrow \V \smallsetminus \ell_\Sigma $.

Again, with $F= \reim{{j_\Sigma}} L[1] $ being localized, we switch to the dual Borel-Moore homology. The transpose of \eqref{eq:USigma} reads as
\begin{equation}\label{eq:USigmaBM}
\begin{array}{c}
\xymatrix{
& H_2^{BM}(\ell^\times_\Delta; L^\ast)
&
H_1^{BM}(\ell_\Delta;L^\ast) \ar[l]_(.5){\widetilde u_\Delta^\ast}\\
H_1^{BM}(\ell_j^\times;L^\ast) &
H_2^{BM}(\V\smallsetminus \ell_\Sigma;L^\ast) \ar[l]_{\rouge\sim} \ar[u]^{\simeq}_{\text{\eqref{eq:Psismash}}^\ast} &
\bigoplus_{c\in\Sigma} H_1^{BM}(\ell^\times_c;L^\ast) \ar[l]_{\rouge{U_\Sigma^*}} \ar[u]^{\simeq}
}
\end{array}
\end{equation}

Following the element $\fra_j\tens v $ for a local section $v $ of $L^\ast $ around the diagram \eqref{eq:USigmaBM}, recalling that $\widetilde{u}_\Delta $ involves the cohomology of $B^>_\Delta $ -- see \eqref{eq:Psismash} -- we obtain that
\begin{equation}\label{eq:USigmaformula}
\fra_j\tens v \Mto{\text{\eqref{eq:USigmaBM}}} \fra_j \tens \big( (T_{n}^\ast)^{-1} (T_{n-1}^\ast)^{-1} \cdots (T_{j+1}^\ast)^{-1} \big) v
\end{equation}
see Figure \ref{fig:Quiver_DeltanewBM}. This proves the statement on $U_\Sigma $.
\end{proof}

\subsection{Fourier-Sato transform}

Let us pair the real vector spaces $\V_\R$ and $(\W)_\R$ using the scalar
product $(z,w) \mapsto \Re\langle z,w \rangle$. Then the Fourier-Sato
transform has kernel $\field_{\{(z,w)\semicolon \Re\langle z,w \rangle \leq 0
\}}$.

Note that $\Perv_{\{0\}}(\field_\V)$ is a full subcategory of
$\BDC_{\R^+}(\field_{\V_\R})$.
As shown for example in \cite[Proposition 10.3.18]{KS90}, the Fourier-Sato
transform induces an equivalence of abelian categories
\[
\Perv_{\{0\}}(\field_\V) \isoto \Perv_{\{0\}}(\field_{\W}),
\quad F \mapsto \widehat F.
\]
The following fact is well known. We give here an easy proof in terms of our
definition of nearby and vanishing cycles.

\begin{lemma}\label{lem:FouSat}
Let $F\in\Perv_{\{0\}}(\field_\V)$, and consider its quiver
\[
Q_{\{0\}}^{(\alpha,\beta)}(F) = (\Psi,\Phi,u,v).
\]
Then the quiver of $\widehat F\in\Perv_{\{0\}}(\field_{\W})$ is given by
\[
Q_{\{0\}}^{(\beta,-\alpha)}(\widehat F) = (\Phi,\Psi,v,u).
\]
\end{lemma}

\begin{proof}
Considering the Fourier-Sato transform on the real line $\R$,
paired with itself by the product $(u,v)\mapsto uv$, we have
\begin{equation}
\label{eq:FSR}
\FS{\field_{\R_{\leq 0}}} \simeq \field_{\R_{<0}}, \quad
\FS{\field_{\R_{> 0}}} \simeq \field_{\R_{\leq 0}} [-1].
\end{equation}

Note that the $\R$-linear maps
\begin{align*}
p_\beta &\colon\V_\R\to\R, \quad z\mapsto\Re\langle z,\beta \rangle, \\
i^\beta &\colon\R\to(\W)_\R, \quad u\mapsto u\beta,
\end{align*}
are transpose to each other. Moreover, the sets
\begin{align*}
&B_0^<(\beta) = \{z\colon \Re\langle z,w \rangle > 0 \} = \opb
p_\beta(\R_{>0}), \\
&B_0^\geq(\beta) = \{z\colon \Re\langle z,w \rangle \leq 0 \} = \opb
p_\beta(\R_{\leq 0}), \\
&h_\beta = \R_{>0} \beta \subset \W = i^\beta(\R_{>0}) \\
&\ell_0(\beta) = \R_{\geq0} \beta \subset \W = i^\beta(\R_{\geq0}),
\end{align*}
are conic for the radial action of $\R^+$ on $\V$.

We deduce from \eqref{eq:FSR} the isomorphisms
\begin{align*}
\FS{\field_{B_0^\geq(\beta)}}
&\simeq \FS{\opb p_\beta\field_{\R_{\leq 0}}} &
\FS{\field_{B_0^<(\beta)}}
&\simeq \FS{\opb p_\beta\field_{\R_{> 0}}}
\\
&\simeq \reim{i^\beta}\FS{\field_{\R_{\leq 0}}} &
&\simeq \reim{i^\beta}\FS{\field_{\R_{> 0}}} \\
&\simeq \reim{i^\beta}\field_{\R_{<0}} &
&\simeq \reim{i^\beta}\field_{\R_{\leq 0}} [-1] \\
&\simeq (\field_{h_\beta})^a, &
&\simeq (\field_{\ell_0(\beta)})^a [-1],
\end{align*}
and hence we get the commutative diagram
\begin{equation*}
\xymatrix{
\rsect_\rc(\W; \field_{h_\beta} \tens \widehat F) \ar[r]^{\tilde u}
\ar@{-}[d]_\wr^{(1)} &
\rsect_\rc(\W; \field_{\ell_0(\beta)} \tens \widehat F)
\ar@{-}[d]_\wr^{(2)} \\
\rsect_\rc(\W; \FSa{\field_{B_0^\geq(\beta)}} \tens \widehat F) \ar[r]
\ar@{-}[d]_\wr^{(3)} &
\rsect_\rc(\W; \FSa{\field_{B_0^<(\beta)}} \tens \widehat F)[1]
\ar@{-}[d]_\wr^{(4)} \\
\rsect_\rc(\V; \field_{B_0^\geq(\beta)} \tens F) \ar[r]^{\tilde v'} &
\rsect_\rc(\V; \field_{B_0^<(\beta)} \tens F)[1].
}
\end{equation*}
Recalling that $F^{\wedge\wedge}\simeq F^a$ (with the canonical orientation
$\ori_\V\simeq\field_\V$), we
deduce the commutative diagram
\begin{equation*}
\xymatrix{
\rsect_\rc(\V; \field_{h(\alpha)} \tens F) \ar[r]^{\tilde u} \ar@{-}[d]_\wr
&
\rsect_\rc(\V; \field_{\ell_0(\alpha)} \tens F) \ar@{-}[d]_\wr \\
\rsect_\rc(\W; \field_{B_0^\geq(-\alpha)} \tens \widehat F) \ar[r]^{\tilde v'} &
\rsect_\rc(\W; \field_{B_0^<(-\alpha)} \tens \widehat F)[1].
}
\end{equation*}
Summarizing, we have shown that there are commutative diagrams
\[
\xymatrix{
\Psi_0^{(\beta,-\alpha)}(\widehat F) \ar[r]^{\tilde u} \ar@{-}[d]_\wr &
\Phi_0^{(\beta,-\alpha)}(\widehat F) \ar@{-}[d]_\wr \\
\Phi_0'^{(\alpha,\beta)}(F) \ar[r]^{\tilde v'} &
\Psi_0'^{(\alpha,\beta)}(F),
}
\qquad	
\xymatrix{
\Psi_0^{(\alpha,\beta)}(F) \ar[r]^{\tilde u} \ar@{-}[d]_\wr &
\Phi_0^{(\alpha,\beta)}(F) \ar@{-}[d]_\wr \\
\Phi_0'^{(\beta,-\alpha)}(\widehat F) \ar[r]^{\tilde v'} &
\Psi_0'^{(\beta,-\alpha)}(\widehat F).
}
\]
This implies the statement, since $\Psi_0'\simeq\Psi_0$ and $\Phi_0'\simeq\Phi_0$.
\end{proof}

\section{More Stokes phenomena}\label{sec:moreStoles}

Theorem~\ref{thm:Stokesmultipliers} describes the Stokes phenomenon at infinity for $\widehat \shm$, where $\shm$ is a holonomic $\D_\V$-module regular everywhere, including at infinity. In this section we will consider some special cases where $\shm$ is not regular, reducing to the regular case after some geometric manipulations.

\subsection{Airy equation}

The Stokes phenomenon was first analyzed in \cite{Sto57},
in relation with the Airy function $\operatorname{Ai}(y)$.
This is an entire solution of the Airy equation
\[
Qv=0, \quad\text{where } Q=\partial_y^2-y.
\]
The corresponding $\D$-module $\sha = \D_{\C_y}/\D_{\C_y} Q$ has $0$ and $\infty$ as its only singularities, with $0$ regular and $\infty$ irregular.
In this section we will compute the exponential components and the Stokes multipliers of $\sha$ at $\infty$ by analyzing its enhanced solutions
\[
A \defeq \solE_{(\C_y)_\infty}(\sha).
\]
Denote by $x$ the dual coordinate to $y$. Since $Q=\widehat P$ for $P=\partial_x-x^2$,
and $\D_{\C_x}/\D_{\C_x} P \simeq \she^{x^3/3}$, we have
\[
\sha \simeq (\she^{x^3/3})^\wedge.
\]
We know a priori (e.g.\ from \cite{Sab08}) that the formal structure of $(\she^{x^3/3})^\wedge$ at $\infty$ is ramified. We therefore consider the ramification
\[
r\colon \C_v \to \C_y, \quad v \mapsto v^2,
\]
and restrict ourselves to analyze the pull-back
$\Eopb r A$.

\begin{notation}
For $j=1,\dots,6$, consider the closed sectors in $\C_v^\times$
\[
H_j = \{ v=re^{\imm\theta}\semicolon r>0,\ \theta\in (j-1)\tfrac\pi 3 +[-\tfrac\pi 6,\tfrac\pi 6]\},
\]
and the open half-lines
\[
\ell_j = (H_j \cap H_{j+1})
= \R_{>0}\,e^{\imm\,j\tfrac\pi 3}
\]
(with $H_7 \defeq H_1$), as in the picture below.
\[
\begin{tikzpicture}
\foreach \j in {1,...,6}
	{
	\draw[thick] (0,0) -- (60*\j-30:1.5) ++(60*\j-30:2ex) node{$\ell_\j$} ;
	\node[at={(60*\j-60:1)}]{$H_\j$} ;
	}
\filldraw[fill=white, draw=black] (0,0) circle (2pt);
\end{tikzpicture}
\]
\end{notation}

Note that $\ell_j$ are the anti-Stokes lines for the exponentials $E^{\pm\frac23 \imm\,v^3}$, and that there are identifications
$$
e_j \colon \Endo\left(\Erest{\ell_j}{\bigl(E^{-\frac23 \imm\,v^3} \dsum E^{\frac23 \imm\,v^3} \bigr)} \right) \isoto \Endo^{(-1)^j}(\C^2),
$$
where $\Endo^{\pm 1}(\C^2)\subset\Endo(\C^2)$ denotes the vector
subspace of lower/upper triangular matrices.

\begin{prop}\label{pro:Airy}
With the above notations, let $j=1,\dots,6$.
\begin{itemize}
\item[(i)] The exponential components of $\Eopb r A$ at $\infty$ are $E^{\pm\frac23 \imm\,v^3}$. More precisely, there are isomorphisms
$$
s_j \colon \Erest{H_j}{\Eopb r A}
\isoto
\Erest{H_j}{\bigl(E^{-\frac23 \imm\,v^3} \dsum E^{\frac23 \imm\,v^3} \bigr)}.
$$
\item[(ii)]
The Stokes multipliers of $\Eopb r A$ at $\infty$, defined by
\[
S_j \defeq e_j\bigl((\Erest{\ell_j}{s_j}) \circ (\Erest{\ell_j}{s_{j+1}^{-1}}) \bigr),
\]
are given by
\[
S_{2k} =
\begin{pmatrix}
-1 & 0 \\
-1 & -1
\end{pmatrix}, \quad
S_{2k-1} =
\begin{pmatrix}
1 & -1 \\
0 & 1
\end{pmatrix},
\]
for $k=1,2,3$.
\end{itemize}
\end{prop}

Before entering the proof, note that one has
\begin{align}
\label{eq:AiryPrelim}
\Eopb r A
&\simeq \Eopb r \solE_{(\C_y)_\infty}((\she^{x^3/3})^\wedge) \\ \notag
&\simeq \Eopb r (E^{x^3/3})^\curlywedge[1] \\ \notag
&\simeq \Eopb r \Eeeim q E^{xy+x^3/3}[1] \\ \notag
&\simeq \Eeeim q E^{xv^2+x^3/3}[1].
\end{align}
As it is customary in the study of the Airy equation, consider the change of variables $\C_u\times\C^\times_v \isoto \C_x\times\C^\times_v$ given by
\begin{equation}
\label{eq:AiryCoord}
\begin{cases}
x = \imm\,u v,\\
v = v,
\end{cases}
\end{equation}
and consider the maps
\begin{align*}
f&\colon\C_u \to \C_z=\V, \quad u \mapsto u^3 - 3u,\\
g&\colon\C_v \to \C_w=\W, \quad v \mapsto \imm\,v^3/3,
\end{align*}
so that $xy+x^3/3 = x v^2+x^3/3 = f(u)g(v) = zw$. Set
\[
F\defeq \reim{f} \field_{\C_{ u}}[1] \in \BDC(\field_\V).
\]
The following two lemmas will be proved at the end of this section.

\begin{lemma}\label{lem:AiryReduct}
One has
\[
\Erest{\C^\times_v}{\Eopb r A} \simeq
\Erest{\C^\times_v}{\Eopb g ((e F)^\curlywedge)}.
\]
\end{lemma}

\begin{lemma}\label{lem:AiryQuiv}
One has $F\in\Perv_\Sigma(\V)$ for $\Sigma = \{-2,2\}$. Moreover, the quiver of $F$ is given by
\[
\vcenter{\vbox{\xymatrix{
\Phi_2(F) \ar@<.5ex>[d]^{v_2} \\
\Psi(F) \ar@<.5ex>[u]^{u_2} \ar@<-.5ex>[d]_{u_{-2}} \\
\Phi_{-2}(F) \ar@<-.5ex>[u]_{v_{-2}}
}}}
\simeq
\vcenter{\vbox{\xymatrix{
\field \ar@<.5ex>[d]^{\left(\begin{smallmatrix}0\\1\\-1\end{smallmatrix} \right)} \\
\field^3 \ar@<.5ex>[u]^{\left(\begin{smallmatrix}0&1&-1\end{smallmatrix} \right)} \ar@<-.5ex>[d]_{\left(\begin{smallmatrix}1&0&-1\end{smallmatrix} \right)} &.\\
\field \ar@<-.5ex>[u]_{\left(\begin{smallmatrix}1\\0\\-1\end{smallmatrix} \right)}.
}}}
\]
\end{lemma}

\begin{proof}[Proof of Proposition~\ref{pro:Airy}]
By Lemma~\ref{lem:AiryReduct}, it is enough to compute the exponential components of $\Eopb g ((e F)^\curlywedge)$ at $\infty$, and its Stokes multipliers.

Let $\alpha=\imm\in\C_z$, $\beta=1\in\C_w$. By Lemma~\ref{lem:AiryQuiv} and Proposition~\ref{prop:expsectors}, the exponential components of $(e F)^\curlywedge$ at $\infty$ are $E^{\pm 2w}$, and its Stokes multipliers are
\[
S_\beta =
\begin{pmatrix}
-1 & 0 \\
-1 & -1
\end{pmatrix}, \quad
S_{-\beta} =
\begin{pmatrix}
1 & 1 \\
0 & 1
\end{pmatrix}.
\]

To conclude, note that $\Eopb g E^{\pm 2w} \simeq E^{\pm\frac23 \imm\,v^3}$,
\[
\opb{g}H_\alpha = \Union_k H_{2k- 1},
\quad \opb{g}H_{-\alpha} = \Union_k H_{2k},
\]
and hence
\[
S_{2k} = S_\beta, \quad S_{2k-1} = S_{-\beta}^{-1}.\qedhere
\]
\end{proof}

\begin{proof}[Proof of Lemma~\ref{lem:AiryReduct}]
By \eqref{eq:AiryPrelim} and \eqref{eq:AiryCoord}, one has
\begin{align*}
\Erest{\C^\times_v}{\Eopb r A}
&\simeq \Erest{\C^\times_v}{\Eeeim q E^{xv^2+x^3/3}[1]} \\
&\simeq \Erest{\C^\times_v}{\Eeeim q E^{f(u)g(v)}[1]}.
\end{align*}
Considering the commutative diagram with Cartesian squares
\[
\xymatrix{
\C_u \ar[d]_f & \C_u\times\C_v \ar[d]_{f' \defeq f\times\id} \ar[dr]^{q} \ar[l]_-p \\
\C_z & \C_z\times\C_v \ar[d]^{g''\defeq \id\times g} \ar[r]^q \ar[l]_-p & \C_v \ar[d]^g \\
& \C_z\times\C_w \ar[r]^q \ar[ul]_-p & \C_w,
}
\]
one has
\begin{align*}
\Eeeim q E^{f(u)g(v)}[1]
&\simeq \Eeeim q \Eopb{f'}E^{zg(v)}[1] \\
&\simeq \Eeeim q \Eeeim{f'}(\Eopb{f'}E^{zg(v)}\tens e\field_{\C_u\times\C_v}[1]) \\
&\simeq \Eeeim q (E^{zg(v)}\tens e\reim{f'}\field_{\C_u\times\C_v}[1]) \\
&\simeq \Eeeim q (\Eopb{g''}E^{zw}\tens \Eopb p e\reim f\field_{\C_u}[1]) \\
&\simeq \Eeeim q \Eopb{g''}(E^{zw}\tens\Eopb p e\reim f\field_{\C_u}[1]) \\
&\simeq \Eopb g \Eeeim q (E^{zw}\tens\Eopb p e\reim f\field_{\C_u}[1]) \\
&= \Eopb g ((e F)^\curlywedge).\qedhere
\end{align*}
\end{proof}

\begin{proof}[Proof of Lemma~\ref{lem:AiryQuiv}] The map $f $ is a 3:1 cover branched at $z\in\Sigma =
\{-2,2\}$. One has $f^{-1}(-2)=\{ 1, -2 \} $ with $u=-2$ a single point, and $u=1$ a double point.
Analogously, $f^{-1}(2)=\{-1,2 \} $ with $u=-1 $ of degree two and $u=2 $ of degree one. This is pictured
in Figure~\ref{fig:Airy_g}. (Write $z=z_1+\imm\, z_2$ and $u=u_1+\imm\, u_2$. To draw the picture we used the
fact that the pull-back of the real line $z_2=0$ by $f$ is given by
$0=\Im(u^3/3-u)=u_2(u_1^2-u_2^2/3-1)$.)

\begin{figure}
\includestandalone{Airy1}
\caption{The map $f:u \mapsto z=u^3-3 u $.}\label{fig:Airy_g}
\end{figure}

It follows that the restriction of $\reim{f} \field_{\C_{ u}}$ to $\V \smallsetminus \Sigma $ is a local system, and hence $F = \reim{f} \field_{\C_{ u}}[1]\in\Perv_\Sigma(\V)$.

Since $F$ is concentrated in degree $-1$,
$F\tens\field_\Sigma[-1]$ is perverse. Hence the distinguished triangle
\[
F\tens\field_{\V\setminus\Sigma} \to F \to F\tens\field_\Sigma \to[+1]
\]
induces the short exact sequence of perverse sheaves
\[
0 \to F\tens\field_\Sigma[-1]
\to F\tens\field_{\V\setminus\Sigma} \to F \to 0.
\]
In terms of quivers, this reads
\begin{equation}
\label{eq:AiryQuiverSES}
\begin{array}{c}
\xymatrix{
&\opb{i_2} F[-1] \ar@<.5ex>[d] \ar[r]^{b_2}
& \Psi(F) \ar@<.5ex>[d]^{1-T_2} \ar[r]^{u_2}
& \Phi_2(F) \ar@<.5ex>[d]^{v_2} \\
0\ar[r]
&0 \ar@<.5ex>[u] \ar@<-.5ex>[d] \ar[r]
& \Psi(F) \ar@<.5ex>[u]^{1} \ar@<-.5ex>[d]_{1} \ar[r]^{1}
& \Psi(F) \ar@<.5ex>[u]^{u_2} \ar@<-.5ex>[d]_{u_{-2}} \ar[r]
& 0. \\
&\opb{i_{-2}} F[-1] \ar@<-.5ex>[u] \ar[r]^{b_{-2}}
& \Psi(F) \ar@<-.5ex>[u]_{1-T_{-2}} \ar[r]^{u_{-2}}
& \Phi_{-2}(F) \ar@<-.5ex>[u]_{v_{-2}}
}
\end{array}
\end{equation}

\begin{figure}
\includestandalone[scale=.85]{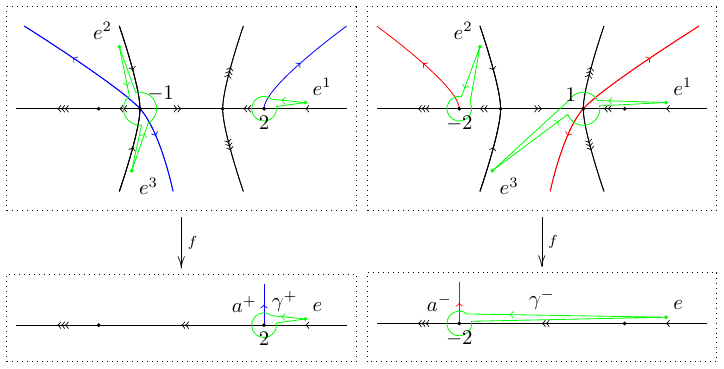}
\caption{The pullback by $f$ of the paths $\gamma^+$ and $\gamma^-$}.\label{fig:Airy_g_paths}
\end{figure}

Let us compute the maps $b_{\pm2}$ and $T_{\pm2}$,
referring to Figure~\ref{fig:Airy_g_paths}. Choose a base-point $e \in \V $ with $\Re z > 2 $,
and choose a numbering $e_1, e_2 $ and $e_3 $ of the preimages of $e$ by~$f$.
These choices induce isomorphisms
\begin{equation}
\label{eq:AiriId1}
\field^3 \simeq \DSum_{j=1}^3 (\field_{\C_{u}})_{e_j} \simeq
F_e[-1] \simeq
\Psi(F)\isofrom[\psi_{\pm2}]
\Psi_{\pm2}(F).
\end{equation}

Choose a path $\gamma^\pm $ starting at the base-point $e $, running parallel to the real axis to the point $\pm 2 $ respectively, circling around it in counter-clockwise orientation and returning to $e $ afterwards.
Denote by $\gamma_j^\pm $ the lift by $f$ of the path $\gamma^\pm $ starting at $e_j $.
(The course of $\gamma_j^\pm $ can be determined by taking into account the various intersections of $\gamma^\pm $ with the curves drawn in black, blue and red.)

As $T_{\pm2}$ is induced by the monodromy around $\gamma^\pm$, with the identification \eqref{eq:AiriId1} one has
\[
T_{2}=
\begin{pmatrix}
1 & 0& 0\\
0 & 0& 1\\
0 & 1& 0
\end{pmatrix}
, \quad
T_{-2}=
\begin{pmatrix}
0 & 0& 1\\
0 & 1& 0\\
1 & 0& 0
\end{pmatrix}
\]

Recall that $f^{-1}(2)=\{-1,2\}$, and fix an isomorphism
\begin{equation}
\label{eq:AiriId2}
\field^2 \simeq (\field_{\C_{u}})_{-1} \dsum (\field_{\C_{u}})_2 \simeq i_2^{-1}F[-1].
\end{equation}

Denote by $a^+$ and $a^-$ the blue and red half-lines in $\V$, and by $a^\pm_j$ the lift of $a^\pm$
in $\C_u$ first encountered by $\gamma^\pm_j $.

The morphism $\psi_2^{-1}\circ b_2$ is induced by the boundary value map from $a^+$ to its origin $z=2$.
Since the origin of $a^+_1$ is $u=2$ and the origin of $a^+_2$ and of $a^+_3$ is $u=-1$, with the identifications \eqref{eq:AiriId1} and \eqref{eq:AiriId2} one has
$b_2 = \left(\begin{smallmatrix}0&1\\1&0\\1&0\end{smallmatrix}\right)$. Similarly,
$b_{-2} = \left(\begin{smallmatrix}1&0\\0&1\\1&0\end{smallmatrix}\right)$.

Then, by \eqref{eq:AiryQuiverSES}, one has
\[
\vcenter{\vbox{\xymatrix{
\Phi_2(F) \ar@<.5ex>[d]^{v_2} \\
\Psi(F) \ar@<.5ex>[u]^{u_2} \ar@<-.5ex>[d]_{u_{-2}} \\
\Phi_{-2}(F) \ar@<-.5ex>[u]_{v_{-2}}
}}}
\simeq
\coker \left(
\vcenter{\vbox{\xymatrix{
\field^2 \ar@<.5ex>[d] \ar[r]^{\left(\begin{smallmatrix}0&1\\1&0\\1&0\end{smallmatrix} \right)}
& \field^3 \ar@<.5ex>[d]^{\left(\begin{smallmatrix}0&0&0\\0&1&-1\\0&-1&1\end{smallmatrix} \right)} \\
0 \ar@<.5ex>[u] \ar@<-.5ex>[d] \ar[r]
& \field^3 \ar@<.5ex>[u]^{\id} \ar@<-.5ex>[d]_{\id} \\
\field^2 \ar@<-.5ex>[u] \ar[r]_{\left(\begin{smallmatrix}1&0\\0&1\\1&0\end{smallmatrix} \right)}
& \field^3 \ar@<-.5ex>[u]_{\left(\begin{smallmatrix}1&0&-1\\0&0&0\\-1&0&1\end{smallmatrix} \right)}
}}}
\right)
\simeq
\vcenter{\vbox{\xymatrix{
\field \ar@<.5ex>[d]^{\left(\begin{smallmatrix}0\\1\\-1\end{smallmatrix} \right)} \\
\field^3 \ar@<.5ex>[u]^{\left(\begin{smallmatrix}0&1&-1\end{smallmatrix} \right)} \ar@<-.5ex>[d]_{\left(\begin{smallmatrix}1&0&-1\end{smallmatrix} \right)} &.\\
\field \ar@<-.5ex>[u]_{\left(\begin{smallmatrix}1\\0\\-1\end{smallmatrix} \right)}
}}}
\]
\end{proof}

\subsection{Elementary meromorphic connection}

Proceeding as in the previous section, we will describe here the Stokes multipliers at infinity of the Fourier-Laplace transform of $\she^{1/x} \defeq \D_{\C_x}e^{1/x}$. A more general situation has been considered in \cite{HS14}, using other methods.

Setting
\[
N \defeq \solE_\bb((\she^{1/x})^\wedge),
\]
we know (e.g. from \cite{Sab08}) that its formal structure is ramified at $\infty $. We therefore consider the
ramification
$$ r:\C_w\to\C_y, \quad w\mapsto w^2, $$
and restrict ourselves to analyze the pull-back $\Eopb r N$.

\begin{notation}
In $\C_w$, consider the closed half spaces in $\C_w^\times$
\[
H_\pm = \{ w\semicolon \pm\Re w \geq 0\},
\]
and the open half-lines
\[
\ell_\pm = \pm \imm\, \R_{>0},
\]
\end{notation}

Note that $\ell_\pm$ are the anti-Stokes lines for the exponentials $E^{\pm2w}$, and that there are identifications
$$
e_{\pm} \colon \Endo\left(\Erest{\ell_\pm}{\bigl(E^{-2 w} \dsum E^{2 w} \bigr)} \right) \isoto \Endo^{\pm 1}(\C^2),
$$
where we recall that $\Endo^{\pm 1}(\C^2)\subset\Endo(\C^2)$ denotes the vector
subspace of lower/upper triangular matrices.

\begin{prop}\label{pro:Elementary}
With the above notations.
\begin{itemize}
\item[(i)] The exponential components of $\Eopb r N$ at $\infty$ are $E^{\pm 2 w}$. More precisely, there are isomorphisms
$$
s_{\pm} \colon \Erest{H_\pm}{\Eopb r N}
\isoto
\Erest{H_\pm}{\bigl(E^{-2 w} \dsum E^{2 w} \bigr)}.
$$
\item[(ii)]
The Stokes multipliers of $\Eopb r N$ at $\infty$, defined by
\[
S_{\pm} \defeq e_{\pm}\bigl((\Erest{\ell_\pm}{s_{\pm}}) \circ (\Erest{\ell_\pm}{s^{-1}_\mp}) \bigr),
\]
are given by
\[
S_{+} =
\begin{pmatrix}
-1 & 0 \\
-2 & -1
\end{pmatrix}, \quad
S_{-} =
\begin{pmatrix}
1 & 2 \\
0 & 1
\end{pmatrix}.
\]
\end{itemize}
\end{prop}

As before, note that one has
\begin{align}
\label{eq:ElementaryPrelim}
\Eopb r N
&\simeq \Eopb r \Eeeim q E^{xy+\tfrac 1x}[1] \\ \notag
&\simeq \Eeeim q E^{xw^2+\tfrac 1x}[1].
\end{align}
Now consider the change of variables $\C_u\times\C^\times_w \isoto \C_x\times\C^\times_w$ given by
\begin{equation}
\label{eq:ElementaryCoord}
\begin{cases}
x = \tfrac u w,\\
w = w,
\end{cases}
\end{equation}
and consider the meromorphic function
\begin{align*}
f&\colon\C_u \to \C_z=\V, \quad u \mapsto u + \tfrac 1 u,
\end{align*}
so that $x w^2+\tfrac 1 x = f(u)w = zw$. Set
\[
F\defeq \reim{f} \field_{\C_{ u}}[1] \in \BDC(\field_\V).
\]

As in the case of the Airy equation the proof of Proposition
\ref{pro:Elementary} is obtained from the following

\begin{lemma}\label{lem:ElementaryReductQuiv}
\begin{itemize}
\item [(i)]
One has
\[
\Erest{\C^\times_w}{\Eopb r N} \simeq
\Erest{\C^\times_w}{(e F)^\curlywedge}.
\]

\item [(ii)]
One has $F\in\Perv_\Sigma(\V)$ for $\Sigma = \{-2,2\}$. Moreover, the quiver of $F$ is given by
\[
\vcenter{\vbox{\xymatrix{
\Phi_2(F) \ar@<.5ex>[d]^{v_2} \\
\Psi(F) \ar@<.5ex>[u]^{u_2} \ar@<-.5ex>[d]_{u_{-2}} \\
\Phi_{-2}(F) \ar@<-.5ex>[u]_{v_{-2}}
}}}
\simeq
\vcenter{\vbox{\xymatrix{
\field \ar@<.5ex>[d]^{\left(\begin{smallmatrix}1\\-1\end{smallmatrix} \right)} \\
\field^2 \ar@<.5ex>[u]^{\left(\begin{smallmatrix}1&-1\end{smallmatrix} \right)} \ar@<-.5ex>[d]_{\left(\begin{smallmatrix}1&-1\end{smallmatrix} \right)} &.\\
\field \ar@<-.5ex>[u]_{\left(\begin{smallmatrix}1\\-1\end{smallmatrix} \right)}.
}}}
\]
\end{itemize}
\end{lemma}

\begin{proof}[Proof of Lemma~\ref{lem:ElementaryReductQuiv}]

\emph{(i)} The proof follows the steps of the proof of Lemma
\ref{lem:AiryReduct}.

\emph{(ii)} The proof follows the steps of the proof of Lemma
\ref{lem:AiryQuiv}. We are giving the details of the differences in the present case.

The map $f $ is a branched covering of degree two with branch locus
$\{ \pm 2 \} $ and $f^{-1}(\pm 2)= \pm 1 $. The restriction of $F[-1] $ to
$\C_w \smallsetminus \{\pm 2\} $ is a local system of rank two. With
$\Sigma:= \{\pm 2\} $, we denote by $L $ the local system
$L:=F|_{\C_w \smallsetminus \Sigma}[-1] $.

\begin{figure}
\includestandalone[scale=.85]{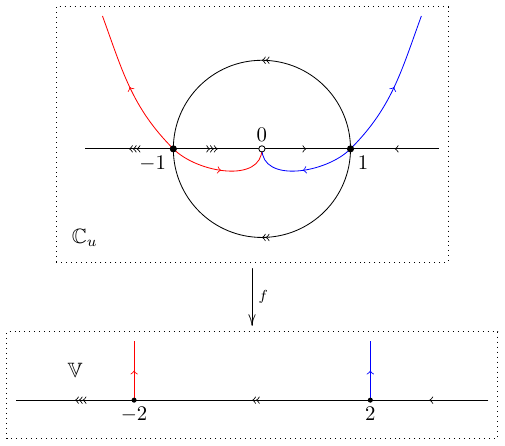}
\caption{The map $f\colon u\mapsto z=u+\frac1u$. Note that $f(0)=\infty$.}\label{fig:expmonodromy}
\end{figure}

We choose the same paths in $\C_w $ as in the case of the
Airy-function. See Figure \ref{fig:expmonodromy}.
(Write $z=z_1+\imm\,z_2$ and $u=u_1+\imm\,u_2$. To draw the picture we used the fact that the pull-back of
the real line $z_2=0$ by $f$ is given by $\Im(u+1/u)=0$, which is equivalent to $u_2(u_1^2+u_2^2-1)=0$.)

Let $e=x+\imm\,\varepsilon \in \C_w $ with
$x \gg 0 $ and small $\varepsilon>0 $ with pre-images $e_1 $ and $e_2
$, and choose paths $\gamma^\pm $ as before.

The choice of the numbering identifies $L_e= \DSum_{j=1}^2
(\field_{\C_{u}})_{e_j} \simeq \field^2 $. For the generic
stalk, with $\alpha=\imm\in\C_z$, $\beta=1\in\C_w$, we get
$$
\Psi_{\pm 2}(F) = \ \Psi^{(\alpha,\beta)}_{\pm 2}(F) =  L_e \simeq \field^2
$$
and the monodromies $T_{\pm 2}$ are determined by following the lifts $\gamma_j^\pm $ of the paths $\gamma^\pm $. We obtain,
\begin{equation}\label{eq:expmu-2}
T_{-2}= T_{2}=
\begin{pmatrix}
0 & 1\\
1 & 0
\end{pmatrix}
.
\end{equation}

\begin{figure}
\includestandalone[scale=.85]{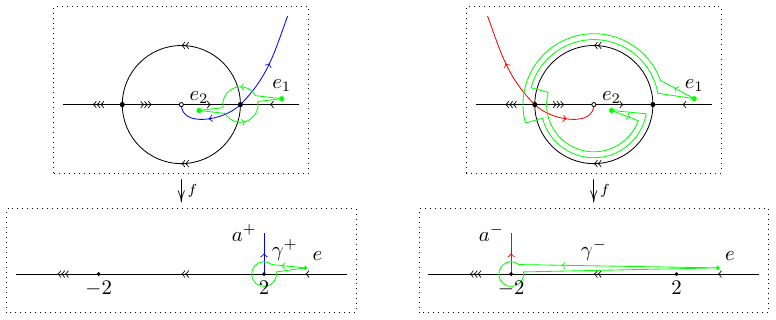}
\caption{The lifts of $\gamma^+ $.}\label{fig:expmonodromy2}
\end{figure}

To conclude, we notice that we have

\[
\vcenter{\vbox{\xymatrix{
\Phi_2(F) \ar@<.5ex>[d]^{v_2} \\
\Psi(F) \ar@<.5ex>[u]^{u_2} \ar@<-.5ex>[d]_{u_{-2}} \\
\Phi_{-2}(F) \ar@<-.5ex>[u]_{v_{-2}}
}}}
\simeq
\coker \left(
\vcenter{\vbox{\xymatrix{
\field \ar@<.5ex>[d] \ar[r]^{\left(\begin{smallmatrix}1\\1\end{smallmatrix} \right)}
& \field^2 \ar@<.5ex>[d]^{\left(\begin{smallmatrix}1&-1\\-1&1\end{smallmatrix} \right)} \\
0 \ar@<.5ex>[u] \ar@<-.5ex>[d] \ar[r]
& \field^2 \ar@<.5ex>[u]^{\id} \ar@<-.5ex>[d]_{\id} \\
\field \ar@<-.5ex>[u] \ar[r]_{\left(\begin{smallmatrix}1\\1\end{smallmatrix} \right)}
& \field^2 \ar@<-.5ex>[u]_{\left(\begin{smallmatrix}1&-1\\-1&1\end{smallmatrix} \right)}
}}}
\right)
\simeq
\vcenter{\vbox{\xymatrix{
\field \ar@<.5ex>[d]^{\left(\begin{smallmatrix}1\\-1\end{smallmatrix} \right)} \\
\field^2 \ar@<.5ex>[u]^{\left(\begin{smallmatrix}1&-1\end{smallmatrix} \right)} \ar@<-.5ex>[d]_{\left(\begin{smallmatrix}1&-1\end{smallmatrix} \right)} &.\\
\field \ar@<-.5ex>[u]_{\left(\begin{smallmatrix}1\\-1\end{smallmatrix} \right)}
}}}
\]
\end{proof}

\appendix

\section{Complements on Borel-Moore homology}\label{app:BM}

Let $X$ be a subanalytic space, and consider the space $\BM_p^X(G)$ of subanalytic Borel-Moore $p$-chains relative to an $\R$-constructible sheaf $G$. The definition of $\BM_p^X(G)$ is parallel to the definition in \cite{Kas84} of the sheaf $TH(G)$ of tempered homomorphisms. Both constructions are best understood in the framework of subanalytic sheaves of \cite{KS01}. There, one defines the subanalytic sheaf $\Db^t_X$ of tempered distributions, which satisfies $TH(G)\simeq\alpha\hom(\iota G,\Db^t_X)$ (with $\iota$ is the embedding of sheaves into subanalytic sheaves, and $\alpha$ its left adjoint). Here, we introduce the subanalytic sheaf of Borel-Moore $p$-chains $\bm^X_p$, which satisfies $\BM_p^X(G)\simeq\Hom(\iota G,\bm^X_p)$.
We then indicate how to adapt the arguments in \cite[\S9.2]{KS90} in order to prove the results we used in \S\ref{sec:Beilmaxext}.

\subsection{Borel-Moore chains}\label{sse:BMchains}
Let $X$ be a subanalytic space, that is, an $\R$-ringed space locally modeled on closed subanalytic subsets of real analytic manifolds.
Recall from \cite[\S9.2]{KS90} that,
for $p\in\Z_{\geq 0}$, the sheaf $\cs^X_p$ of subanalytic $p$-chains is the sheaf associated with the presheaf $CS^X_p$ defined as follows.
For $V\subset X$ an open subset, $CS^X_p(V)$ is the $\field$-vector space generated by the symbols $[S]$, where $S$ ranges through the family of $p$-dimensional oriented subanalytic submanifolds of $V$, with the relations
\begin{itemize}
\item[(a)] $[S_1\union S_2] = [S_1] + [S_2]$ if $S_1\cap S_2=\emptyset$,
\item[(b)] $[S]=[S']$ if $S'$ is an open dense subset of $S$ with the induced orientation,
\item[(c)] $[S^a] = -[S]$ if $S^a$ denoted the manifold $S$ endowed with the opposite orientation.
\end{itemize}

There is a boundary map $\partial\colon\cs^X_p\to\cs^X_{p-1}$, whose construction is detailed in \cite[\S9.2]{KS90},
inducing a complex of sheaves $\cs^X_\bullet$.
The map $\partial$ coincides with the ``na\"\i f'' boundary map in the case of interest to us, which is as follows.
Let $S\subset X$ be an oriented $p$-dimensional subanalytic submanifold
such that the embedding $S\subset \overline S$ is locally isomorphic in the category of subanalytic spaces to the embedding $\{x_1>0\} \subset \{x_1\geq 0\}$ in $\R^p\owns(x_1,\dots,x_p)$. Then $\partial[S] = [\partial S]$, where $\partial S$ is endowed with the induced orientation.

For $G\in\Mod_\Rc(\field_X)$, we define
the space of subanalytic Borel-Moore $p$-chains relative to $G$ as the subspace
\[
\BM^X_p(G) \subset \Hom(G,\cs^X_p)
\]
of morphisms $\varphi\in\Hom(G,\cs^X_p)$ such that for any relatively compact open subanalytic subset $U$ of $X$, and $s\in G(U)$, there exists $\sigma\in\cs^X_p(X)$ with $\sigma|_U=\varphi(s)$. (This last condition is equivalent to asking that $\supp\varphi(s)$, which is a closed subanalytic subset of $U$, is subanalytic in $X$.)

One has a complex $\BM^X_\bullet(G)$ with boundary map induced by $\partial$.

\subsection{Subanalytic sheaves}
For the theory of subanalytic sheaves we refer to \cite{KS01} (see also \cite{Pre08}).
Let us also mention the paper \cite{Pre16}, where subanalytic sheaves are used for analogous purposes. Roughly, a subanalytic sheaf on $X$ is a presheaf defined on the relatively compact subanalytic open subsets of $X$ satisfying the patching conditions on finite covers.
One denotes by $\Mod(\field_X^\sub)$ the category of subanalytic sheaves and by $\BDC(\field_X^\sub)$ its bounded derived category.
There is a natural fully faithful functor
\[
\iota\colon\Mod(\field_X) \to \Mod(\field_X^\sub),
\]
which has an exact left adjoint functor $\alpha$.
The restriction of $\iota$ to $\R$-constructible sheaves is exact, and induces a fully faithful functor
\[
\iota\colon\BDC_\Rc(\field_X) \to \BDC(\field_X^\sub).
\]

One says that $J\in\Mod(\field_X^\sub)$ is quasi-injective if the restriction map $J(U)\to J(V)$ is surjective for every $V\subset U$ relatively compact subanalytic open subsets. Quasi-injective objects are injective for the functors $\Hom(\iota G,\bullet)$ and $\hom(\iota G,\bullet)$, where $G\in\Mod_\Rc(\field_X)$. In particular, the functors $\Hom(\bullet,J)$ and $\hom(\bullet,J)$ are exact on $\Mod_\Rc(\field_X)$.

\subsection{Borel-Moore homology}\label{sse:BMH}
For $p\in\Z_{\geq 0}$ and $V\subset X$ an open subanalytic relatively compact subset, let $\bm^X_p(V)$ be the $\field$-vector space generated by the symbols $[S]$, where $S$ ranges through the family of $p$-dimensional oriented analytic submanifolds of $V$ which are subanalytic in $X$, with the relations (a)-(b)-(c) in section \S\ref{sse:BMchains}.
Then $V\mapsto \bm^X_p(V)$ is a quasi-injective subanalytic sheaf.

Note that, for $G\in\Mod_\Rc(\field_X)$, one has
\begin{equation}
\label{eq:BMbm}
\BM^X_p(G) \simeq \Hom(\iota G,\bm^X_p).
\end{equation}
Note also that $\alpha(\bm^X_p) \simeq \mathcal{CS}^X_p$.

There is a boundary map $\partial\colon\bm^X_p\to\bm^X_{p-1}$, inducing a complex of subanalytic sheaves $\bm^X_\bullet$, whose
construction goes along the same lines as for the boundary map on $\cs^X_\bullet$ explained in \cite[\S9.2]{KS90}.

\begin{prop}\label{pro:CSomega}
There is a natural isomorphism in $\BDC(\field_X^\sub)$
\[
\iota\omega_X\isoto \bm^X_\bullet,
\]
where $\omega_X$ denotes the dualizing complex.
\end{prop}

\begin{proof}
Let us assume for simplicity that $X$ is smooth.
One can follow the lines of the proof of the analogue result in \cite[\S9.2]{KS90}. Then the part of the proof of Theorem 9.2.10 in loc.\ cit.\
where exactness is checked on stalks has to be adapted.
Let $U\subset X$ be a relatively compact subanalytic open subset, and $\sigma\in\sect(U;\bm^X_p)$ with $p<n$ a subanalytic cycle (i.e.\ a chain satisfying $\partial\sigma=0$).
We have to show that $\sigma$ is locally a boundary, i.e.\ that there exist a finite cover $U=\Union\nolimits_i U_i$ and subanalytic chains $\tau_i\in\sect(U_i;\bm^X_{p+1})$ such that $\partial\tau_i=\sigma|_{U_i}$. By the arguments in Theorem 9.2.10, we know that for any $x\in\overline U$ there are an open neighborhood $V_x$ of $x$ and a section $\widetilde\tau_x\in\sect(V_x;\bm^X_{p+1})$ such that $\partial\widetilde\tau_x=\sigma|_{V_x}$. We can assume that $V_x$ is relatively compact and subanalytic. Since $\overline U$ is compact, we
have $\overline U = \Union\nolimits_{i\in I} V_i$
for some $I\subset \overline U$ finite. We then set $U_i=V_i\cap U$ and $\tau_i=\widetilde\tau_i|_{U_i}$.
\end{proof}

Consider the complex $\BM^X_\bullet(G)$ with boundary map induced by $\partial$.
We can now prove
\begin{equation}
\label{eq:BMG}
\BM^X_\bullet(G) \simeq \RHom(G,\omega_X).
\end{equation}

\begin{proof}[Proof of \eqref{eq:BMG}]
One has
\begin{align}
\label{eq:GoBM}
\BM^X_\bullet(G)
&\simeq \Hom(\iota G,\bm^X_\bullet) \\ \notag
&\underset{(*)}\simeq \RHom( \iota G, \iota \omega_X ) \\ \notag
&\underset{(**)}\simeq \RHom( G, \omega_X ),
\end{align}
where $(*)$ follows from the fact that $\iota$ is fully faithful on $\BDC_\Rc(\field_X)$, and $(**)$ follows from Proposition~\ref{pro:CSomega}.
\end{proof}

\begin{remark}
As pointed out in \cite[Exercise 9.1]{KS90}, one has
\[
\rhom(G,\omega_X) \simeq \hom( G, \mathcal{CS}^X_\bullet ).
\]
However, $\RHom(G,\omega_X) \not\simeq \Hom( G, \mathcal{CS}^X_\bullet )$ in general,
since $\mathcal{CS}^X_\bullet$ is a $c$-soft resolution of $\omega_X$, not a flabby resolution. In particular, the previous isomorphism does not hold in one of the cases of interest for us, that is when $G$ is the extension by zero of a local system of finite rank on a closed subanalytic subset of $X$.
\end{remark}

In particular, if $L$ is a local system of finite rank, and $Z\subset X$ is a locally closed subanalytic subset,
one has
\begin{align}
\label{eq:BMZL}
\BM^X_\bullet(\field_Z \tens L)
&\simeq \RHom(\field_Z \tens L,\omega_X) \\ \notag
&\simeq \RHom(L,\rsect_Z\omega_X) \\ \notag
&\simeq \rsect(X; \rsect_Z\omega_X \tens L^*).
\end{align}
It is thus natural to set the

\begin{definition}\label{def:HpBM}
Let $L$ be a local system of finite rank, and $Z\subset X$ a locally closed subanalytic subset.
For $j\in\Z$, the Borel-Moore homology of $Z$, relative to $X$, with coefficients in $L^*$ is given by
\[
H^{BM}_j(Z; L^*) \defeq H_j \BM^X_\bullet(\field_Z \tens L).
\]
\end{definition}

\subsection{Long exact homology sequence}

Let $U \subset X $ be an open subanalytic subset, and set $Y:=X \smallsetminus U $. By \eqref{eq:BMbm}, one has
\begin{align}
\label{eq:HBMZL}
H^{BM}_p(Y; L^*)
&\simeq H_p \Hom(\iota (\field_Y\tens L),\bm^X_\bullet) \\ \notag
&\simeq H_p \Hom(\iota L, \sect_Y\bm^X_\bullet).
\end{align}

\begin{lemma}\label{lem:BMrel}
With the above notations, one has
\[
H_j^{BM}(U ; L^*) \simeq H_j \frac{\BM^X_\bullet(L)}{\BM^X_\bullet(\field_Y \tens L)}.
\]
\end{lemma}
\begin{proof}
By \eqref{eq:HBMZL}, we have to prove
\[
H_p^{BM}(U; L^*) \simeq H_p \frac{\Hom(\iota L, \bm^X_\bullet)}{\Hom(\iota L, \sect_Y\bm^X_\bullet)}.
\]
Since the subanalytic sheaves $\bm^X_p$ are quasi-injective,
the functor
\[
\Hom(\iota(L\tens\bullet),\bm^X_\bullet) \colon \Mod_\Rc(\field_X) \to \mathsf{C}^b(k),
\]
with values in the category of bounded complexes of $\field$-vector spaces,
is exact. Applying it to the exact sequence
\[
0 \to \field_U\to \field_X\to \field_Y\to 0,
\]
we get the exact sequence of complexes
\[
0 \to
\Hom(\iota L, \sect_Y\bm^X_\bullet) \to
\Hom(\iota L, \bm^X_\bullet)
\to \Hom(\iota L, \sect_U\bm^X_\bullet) \to
0.
\]
Then
\[
\Hom(\iota L, \sect_U\bm^X_\bullet) \simeq
\frac{\Hom(\iota L, \bm^X_\bullet)}{\Hom(\iota L, \sect_Y\bm^X_\bullet)},
\]
and the statement follows by taking homology groups.
\end{proof}

\begin{remark}
If $Z\subset X$ is a locally closed subanalytic subset, there is an isomorphism
\begin{equation}
\label{eq:HBMZXZ}
H^{BM}_p(Z; L^*)\simeq H^{BM}_p(Z; L^*|_Z).
\end{equation}
In fact, by \eqref{eq:BMZL} there is an isomorphism in $\BDC(\field)$
\begin{align*}
\BM^X_\bullet(\field_Z \tens L)
&\simeq \RHom(L,\rsect_Z\omega_X) \\
&\simeq \RHom(L|_Z,\omega_Z) \\
&\simeq \BM^Z_\bullet(L|_Z).
\end{align*}
If $Z=Y$ is closed, one also has $\BM^X_p(\field_Y \tens L)\simeq \BM^Y_p(L|_Y)$.
However, since $\BM^X_p(\field_Z \tens L)\not\simeq \BM^Z_p(L|_Z)$ in general, the description of the two homology groups in \eqref{eq:HBMZXZ} is different.
For example, if $Z=U$ is open and $S\subset U$ is a $p$-dimensional oriented subanalytic submanifolds of $U$, then
$[S]\in \BM^U_p(L|_U)$ even if $S$ is not subanalytic in $X$. This makes the boundary value map $\delta$ less explicit when written as
\[
H_{j+1}^{BM}(U; L^*|_U) \\
\to[\delta] H_j^{BM}(Y; L^*|_Y).
\]
\end{remark}

As we detail below, Lemma~\ref{lem:lemmaFT} follows from the next lemma.
\begin{lemma}\label{lem:BMderdual}
For $G\in\Mod_\Rc(\field_X)$, there is a functorial isomorphism
\[
\dual\, \rsect_c(X; G) \simeq \Hom( \iota G, \bm^X_\bullet ),
\]
where $\dual(\bullet)=\rhom(\bullet,\field)$ is the dual in $\BDC(\field)$.
\end{lemma}

\begin{proof}
Denoting $a_X\colon X\to\{pt\}$ the map to a singleton, one has
\begin{align*}
\dual\, \rsect_c(X; G)
&= \RHom( \rsect_c(X; G), \field ) \\
&\simeq \RHom( \reim{{a_X}}G, \field ) \\
&\simeq \RHom( G, \omega_X ) \\
&\underset{(*)}\simeq \Hom( \iota G, \bm^X_\bullet ) ,
\end{align*}
where $(*)$ follows from \eqref{eq:GoBM}.
\end{proof}

\begin{proof}[Proof of Lemma~\ref{lem:lemmaFT}]
(i) Lemma~\ref{lem:BMderdual} gives an isomorphism
\[
\dual\, \rsect_c(X; \field_Z\tens F) \simeq \Hom( \iota (\field_Z\tens F), \bm^X_\bullet ),
\]
and the statement follows by taking homologies.

\smallskip\noindent(ii) follows by applying the functorial isomorphism from Lemma~\ref{lem:BMderdual}
\[
\dual\, \rsect_c(X; \bullet\tens F) \simeq \Hom( \iota (\bullet\tens F), \bm^X_\bullet )
\]
to the exact sequence \eqref{eq:UXY}, and taking homologies.
\end{proof}

\section{Complements on the smash functor}\label{app:smash}

In this section, one may take for $\V$ any real vector space of finite dimension.

\subsection{Stalk at the origin}

Recall the maps
\begin{equation*}
\xymatrix{
\V & \V\times\R_{<0} \ar@{ >->}[r]^-{j} \ar[l]_-\gamma & \V\times\R & \V
\ar@{ >->}[l]_-i,
}
\end{equation*}
where $j$ is the open embedding, $i(x) = (x,0)$, and $\gamma(x,s) = x/|s|$.

\begin{lemma}\label{lem:tildenuvariant}
For $F\in\BDC(\field_\V)$ one has
\[
\smsh(F) \simeq
\epb{i} \reim{j} \epb{\gamma} F.
\]
\end{lemma}

\begin{proof}
We give a proof analogue to that of \cite[Lemma 4.2.1]{KS90}.
The distinguished triangle
\[
\rsect_{\{s=0\}}(\reim {j} \opb \gamma F) \to \rsect_{\{s\leq 0\}}(\reim {j}
\opb \gamma F)
\to \rsect_{\{s<0\}}(\reim {j} \opb \gamma F) \to[+1]
\]
is isomorphic to the distinguished triangle
\[
\roim i \epb{i} \reim {j} \opb \gamma F \to \reim {j} \opb \gamma F
\to \roim {j} \opb \gamma F \to[+1].
\]
Applying $\opb i$ we get
\begin{align*}
\opb i \roim {j} \opb \gamma F
&\simeq \epb {i} \reim {j} \opb \gamma F[1] \\
&\simeq \epb {i} \reim {j} \epb {\gamma} F.\qedhere
\end{align*}
\end{proof}

\begin{lemma}
$F\in\BDC(\field_\V)$. Then, denoting $i_0\colon \{0\} \to \V$ the embedding, one has
\begin{align*}
\opb{i_0}\smsh(F) &\simeq \rsect(\V; F), \\
\epb{i_0}\smsh(F) &\simeq \rsect_\rc(\V; F).
\end{align*}
\end{lemma}

\begin{proof}
Let $i_{(0,0)}\colon \{(0,0)\} \to \V\times\R$ be the embedding.
One has
\begin{align*}
\opb{i_0}\smsh(F)
&= \opb{i_0}\opb i \roim {j} \opb\gamma F \\
&\simeq \opb{i_{(0,0)}}\roim {j} \opb\gamma F \\
&\underset{(1)}\simeq \rsect(\V\times\R; \roim {j} \opb\gamma F) \\
&\simeq \rsect(\V\times\R_{<0}; \opb\gamma F) \\
&\simeq \rsect(\V; \roim\gamma\opb\gamma F) \\
&\underset{(2)}\simeq \rsect(\V; F),
\end{align*}
where $(1)$ follows from \cite[Proposition 3.7.5]{KS90} since $\roim {j}
\opb{\gamma} F$ is conic, and~$(2)$ follows from the fact that $\gamma$ has contractible fibers.

This proves the first isomorphism in the statement.
The proof of the second isomorphism is similar, using
Lemma~\ref{lem:tildenuvariant}.
\end{proof}

\subsection{Smash and specialization functors}

Let $\U = T_\infty\PP$, where $\PP = \V \dunion \{\infty\}$ is the one-point compactification.
In the next lemma, we relate the smash functor to Sato's specialization
\[
\nu_{\{0\}} \colon \BDC(\field_\U) \to \BDC(\field_\U), \quad
G\mapsto \opb {i'} \roim j' \opb\muu G,
\]
where we considered the maps
\begin{equation}
\label{eq:Vnormalcone}
\xymatrix{
\U & \U\times\R_{>0} \ar@{ >->}[r]^-{j'} \ar[l]_-\muu & \U\times\R & \U
\ar@{ >->}[l]_-{i'},
}
\end{equation}
with $j'$ the open embedding, $i'(x) = (x,0)$, and $\muu(x,t) = tx$ the action
of $\R^+$ on $\U$.
More precisely, we will consider the induced functor
\[
\nu_{\{0\}} \colon \BDC(\field_{\U\setminus\{0\}}) \to
\BDC(\field_{\U\setminus\{0\}}),
\]
obtained by replacing $\U$ with $\U\setminus\{0\}$ in \eqref{eq:Vnormalcone}
(cf.\ \cite[Exercise IV.2]{KS90}).

The stereographic projections induce a bijection
\[
u \colon \V\setminus\{0\} \to \U\setminus\{0\}, \quad
x\mapsto x/|x|^2 = 1/\overline x,
\]
as indicated in the picture below, where $\PP$ is represented as a sphere.
\[
\begin{tikzpicture}[scale=2]
\coordinate [label=above:$\infty$] (N) at (0,0);
\coordinate (Nl) at (-1,0);
\coordinate [label=right:$\U$] (Nr) at (2,0);
\coordinate [label=below:$0$] (S) at (0,-1);
\coordinate (Sl) at (-1,-1);
\coordinate [label=right:$\V$] (Sr) at (2,-1);
\coordinate (C) at (0,-1/2);
\coordinate (P) at (1.5,0);
\node [label=left:$\PP$] (ce) at (C) [draw,circle through=(N)] {};
\draw (Sl) -- (Sr) ;
\draw (Nl) -- (Nr);
\draw (S) -- (P);
\coordinate (A) at (intersection of ce and S--P);
\coordinate (Q) at (intersection of N--A and S--Sr);
\draw (N) -- (Q);
\fill (N) circle (.6pt);
\fill (S) circle (.6pt);
\fill (A) circle (.6pt);
\fill (P) circle (.6pt);
\fill (Q) circle (.6pt);
\draw (S) -- (N);
\draw (N) -- (P) node[midway,above] {$|x|^{-1}$} node[above] {$u(x)$};
\draw (S) -- (Q) node[midway,below] {$|x|$} node[below] {$x$};
\end{tikzpicture}
\]

\begin{lemma}\label{lem:nu0smash}
Let $F\in\BDC(\field_\V)$.
Then one has
\[
\smsh(F)|_{\V\setminus\{0\}}
\simeq \opb u \nu_{\{0\}}(\reim u(F|_{\V\setminus\{0\}})).
\]
\end{lemma}

\begin{proof}
Consider the commutative diagram with cartesian squares
\begin{equation*}
\xymatrix@C=.7cm{
\V\setminus\{0\} \ar[d]^\wr_{u} & \V\setminus\{0\}\times\R_{<0}
\ar[d]^\wr_{{u}\times r} \ar@{ >->}[r]^-{j} \ar[l]_-\gamma &
\V\setminus\{0\}\times\R \ar[d]^\wr_{{u}\times r} & \ar@{ >->}[l]_-i
\V\setminus\{0\} \ar[d]^\wr_{u} \\
\W\setminus\{0\} & \W\setminus\{0\}\times\R_{>0} \ar@{ >->}[r]^-{j'}
\ar[l]_-{\muu} & \W\setminus\{0\}\times\R & \ar@{ >->}[l]_-{i'}
\W\setminus\{0\},
}
\end{equation*}
where $r(s)=-s$.
One has
\begin{align*}
\opb u \nu_{\{0\}}(\reim u(F|_{\V\setminus\{0\}}))
&= \opb u \opb {i'} \roim {j'} \opb{\muu} \reim{u} (F|_{\V\setminus\{0\}}) \\
&\simeq \opb {i'} \opb {(u\times r)} \roim {j'} \reim{(u\times r)} \opb{\gamma}
(F|_{\V\setminus\{0\}}) \\
&\simeq \opb {i'} \opb {(u\times r)} \roim {j'} \roim{(u\times r)} \opb{\gamma}
(F|_{\V\setminus\{0\}}) \\
&\simeq \opb {i'} \opb {(u\times r)} \roim{(u\times r)} \roim {j} \opb{\gamma}
(F|_{\V\setminus\{0\}}) \\
&\simeq \opb i \roim {j}\opb\gamma (F|_{\V\setminus\{0\}}) \\
&= \smsh(F)|_{\V\setminus\{0\}}. \qedhere
\end{align*}
\end{proof}

\providecommand{\eprint}[1]{\href{http://arxiv.org/abs/#1}{\texttt{arXiv\string:\allowbreak#1}}}\providecommand{\doi}[1]{\href{http://dx.doi.org/#1}{\texttt{doi\string:\allowbreak#1}}}
\def\bysame{\leavevmode ---------\thinspace}
\gdef\og{``}\gdef\fg{''}
\def\cdrandname{\&}
\providecommand\cdrnumero{no.~}
\providecommand{\cdredsname}{eds.}
\providecommand{\cdredname}{ed.}
\providecommand{\cdrchapname}{chap.}
\providecommand{\cdrmastersthesisname}{Memoir}
\providecommand{\cdrphdthesisname}{PhD Thesis}

\end{document}